\documentclass[10pt]{article}

\usepackage{amsmath,amsthm,amsfonts,amssymb,bm,bbm,enumerate, graphicx, mathtools,color}
\usepackage{tcolorbox}
\usepackage[english]{babel}
\usepackage[utf8]{inputenc}
\usepackage{fancyhdr}
\usepackage{caption}
\usepackage{subcaption}

\newcommand{\pr}[1]{\mathbb{P}\!\left(#1\right)}
\newcommand{\E}[1]{\mathbb{E}\!\left[#1\right]}

\newcommand{\prstart}[2]{\mathbb{P}_{#2}\!\left(#1\right)}
\newcommand{\prnstart}[3]{\mathbb{P}^{#3}_{#2}\!\left(#1\right)}
\newcommand{\prtilnstart}[3]{\widetilde{\mathbb{P}}^{#3}_{#2}\!\left(#1\right)}
\newcommand{\prcond}[3]{\mathbb{P}_{#3}\!\left(#1\;\middle\vert\;#2\right)}

\newcommand{\prtilstart}[2]{\mathbb{\widetilde{P}}_{#2}\!\left(#1\right)}

\newcommand{\prb}[1]{\mathbf{P}\!\left(#1\right)}
\newcommand{\Eb}[1]{\mathbf{E}\!\left[#1\right]}
\newcommand{\estartb}[2]{\mathbf{E}_{#2}\!\left[#1\right]}

\newcommand{\prcondb}[3]{\mathbf{P}_{#3}\!\left(#1\;\middle\vert\;#2\right)}

\newcommand{\prbl}[1]{\mathbf{P}^{\ell}\!\left(#1\right)}

\newtheorem{theorem}{Theorem}[section]
\newtheorem{lemma}[theorem]{Lemma}
\newtheorem{proposition}[theorem]{Proposition}
\newtheorem{corollary}[theorem]{Corollary}
\newtheorem{definition}[theorem]{Definition}

\newtheorem{remark}[theorem]{Remark}

\parskip \medskipamount
\newcommand{\dGH}[1]{d_{GH}\!\left(#1\right)}
\newcommand{\dGHPp}[1]{d_{GHP}\!\left(#1\right)}

\newcommand{\N}{\mathbb{N}}
\newcommand{\Z}{\mathbb{Z}}
\newcommand{\R}{\mathbb{R}}

\newcommand{\bP}{\mathbb{P}}
\newcommand{\bPb}{\mathbf{P}}
\newcommand{\F}{\mathcal{F}}
\newcommand{\EE}{\mathcal{E}}

\newcommand{\X}{X^{\text{exc}}}
\newcommand{\Xb}{X^{\text{br}}}
\newcommand{\HH}{H^{\text{exc}}}
\renewcommand{\phi}{\varphi}
\renewcommand{\epsilon}{\varepsilon}

\newcommand{\T}{\mathcal{T}}
\newcommand{\Levy}{L\'{e}vy }
\newcommand{\Ito}{It\^o }
\renewcommand{\L}{\mathcal{L}}
\newcommand{\cadlag}{c\`{a}dl\`{a}g }
\newcommand{\Loopp}{\textsf{Loop'}}
\newcommand{\Loop}{\textsf{Loop}}
\newcommand{\Tree}{\textsf{Tree}}

\renewcommand{\d}{\tilde{d}}
\newcommand{\exc}{\text{exc}}
\newcommand{\br}{\text{br}}
\newcommand{\osc}{\textsf{Osc}}

\newcommand{\La}{\L_{\alpha}}
\newcommand{\Ta}{\T_{\alpha}}

\newcommand{\Hm}{H_{\text{max}}}

\newcommand{\Lal}{\L^{\l}_{\alpha}}
\newcommand{\Lai}{\L^{\infty}_{\alpha}}

\newcommand{\Xbl}{X^{\text{br},\ell}}
\newcommand{\Xl}{X^{\text{exc},\ell}}
\renewcommand{\l}{\ell}

\renewcommand{\Xi}{X^{\infty}}
\newcommand{\dGHP}{d_{GHP}}
\newcommand{\Ei}{\EE_{\infty}}
\newcommand{\Fi}{\F_{\infty}}
\newcommand{\Er}{\EE_r}
\newcommand{\Err}{\EE_r'}
\newcommand{\Fr}{\F_{r}}
\newcommand{\Frr}{\F_{r}'}
\newcommand{\bF}{\beta_1}
\newcommand{\bg}{\beta_2}
\newcommand{\bE}{\beta_3}
\newcommand{\bd}{\beta_4}
\newcommand{\Refi}{R_{\text{eff}}^{\infty}}
\newcommand{\Height}{\textsf{Height}}
\newcommand{\PhiJS}{\Phi_{\text{JS}}(T)}
\newcommand{\Reff}{R_{\text{eff}}}
\newcommand{\Tai}{T_{\alpha}^{\infty}}

\allowdisplaybreaks
\begin{document}

\title{Infinite stable looptrees}
\author{Eleanor Archer\thanks{Mathematics Institute, University of Warwick, Coventry CV4 7AL, United Kingdom. Email: {E.Archer.1@warwick.ac.uk}. Supported by EPSRC Grant Number EP/N509796/1.}}
\maketitle

\begin{abstract}
We give a construction of an infinite stable looptree, which we denote by $\Lai$, and prove that it arises both as a local limit of the compact stable looptrees of Curien and Kortchemski (2015), and as a scaling limit of the infinite discrete looptrees of Richier (2017) and Bj\"ornberg and Stef\'ansson (2015). As a consequence, we are able to prove various convergence results for volumes of small balls in compact stable looptrees, explored more deeply in a companion paper. We also establish the spectral dimension of $\Lai$, and show that it agrees with that of its discrete counterpart. Moreover, we show that Brownian motion on $\Lai$ arises as a scaling limit of random walks on discrete looptrees, and as a local limit of Brownian motion on compact stable looptrees, which has similar consequences for the limit of the heat kernel.
\end{abstract}
%


\section{Introduction}
Stable looptrees are a class of random fractal objects indexed by a parameter $\alpha \in (1,2)$ and can informally be thought of as the dual graphs of stable trees. Motivated by \cite{LeGMiermontScalingLimitsLargeFaces}, they were originally introduced by Curien and Kortchemski in \cite{RSLTCurKort}, and along with their discrete counterparts have been shown to be of increasing significance in the study of statistical mechanics models on random planar maps. For example, the same authors showed in \cite{CurKortUIPTPerc} that a stable looptree arises as the scaling limit of the boundary of a critical percolation cluster on the UIPT, and Richier showed in \cite{RichierIICUIHPT} that the incipient infinite cluster of the UIHPT has the form of an infinite discrete looptree. Further results along these lines can be found in \cite{CurKortUIPTPerc}, \cite{CurKortDuqMan}, \cite{StefStufBolzOuterplanar}, \cite{BaurRichUIPQSkew}, \cite{CurRichDualityRPMPerc} and \cite{KortRichBoundaryRPMLooptrees}, though this is a very non-exhaustive list. More generally, they also arise as the scaling limits of boundaries of stable maps \cite{RichierMapBoundaryLimit}, and are emerging as an important tool in the programme to reconcile the theories of random planar maps and Liouville quantum gravity, demonstrated for example in \cite{MillSheff}, \cite{GwynnePfefferConnectivitySLE} and \cite{BernardiHoldenSun}.

\begin{figure}[h]
\includegraphics[width=14cm]{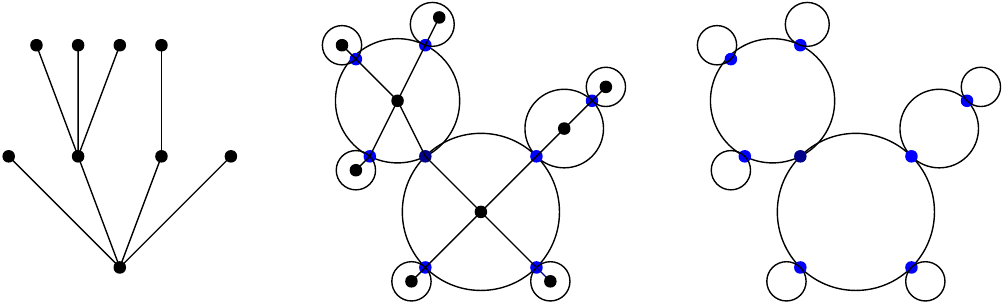}
\centering
\caption{A tree $T$ and the looptree $\Loop (T)$.}\label{fig:disc looptree intro}
\end{figure}

Given a discrete tree $T$, the corresponding discrete looptree ${\Loop}(T)$ as defined in \cite{RSLTCurKort} is constructed by replacing each vertex $u \in T$ with a cycle of length equal to the degree of $u$ in $T$, and then gluing these cycles along the tree structure of $T$. This is illustrated in Figure \ref{fig:disc looptree intro}. This operation can also be applied in the case where $T$ is an infinite tree. If $T$ is rooted, we will take the convention that the root of $\Loop (T)$ is the vertex of $\Loop (T)$ corresponding to the edge of $T$ joining the root of $T$ to its first child.

In this article we will mainly be interested in the case where our tree $T$ has a critical offspring distribution in the domain of attraction of an $\alpha$-stable law, by which we mean that there exists an increasing sequence $a_n \uparrow \infty$ such that, if $(\xi^{(i)})_{i=1}^{\infty}$ are i.i.d. copies of $\xi$, then
\begin{equation*}\label{eqn:dom of att def}
\frac{\sum_{i=1}^n \xi^{(i)} - n}{a_n} \overset{(d)}{\rightarrow} Z_{\alpha}
\end{equation*}
as $n \rightarrow \infty$, where $Z_{\alpha}$ is an $\alpha$-stable random variable (and can be normalised so that $\E{e^{-\lambda Z_{\alpha}}} = e^{-\lambda^{\alpha}}$ for all $\lambda > 0$). It is shown in \cite[Section 8.3.2]{BGTRegVariation} that necessarily $a_n = n^{\frac{1}{\alpha}}L(n)$ for some slowly-varying function $L$, where we recall that slowly varying means that $L(x) > 0$ for all sufficiently large $x$, and $\lim_{x \rightarrow \infty} \frac{L(tx)}{L(x)} = 1$ for all $t>0$.

Equivalently, $\xi([n, \infty)) = k^{-\alpha} L(n)$. In the case where $\xi([n, \infty)) \sim cn^{-\alpha}$ as $n \rightarrow \infty$, we can take $a_n = (c |\Gamma(-\alpha)|n)^{\frac{1}{\alpha}}$.

Throughout the article we will make the assumption that $\alpha \in (1,2)$. In \cite[Theorem 4.1]{RSLTCurKort}, it is shown that if $T_n$ is a Galton Watson tree conditioned to have $n$ vertices with offspring distribution $\xi$ in the domain of attraction of an $\alpha$-stable law, then we can define the $\alpha$-stable looptree (which we denote by $\L_{\alpha}$) to be the random compact metric space such that 
\[
a_n^{-1} {\Loop}(T_n) \overset{(d)}{\rightarrow} \L_{\alpha}
\]
in the Gromov-Hausdorff topology as $n \rightarrow \infty$. A simulation is shown in Figure \ref{fig:stable looptree intro}. In the case $\alpha = 2$, the looptrees instead rescale to the Brownian Continuum Random Tree \cite[Theorem 2]{KortRichCondensationCritical}.

\begin{figure}[h]
\includegraphics[width=10cm]{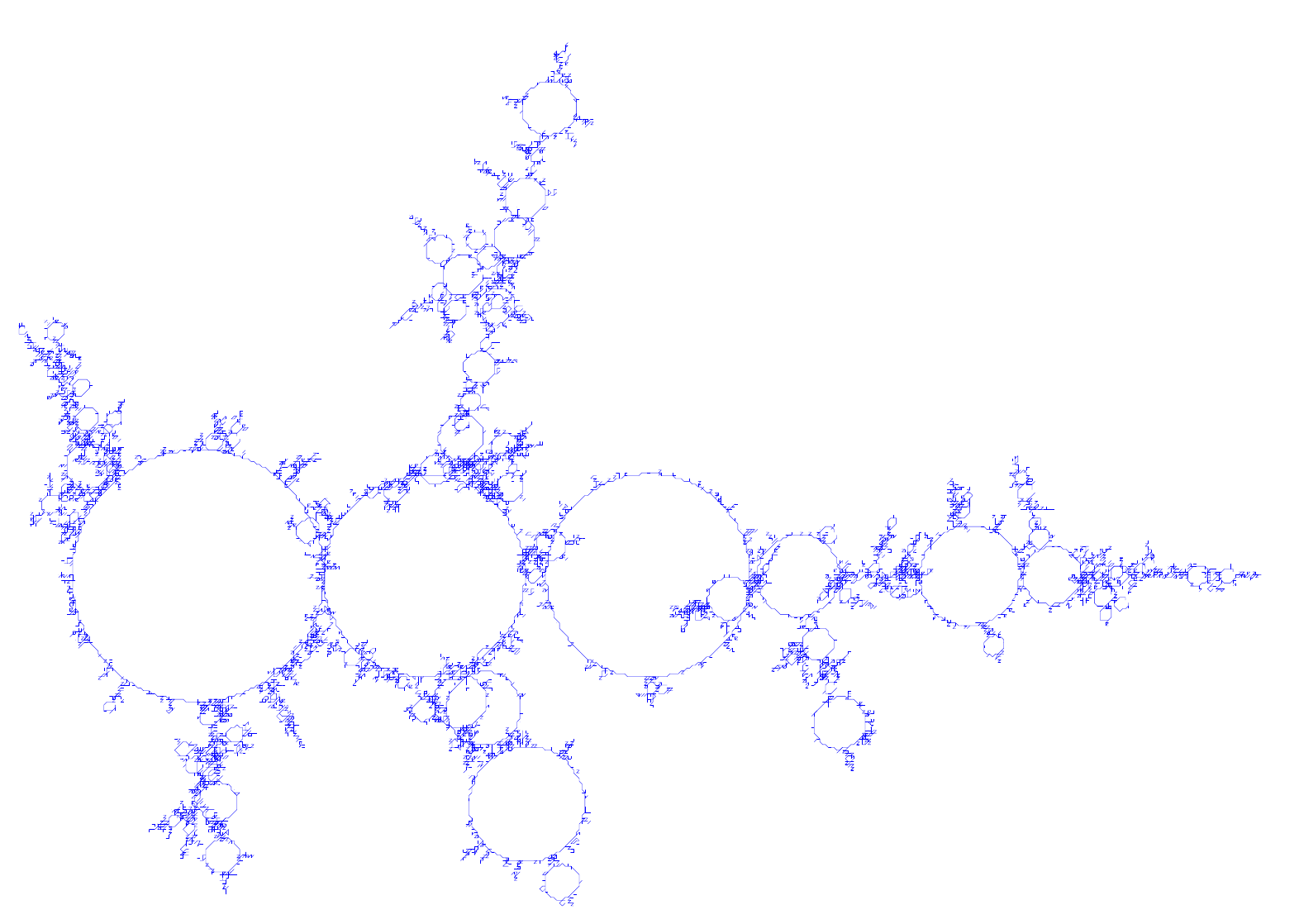}
\centering
\caption{Simulation of a stable looptree, by Igor Kortchemski.}\label{fig:stable looptree intro}
\end{figure}

The main purpose of this paper is to give a construction of infinite stable looptrees. The construction is similar in spirit to Duquesne's construction of stable sin-trees in \cite{DuqSinTree}, which is the continuum analogue of Kesten's discrete construction of an infinite critical tree. Additionally, infinite discrete looptrees have been defined by Bj\"ornberg and Stef\'ansson in \cite{BjornStef} by applying a related loop operation to Kesten's infinite critical tree $T_{\infty}$, and similarly by Richier in \cite{RichierIICUIHPT} by applying a similar operation to a two-type version of Kesten's tree.

As is done for stable sin-trees in \cite{DuqSinTree}, we define the infinite stable looptree $\Lai$ from two independent stable \Levy processes, each of which is used to code the looptree on one side of its singly infinite loopspine. This is the construction suggested in \cite[Section 6]{RichierIICUIHPT} and is the natural extension of the coding mechanism used to define stable looptrees from stable \Levy excursions.

The construction is given in Section \ref{sctn:construction of infinite stable looptrees}. The remainder of the article is devoted to proving various limit theorems to justify the definition, and then using these to make deductions about Brownian motion on compact stable looptrees, which is explored more deeply in the companion paper \cite{ArchBMCompactLooptrees}. In particular, we prove a local limit theorem showing that $\Lai$ can be characterised as the local limit of compact stable looptrees, and also as the scaling limit of infinite discrete looptrees. When combined with earlier results of Curien and Kortchemski, Bj\"ornberg and Stef\'ansson, and Richier, this shows that the diagram of Figure \ref{fig:commuting diag intro} commutes as indicated. 

\begin{figure}[h]
\begin{center}
\includegraphics[width=13cm]{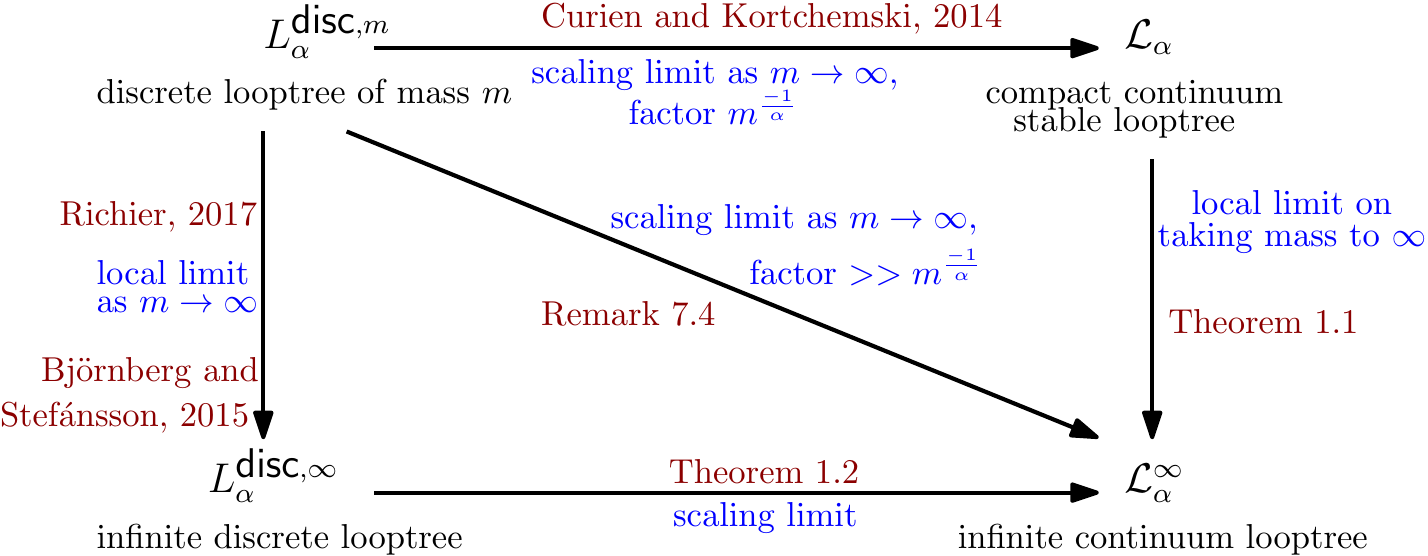}
\caption{Commuting Diagram.}\label{fig:commuting diag intro}
\end{center}
\end{figure}

We start by giving the local limit result. In what follows, we let $\Lal$ be a compact stable looptree conditioned to have total volume $\l$, and let $\La^{\infty}$ be as above. We recall from \cite{RSLTCurKort} that $\Lal$ is endowed with a measure $\nu^{\l}$ which can be thought of as the natural analogue of uniform measure on $\Lal$. We will define a similar measure on $\Lai$ in Section \ref{sctn:construction of infinite stable looptrees}, and denote it by $\nu^{\infty}$. We also recall from \cite{RSLTCurKort} (respectively \cite{ArchBMCompactLooptrees}) that there is a natural way to define shortest-distance metric (respectively a resistance metric) on $\Lal$, and we will define analogous metrics for $\Lai$ in Section \ref{sctn:construction of infinite stable looptrees}.

\begin{theorem}\label{thm:LLT}
Let $\Lal$ be a compact stable looptree conditioned to have mass $\l$, and let $\La^{\infty}$ be as above. Then,
\begin{equation*}
(\Lal, \d^{\l}, \nu^{\l}, \rho^{\l}) \overset{(d)}{\rightarrow} (\La^{\infty}, \d^{\infty}, \nu^{\infty}, \rho^{\infty})
\end{equation*}
as $\l \rightarrow \infty$, with respect to the Gromov-Hausdorff-vague topology. Here $\d^{\l}$ and $\d^{\infty}$ can denote either the geodesic metrics, or the effective resistance metrics on the respective spaces.
\end{theorem}

Similarly, we prove the following scaling result.

\begin{theorem}\label{thm:main scaling lim}
Let $\Tai$ denote Kesten's tree with critical offspring distribution in the domain of attraction of an $\alpha$-stable law. Also let $\nu^{\text{disc}}$ denote the measure that gives mass $1$ to every vertex of \textsf{Loop}($\Tai$). Then
\[
(\textsf{Loop}(\Tai), a_n^{-1} \d, n^{-1} \nu^{\text{disc}}, \rho) \overset{(d)}{\rightarrow} (\Lai, \d^{\infty}, \nu^{\infty}, \rho^{\infty})
\]
with respect to the Gromov-Hausdorff vague topology as $n \rightarrow \infty$. Here $\d$ and $\d^{\infty}$ can denote either the geodesic metrics, or the effective resistance metrics on the respective spaces.
\end{theorem}

We will see in Section \ref{sctn:scaling lims} that similar results hold for the infinite discrete looptrees defined in \cite{BjornStef} and \cite{RichierIICUIHPT}.

Given these two theorems, we are also in the right setting to apply results of \cite{DavidResForms} regarding limits for stochastic processes on these spaces. In particular, we obtain the following results. Note that we formally define Brownian motion on stable looptrees in the article \cite{ArchBMCompactLooptrees} by defining it to be the stochastic process naturally associated with the effective resistance metric on them. In Section \ref{sctn:construction of infinite stable looptrees}, we similarly define an effective resistance metric on $\Lai$ and in Section \ref{sctn:RW consequences} we define Brownian motion on $\Lai$ to be the associated stochastic process. We denote it by $B^{\infty}$. For convenience, we restrict to the case where $\l$ takes integer values below, but the result holds along any countable subsequence diverging to infinity.

\begin{theorem}\label{thm:main RW LLT conv}
Let $(B^{\l}_t)_{t \geq 0}$ be Brownian motion on $\Lal$, and let $(B^{\infty}_t)_{t \geq 0}$ be Brownian motion on $\Lai$. Then there exists a probability space $(\Omega', \mathcal{F}', \mathbf{P}')$ on which we can almost surely define a metric space $(M, R_M)$ in which the spaces $(\Lal, R^{\l}, \nu^{\l}, \rho^{\l})$ and $(\La^{\infty}, R^{\infty}, \nu^{\infty}, \rho^{\infty})$ can all be embedded and such that
\[
(\Lal, R^{\l}, \nu^{\l}, \rho^{\l}) \rightarrow (\La^{\infty}, R^{\infty}, \nu^{\infty}, \rho^{\infty})
\]
with respect to the Gromov-Hausdorff-vague topology as $\l \rightarrow \infty$, and the required Hausdorff convergence specifically holds in the metric space $(M, R_M)$. Letting $(B^{\l})_{\l \geq 1}$ and $B^{\infty}$ be as above, we have that 
\[
(B^{\l}_t)_{t \geq 0} \overset{(d)}{\rightarrow} (B^{\infty}_t)_{t \geq 0}
\]
as $\l \rightarrow \infty$, considered on the space $C(\R^+, M)$ endowed with the topology of uniform convergence on compact time intervals.
\end{theorem} 

\begin{theorem}\label{thm:RW conv infinite intro}
Let $(\textsf{Loop}(\Tai), a_n^{-1}\d, n^{-1} \nu^{\text{disc}}, \rho)$ be as in Theorem \ref{thm:main scaling lim}. Then there exists a probability space $(\Omega'', \mathcal{F}'', \mathbf{P}'')$ on which we can almost surely define a metric space $(M, R_M)$ in which the spaces $(\textsf{Loop}(\Tai), Ca_n^{-1} \d, n^{-1} \nu^{\text{disc}}, \rho)$ and $(\La^{\infty}, {\d}^{\infty}, \nu^{\infty}, \rho^{\infty})$ can all be embedded and such that
\[
(\textsf{Loop}(\Tai), a_n^{-1} \d, n^{-1} \nu^{\text{disc}}, \rho) \overset{(d)}{\rightarrow} (\Lai, \d^{\infty}, \nu^{\infty}, \rho^{\infty})
\]
with respect to the Gromov-Hausdorff-vague topology as $n \rightarrow \infty$, and the required Hausdorff convergence specifically holds on the metric space $(M, R_M)$. Letting $Y$ be a simple random walk on $\textsf{Loop}(\Tai)$, and $B^{\infty}$ be as above, we have that 
\[
(a_n^{-1} Y_{\lfloor 4n a_n t \rfloor})_{t \geq 0} \overset{(d)}{\rightarrow} (B^{\infty}_t)_{t \geq 0}
\]
on the space $D(\R^+, M$) endowed with the Skorohod-$J_1$ topology as $n \rightarrow \infty$.
\end{theorem} 

Again, we will prove a similar result for random walks on the other infinite discrete looptrees in Section \ref{sctn:RW consequences}, along with annealed versions, but the one above is easiest to state as all vertices have degree $4$ in \textsf{Loop}($T_{\infty}$).

The process $B^{\infty}$ is considered further in Section \ref{sctn:RW consequences} where we prove the following results about the spectral dimension of $\Lai$. Recall that the spectral dimension of $\Lai$ is defined as
\begin{equation}\label{eqn:trans dens def}
d_S(\Lai) = \lim_{t \rightarrow \infty} \frac{-2\log( p^{\infty}_t(\rho^{\infty}, \rho^{\infty}))}{\log t},
\end{equation}
where $p^{\infty}_t(\cdot, \cdot)$ is the transition density of the Brownian motion $B^{\infty}$ defined above, i.e. a symmetric $\nu^{\infty} \times \nu^{\infty}$-measurable function on $\Lai \times \Lai$ such that
\[
\estartb{f(B_t)}{x} = \int_{\Lai} f(y) p_t(x,y) \nu^{\infty} (dy)
\]
for all bounded, $\nu^{\infty}$-measurable functions $f$ on $\Lai$ and $\nu^{\infty}$-almost every $x \in \Lai$.

We assume that $\Lai$ is defined on the probability space $(\Omega, \F, \mathbf{P})$, and let $\mathbf{E}$ denote expectation on this space.

\begin{theorem}\label{thm:main spec dim quenched}
$\mathbf{P}$-almost surely, $d_S(\Lai) = \frac{2\alpha }{\alpha + 1}$.
\end{theorem}

In light of Theorem \ref{thm:main spec dim quenched}, we call $d_S(\Lai)$ the \textit{quenched} spectral dimension. We also define the \textit{annealed} spectral dimension as
\[
d^a_S(\Lai) = \lim_{t \rightarrow \infty} \frac{-2\log( \Eb{p^{\infty}_t(\rho^{\infty}, \rho^{\infty})})}{\log t}.
\]
For a general space, the annealed heat kernel is trickier to bound than the quenched one defined above, since the expected transition density may not be finite. This is the case, for example, for the trees with heavy-tailed offspring distributions considered in \cite{CroyKumRWGWTreeInfiniteVar}. In the case of stable looptrees however we are able to bound this using the volume and resistance estimates of Section \ref{sctn:vol bounds and spectral dim infinite}, and then utilise scaling invariance of $\Lai$ to prove the following (more precise) result.

\begin{theorem}\label{thm:main spec dim annealed}
We have that
\[
d^a_S(\Lai) = \frac{2\alpha }{\alpha + 1}.
\]
Moreover, there exists a constant $c_1 \in (0, \infty)$ such that $\Eb{p^{\infty}_t(\rho^{\infty}, \rho^{\infty})} = c_1 t^{\frac{-\alpha}{\alpha + 1}}$.
\end{theorem}

Both the quenched and annealed spectral dimensions match those obtained for the infinite discrete looptrees defined from offspring distributions in the domain of attraction of an $\alpha$-stable law in \cite{BjornStef}. 

The results of this paper and in particular Theorem \ref{thm:LLT} are applied in the paper \cite{ArchBMCompactLooptrees} to prove various limit results for volumes of small balls in compact stable looptrees, and also to obtain limiting heat kernel estimates in the regime $t \downarrow 0$. We refer the reader directly to \cite{ArchBMCompactLooptrees} for more details. Moreover, Richier showed in \cite{RichierIICUIHPT} that the incipient infinite cluster (IIC) of the Uniform Infinite Half-Planar Triangulation has the structure of an infinite discrete looptree, but where each of the loops are filled with independent critically percolated Boltzmann triangulations. The size of the loops of this looptree are given by a distribution in the domain of attraction of a $\frac{3}{2}$-stable law and Theorem \ref{thm:Rich looptree conv} will imply that the boundary of this cluster converges after rescaling to the infinite stable looptree $\L_{3/2}^{\infty}$. The question of the scaling limit of the whole cluster is more subtle but is conjectured to be the $\frac{7}{6}$-stable map \cite[Section 5.4]{BerCurMierPerconTriang}, and we hope the methods used in this article will be a good starting point for studying random walks on the IIC. In particular, we anticipate that such a random walk might fall into a framework similar to the discussions of \cite{AlRuiFreiKigSSG}, in that the looptree forming the boundary of the IIC may play a role somewhat analogous to that of the classical Sierpinski gasket in \cite{AlRuiFreiKigSSG}. If this is the case, then understanding Brownian motion on $\La$ and $\Lai$ is an important preliminary step to understanding the scaling limit of a random walk on the IIC.

This paper is organised as follows. In Section \ref{sctn:prelim} we go over some preliminaries on \Levy processes and stochastic processes associated with resistance forms. In Section \ref{sctn:tree looptree def} we give some background on random trees and looptrees and explain how the stable versions can be coded by \Levy excursions. In Section \ref{sctn:construction of infinite stable looptrees} we give our construction of $\Lai$, which essentially involves replacing the \Levy excursion used to code a compact looptree by two independent \Levy processes. In Section \ref{sctn:vol bounds and spectral dim infinite} we prove some precise volume and resistance bounds for $\Lai$ by making comparisons with arguments of \cite{ArchBMCompactLooptrees}. We then proceed to prove Theorems \ref{thm:LLT} and \ref{thm:main scaling lim} in Section \ref{sctn:LLT}, and explain how these are applied to study compact stable looptrees in \cite{ArchBMCompactLooptrees}. Finally, we conclude with a study of stochastic processes in Section \ref{sctn:RW consequences}, where we use Theorems \ref{thm:LLT} and \ref{thm:main scaling lim} to prove Theorems \ref{thm:main RW LLT conv} and \ref{thm:RW conv infinite intro}, and also prove Theorems \ref{thm:main spec dim quenched} and \ref{thm:main spec dim annealed}.

Throughout this paper, $C, C', c$ and $c'$ will denote constants, bounded above and below, that may change on each appearance. We will use the notation $B^{\infty} (x, r)$ to denote the open ball of radius $r$ around $x$ in $\Lai$, and $\bar{B}^{\infty} (x, r)$ its closure. We will instead use the superscript $\l$ to denote the corresponding quantities on a compact looptree conditioned to have mass $\l$.

\textbf{Acknowledgements.} I would like to thank my supervisor David Croydon for suggesting the problem and for useful discussions, and two anonymous referees for their helpful comments on the original version of this work. I would also like to thank the Great Britain Sasakawa Foundation for supporting a trip to Kyoto during which some of this work was completed, and Kyoto University for their hospitality during this trip.

\section{Preliminaries}\label{sctn:prelim}
\subsection{Gromov-Hausdorff-Prohorov topologies}
In order to prove convergence results for measured metric spaces such as looptrees we will work in the pointed Gromov-Hausdorff-Prohorov topology. To define this, let $\mathbb{F}$ denote the set of quadruples $(F,R,\mu,\rho)$ such that $(F,R)$ is a boundedly finite Heine-Borel metric space, $\mu$ is a locally finite Borel measure of full support on $F$, and $\rho$ is a distinguished point of $F$, which we call the root. Let $\mathbb{F}^c \subset \mathbb{F}$ denote the subset of spaces where $(F,R)$ is compact.

Suppose $(F,R,\mu,\rho)$ and $(F',R',\mu',\rho')$ are elements of $\mathbb{F}^c$. Given a metric space $(M, d_M)$, and isometric embeddings $\phi, \phi'$ of $(F,R)$ and $(F', R')$ respectively into $(M, d_M)$, we define $d^{GHP}_{M}\big((F,R,\mu,\rho, \phi), (F',R',\mu',\rho', \phi')\big)$ to be equal to
\begin{align*}\label{eqn:GHP def}
d_M^H(\phi (F), \phi' (F')) + &d_M^P(\mu \circ \phi^{-1}, \mu' \circ {\phi'}^{-1} ) + d_M(\phi (\rho), \phi' (\rho')).
\end{align*}
Here $d_M^H$ denotes the Hausdorff distance between two sets in $M$, and $d_M^P$ the Prohorov distance between two measures, as defined in \cite[Chapter 1]{BillsleyConv}. The pointed Gromov-Hausdorff-Prohorov distance between $(F,R,\mu,\rho)$ and $(F',R',\mu',\rho')$ is given by
\begin{align}
\begin{split}
\dGHPp{(F,R,\mu,\rho), (F',R',\mu',\rho')} = \inf_{\phi, \phi', M} d^{GHP}_{M}\big((F,R,\mu,\rho, \phi), (F',R',\mu',\rho', \phi')\big)\end{split}
\end{align}
where the infimum is taken over all isometric embeddings $\phi, \phi'$ of $(F,R)$ and $(F', R')$ respectively into a common metric space $(M, d_M)$. It is well-known (for example, see \cite[Theorem 2.3]{AbDelHoschNoteGromov}) that this defines a metric on the space of equivalence classes of $\mathbb{F}^c$, where we say that two spaces $(F,R,\mu,\rho)$ and $(F',R',\mu',\rho')$ are equivalent if there is a measure and root preserving isometry between them.

The pointed Gromov-Hausdorff distance $d_{GH}(\cdot, \cdot)$, which is defined by removing the Prohorov term from (\ref{eqn:GHP def}) above, can also be helpfully defined in terms of \textit{correspondences}. A correspondence $\mathcal{R}$ between $(F,R,\mu,\rho)$ and $(F',R',\mu',\rho')$ is a subset of $F \times F'$ such that for every $x \in F$, there exists $y \in F'$ with $(x,y) \in \mathcal{R}$, and similarly for every $y \in F'$, there exists $x \in F$ with $(x,y) \in \mathcal{R}$. We define the \textit{distortion} of a correspondence by
\[
\textsf{dis} (\mathcal{R}) = \sup_{(x,x'), (y, y') \in \mathcal{R}} |R(x,y) - R(x',y')|.
\]
It is then straightforward to show that
\begin{align*}
\dGH{(F,R,\mu,\rho), (F',R',\mu',\rho')} = \frac{1}{2} \inf_{\mathcal{R}} \textsf{dis}(\mathcal{R}),
\end{align*}
\sloppy{where the infimum is taken over all correspondences $\mathcal{R}$ between $(F,R,\mu,\rho)$ and $(F',R',\mu',\rho')$~that contain the point $(\rho, \rho')$.}

In this article, we will prove pointed Gromov-Hausdorff-Prohorov convergence by first proving  pointed Gromov-Hausdorff convergence using correspondences, and then show Prohorov convergence of the measures on the resulting metric embedding.

For non-compact elements of $\mathbb{F}$, we will need a generalised notion of Gromov-Hausdorff-Prohorov convergence. This is provided by the Gromov-Hausdorff vague topology of \cite{AthLohrWinGromovGap}. To define it, suppose that $(F,R,\mu,\rho)$ and $(F_n,R_n,\mu_n,\rho_n)$ for each $n \geq 0$ are elements of $\mathbb{F}\setminus \mathbb{F}^c$. For $r>0$, we let $\mathcal{B}_r(F)$ denote the quadruple $(\bar{B}_F(\rho, r),R|_{\bar{B}_F(\rho, r)},\mu|_{\bar{B}_F(\rho, r)},\rho)$, where $\bar{B}_F(\rho, r)$ denotes the closed ball of radius $r$ around the root $\rho$ in $F$; similarly for $\mathcal{B}_r(F_n)$. Recall that we are restricting to Heine-Borel metric measure spaces of full support, so that weak convergence is metrized by the Prohorov metric. Following \cite[Definition 5.8]{AthLohrWinGromovGap}, we say that $(F_n,R_n,\mu_n,\rho_n)$ converges to $(F,R,\mu,\rho)$ in the Gromov-Hausdorff-vague topology if
\[
\dGHP \big( \mathcal{B}_r(F_n), \mathcal{B}_r(F) \big) \rightarrow 0
\]
for Lebesgue almost every $r>0$. The following proposition will be useful, as it will allow us to apply the Skorohod Representation Theorem later in Sections \ref{sctn:LLT} and \ref{sctn:RW consequences}.

\begin{proposition}\cite[Proposition 5.12]{AthLohrWinGromovGap}.\label{thm:GH vague separable HB}
The space of Heine-Borel boundedly finite measure spaces equipped with the Gromov-Hausdorff-vague topology is a Polish space.
\end{proposition}

\subsection{Stochastic processes associated with resistance metrics}\label{sctn:res forms}
To study Brownian motion and random walks on metric spaces we will be using the theory of resistance forms and resistance metrics, developed by Kigami in \cite{AOF} and  \cite{KigamiResistanceFormsMono}.

Let $G = (V,E)$ be a discrete graph equipped with non-negative symmetric edge conductances $c(x,y)_{(x,y) \in E}$ and a measure $(\mu(x))_{x \in V}$. Effective resistance on $G$ is a function $R$ on $V \times V$ defined by
\begin{equation}\label{eqn:resistance def variational}
R(x,y)^{-1} = \inf \{ \mathcal{E}(f,f)| f: V \rightarrow \R, f(x)=1, f(y)=0 \},
\end{equation}
where we take the convention that $\inf \emptyset = \infty$, and $\mathcal{E}(f,f)$ is an energy functional given by
\begin{equation*}\label{eqn:energy def}
\mathcal{E}(f,g) = \frac{1}{2} \sum_{x, y \in V} c(x,y) (f(y)- f(x))(g(y) - g(x)).
\end{equation*}
$R(x,y)$ corresponds to the usual physical notion of electrical resistance between $x$ and $y$ in $G$. It can be shown (e.g. see \cite{Tetali}) that $R$ is a metric on $G$, so we call it the \textit{resistance metric}.

%
The notion of a resistance metric can be extended to the continuum as follows.

\begin{definition}\label{def:eff resistance metric}\cite[Definition 2.3.2]{AOF}.
Let $F$ be a set. A function $R : F \times F$ is a \textit{resistance metric} on $F$ if and only if for every finite subset $V \subset F$, there exists a weighted graph with vertex set $V$ such that $R|_{V \times V}$ is the effective resistance on $V$, i.e. is given by (\ref{eqn:resistance def variational}).
\end{definition}

A resistance metric on a set $F$ can be naturally associated with a stochastic process on $F$ via the theory of resistance forms. Roughly speaking, a resistance form is a pair $(\EE, \F)$ where $\EE$ is an energy functional as above, and $\F$ is a subspace of real-valued functions on $F$ with finite energy (additionally it must satisfy the so-called Markov property, see \cite[Definition 3.1]{KigamiResistanceFormsMono}).

\begin{definition}(\cite[Definition 6.2]{KigamiResistanceFormsMono}).\label{def:reg res form}
A resistance form $(\EE, \F)$ is \textit{regular} if $\F \cup C_0(F)$ is dense in $C_0(F)$ with respect to the supremum norm, where $C_0(F)$ represents the space of continuous functions on $F$ with compact support.
\end{definition}

By \cite[Theorems 2.3.4 and 2.3.6]{AOF}, there is a one-to-one correspondence between resistance metrics and resistance forms on $F$, given analogously to (\ref{eqn:resistance def variational}). Moreover, if the corresponding resistance form is regular, then it induces a regular Dirichlet form on the space $L^2(F, \mu)$, which in turn is naturally associated with a Hunt process on $F$ as a consequence of \cite[Theorem 7.2.1]{FOT}. This is automatically the case when $(F,R)$ is a compact resistance metric space endowed with a finite Borel measure $\mu$ of full support, for example, but in the case of infinite looptrees we will have to put some extra work into proving that the resistance form associated with $\Lai$ is regular. This is done in Proposition \ref{prop:regular DF}.

We have tried to keep background on resistance forms and Dirichlet forms to a minimum in this article, but see \cite{KigamiResistanceFormsMono} for more on this. The key point is that, under appropriate regularity conditions on the underlying space (which will always be fulfilled in this paper), there is a one-to-one correspondence between resistance metrics and stochastic processes. The reader should feel free to skip the proof of Proposition \ref{prop:regular DF}, which proves the required regularity in our setting, and merely use this correspondence as a black box throughout the rest of this article.

This correspondence allows us to use results about scaling limits of measured resistance metric spaces to prove results about scaling limits of stochastic processes as detailed in the following result of \cite{DavidResForms}. Before stating it, we note that the notion of effective resistance between points given in (\ref{eqn:resistance def variational}) can be extended to that of effective resistance between two sets $A, B \subset F$ by setting
\begin{equation*}
R(A,B)^{-1} = \inf \{ \mathcal{E}(f,f)| f: F \rightarrow \R, f(x)=1 \ \forall \ x \in A, f(y)=0 \ \forall \ y \in B \}.
\end{equation*}

\begin{theorem}\cite[Theorem 1.2]{DavidResForms}.\label{thm:scaling lim RW resistance}
Suppose that $(F_n, R_n, \mu_n, \rho_n)_{n \geq 0}$ is a sequence in $\mathbb{F}$ such that 
\[
(F_n, R_n, \mu_n, \rho_n) \rightarrow (F, R, \mu, \rho)
\]
Gromov-Hausdorff-vaguely for some $(F, R, \mu, \rho) \in \mathbb{F}$, and $R, (R_n)_{n \geq 1}$ are resistance metrics on the respective spaces. Assume further that 
\begin{equation}\label{eqn:nonexplosion}
\lim_{r \rightarrow \infty} \liminf_{n \rightarrow \infty} R_n (\rho_n, B_n (\rho_n, r)^c) = \infty.
\end{equation}
Let $(Y_t^n)_{t \geq 0}$ and $(Y_t)_{t \geq 0}$ be the stochastic processes respectively associated with $(F_n, R_n, \mu_n, \rho_n)$ and $(F, R, \mu, \rho)$ as described above. Then it is possible to isometrically embed $(F_n, R_n)_{n \geq 1}$ and $(F,R)$ into a common metric space $(M, d_M)$ so that
\[
\prnstart{(Y_t^n)_{t \geq 0} \in \cdot}{\rho_n}{n} \rightarrow \prstart{(Y_t)_{t \geq 0} \in \cdot}{\rho}
\]
weakly as probability measures as $n \rightarrow \infty$ on the space $D(\R_+, M)$ equipped with the Skorohod $J_1$-topology.
\end{theorem}

For more on the Skorohod-$J_1$ topology, see \cite[Chapter 3]{BillsleyConv}. The intuition behind the result above is that the convergence of metrics and measures respectively give the appropriate spatial and temporal convergences of the stochastic processes. We will apply it several times in this paper to take limits of stochastic processes on looptrees. 

By isometrically embedding into the universal Urysohn space $(U, d_U)$, we can get similar results in the annealed setting. This is quite abstract, and we do not give a full background on the Urysohn space, but instead recall that it is a Polish space with the property that any separable metric space can be isometrically embedded into $U$. Moreover, it has a distinguished point $u_0$ and in the case of trees and looptrees we can always assume that the root is mapped to this canonical point. These are the only two properties of $U$ that we will use in this article, but its existence and further properties are discussed in \cite{HuvekUrysohn}.
 
Suppose that $(F_n, R_n, \mu_n, \rho_n, \psi_n)_{n \geq 0}$ is a sequence such that $(F_n, R_n, \mu_n, \rho_n)_{n \geq 0} \in \mathbb{F}$ and $\psi_n$ is an isometric embedding of $(F_n, R_n, \mu_n, \rho_n)_{n \geq 0}$ into $U$ for all $n$. Similarly for $(F, R, \mu, \rho, \psi)$. For the purposes of this paper, if $(F,R,\mu,\rho,\psi)$ is compact we will say that $(F_n, R_n, \mu_n, \rho_n, \psi_n) \rightarrow (F, R, \mu, \rho, \psi)$ in the spatial Gromov-Hausdorff topology if 
\[
d_U^{sp} \big( (F, R, \mu, \rho, \psi), (F_n, R_n, \mu_n, \rho_n, \psi_n)\big) \rightarrow 0
\]
as $n \rightarrow \infty$, where $d_U^{sp} \big( (F, R, \mu, \rho, \psi), (F_n, R_n, \mu_n, \rho_n, \psi_n)\big)$ is defined to be equal to
\begin{align}\label{eqn:spat U dist}
d_U^H(\psi (F), \psi_n (F_n)) + d_U^P(\mu \circ \psi^{-1}, \mu_n \circ {\psi_n}^{-1} ) + d_U(\psi (\rho), \psi_n (\rho_n)).
\end{align}

In the non-compact case, we will say that $(F_n, R_n, \mu_n, \rho_n, \phi_n) \rightarrow (F, R, \mu, \rho, \phi)$ in the spatial Gromov-Hausdorff vague topology if the closed balls of radius $r$ along with their appropriate restrictions converge for Lebesgue-almost every $r>0$.

This definition is a special case of the spatial Gromov-Hausdorff vague topology used in \cite[Section 7]{DavidResForms}, and it follows from the results there that $d_U^{sp}$ is a metric and induces a separable topology on the space of elements of $\mathbb{F}$ isometrically embedded into $U$. The definition can be made more general (and is more meaningful) in the case when we embed non-isometrically into a space other than $U$. In fact the point of restricting to $U$ above is that, in our setting, Gromov-Hausdorff-Prohorov convergence will automatically imply existence of isometries givingconvergence in the spatial topology introduced above, and that $U$ therefore provides a metric space on which we can consider the annealed law for random walks, defined as follows.

Given a sequence of random spaces $(F_n, R_n, \mu_n, \rho_n, \phi_n)_{n \geq 0}$ such that $(F_n, R_n, \mu_n, \rho_n) \in \mathbb{F}$ for all $n$ and $\phi_n: F_n \rightarrow U$ is an isometric embedding, we define the annealed law of the corresponding stochastic process by
\[
\prtilnstart{\phi_n(Y_t^n)_{t \geq 0} \in \cdot}{\phi_n(\rho_n)}{n} = \int \prnstart{\phi_n(Y_t^n)_{t \geq 0} \in \cdot}{\phi_n(\rho_n)}{n}  d\bPb^n,
\]
i.e. as the law of the stochastic process averaged over realisations of the underlying random metric space. (We define this analogously when there is no dependence on $n$).

\begin{theorem}\cite[Theorem 7.2]{DavidResForms}.\label{thm:scaling lim RW resistance annealed}
Suppose that $(F_n, R_n, \mu_n, \rho_n, \phi_n)_{n \geq 0}$ is a sequence such that
\[
(F_n, R_n, \mu_n, \rho_n, \phi_n) \overset{(d)}{\rightarrow} (F, R, \mu, \rho, \phi)
\]
in the spatial Gromov-Hausdorff-vague topology, and $R, (R_n)_{n \geq 1}$ are resistance metrics on the respective spaces. Assume further that 
\begin{equation}\label{eqn:nonexplosion annealed}
\lim_{r \rightarrow \infty} \liminf_{n \rightarrow \infty} \pr{R_n (\rho_n, B_n (\rho_n, r)^c) \geq \lambda} = 1
\end{equation}
for all $\lambda > 0$. Let $(Y_t^n)_{t \geq 0}$ and $(Y_t)_{t \geq 0}$ be the stochastic processes respectively associated with $(F_n, R_n, \mu_n, \rho_n)$ and $(F, R, \mu, \rho)$ as described above. Then
\[
\prtilnstart{\phi_n(Y_t^n)_{t \geq 0} \in \cdot}{\rho_n}{n} \rightarrow \prtilstart{\phi(Y_t)_{t \geq 0} \in \cdot}{\rho}
\]
weakly as probability measures as $n \rightarrow \infty$ on the space $D(\R_+, U)$ equipped with the Skorohod $J_1$-topology.
\end{theorem}

\subsection{Stable \Levy excursions}\label{sctn:Levy background}
Following the presentations of \cite{DuqContourLimit} and \cite{RSLTCurKort}, we now introduce stable \Levy excursions, which will be used to code stable trees and looptrees in Section \ref{sctn:tree looptree def}.

Given intervals $I,J \subset \R$, we first recall that $D(I, J)$ represents the space of \cadlag functions from $I$ to $J$. For an interval $[0, \l] \subset \R$, we also define the \cadlag excursion space $D^{\text{exc}}([0,\l], \R^{\geq 0})$ by
\[
D^{\text{exc}}([0,\l], \R^{\geq 0}) = \{ e \in D([0,\l], \R^{\geq 0}): e(0)=e(\l)=0, e(t) > 0 \text{ for all } t \in (0, \l) \}.
\]

Throughout this article, we take $\alpha \in (1,2)$, and $X$ will be an $\alpha$-stable spectrally positive \Levy process as in \cite[Section VIII]{BertoinLevy}, normalised so that
\[
\E{e^{-\lambda X_t}} = e^{-{\lambda}^{\alpha}t}
\]
for all $\lambda > 0$. $X$ takes values in the space $D([0, \infty), \R)$ of \cadlag functions, which we endow with the Skorohod-$J_1$ topology, and satisfies the scaling property that for any constant $c>0$, $(c^{-\frac{1}{\alpha}} X_{ct})_{t \geq 0}$ has the same law as $(X_t)_{t \geq 0}$. Moreover $X$ has \Levy measure
\[
\Pi(dx) = \frac{\alpha(\alpha - 1)}{\Gamma(2-\alpha)} x^{-\alpha - 1} \mathbb{1}_{(0, \infty)}(x) dx.
\]

To define a normalised excursion of $X$, we follow \cite{Chaumont} and let $\underline{X}_t = \inf_{s \in [0,t]} X_s$ denote its running infimum process, and set
\begin{align*}
g_1 = \sup \{ s \leq 1: X_s = \underline{X}_s \}, \hspace{10mm} d_1 = \inf \{ s > 1: X_s = \underline{X}_s \}.
\end{align*}

Note that $X_{g_1} = X_{d_1}$ almost surely, since $X$ is spectrally positive. As in \cite[Proposition 1]{Chaumont}, we define the normalised excursion $X^{\text{exc}}$ of $X$ above its infimum at time $1$ by
\[
X_s^{\text{exc}} = (d_1 - g_1)^{\frac{-1}{\alpha}} (X_{g_1 + s(d_1 - g_1)} - X_{g_1})
\]
for every $s \in [0,1]$. Note that $X^{\text{exc}}$ is almost surely an $\alpha$-stable \cadlag function on $[0,1]$ with $X^{\text{exc}}(s)>0$ for all $s \in (0,1)$, and $X_0^{\text{exc}}=X_1^{\text{exc}}=0$.

\subsubsection{\Ito excursion measure}\label{sctn:Ito exc}
We can alternatively define $\X$ using the {\Ito excursion measure}. For full details, see \cite[Chapter IV]{BertoinLevy}, but the measure is defined by applying excursion theory to the process $X - \underline{X}$, which is strongly Markov and for which the point $0$ is regular for itself. We normalise local time so that $-\underline{X}$ denotes the local time of $X - \underline{X}$ at its infimum, and let $(g_j, d_j)_{j \in \mathcal{I}}$ denote the excursion intervals of $X - \underline{X}$ away from zero. For each $i \in \mathcal{I}$, the process $(e^i)_{0 \leq s \leq d_i-g_i}$ defined by $e^i(s) = X_{g_i + s} - X_{g_i}$ is an element of the excursion space
\[
E = \bigcup_{\l > 0} D^{\text{exc}}([0,\l], \R^{\geq 0}).
\]
We let $\zeta (e) = \sup \{s>0: e(s)>0\}$ denote the \textit{lifetime} of the excursion $e$. It was shown in \cite{ItoPP} that the measure
\[
N(dt, de) = \sum_{i \in \mathcal{I}} \delta (-\underline{X}_{g_i}, e^i)
\]
is a Poisson point measure of intensity $dt N(de)$, where $N$ is a $\sigma$-finite measure on the set $E$ known as the \textit{\Ito excursion measure}.

Moreover, the measure $N(\cdot)$ inherits a scaling property from the $\alpha$-stability of $X$. Indeed, for any $\lambda > 0$ we define a mapping  $\Phi_{\lambda}: E \rightarrow E$ by  $\Phi_{\lambda}(e)(t) = \lambda^{\frac{1}{\alpha}} e(\frac{t}{\lambda})$, so that $N \circ \Phi_{\lambda}^{-1} = \lambda^{\frac{1}{\alpha}} N$ (e.g. see \cite{WataIto}). It then follows from the results in \cite[Section IV.4]{BertoinLevy} that we can uniquely define a set of conditional measures $(N_{(s)}, s>0)$ on $E$ such that:
\begin{enumerate}[(i)]
\item For every $s > 0$, $N_{(s)}( \zeta=s)=1$.
\item For every $\lambda > 0$ and every $s>0$, $\Phi_{\lambda}(N_{(s)}) = N_{(\lambda s)}$.
\item For every measurable $A \subset E$
\[
N(A) = \int_0^{\infty} \frac{N_{(s)}(A)}{\alpha \Gamma(1 - \frac{1}{\alpha}) s^{\frac{1}{\alpha}+1}} ds.
\]
\end{enumerate}

$N_{(s)}$ is therefore used to denote the law $N( \cdot | \zeta = s)$. The probability distribution $N_{(1)}$ coincides with the law of $X^{\exc}$ as constructed above.

\subsubsection{Relation between $X$ and $X^{\exc}$}
It is easier to analyse an unconditioned \Levy process rather than an excursion, so throughout this paper we will use the following two tools to compare the probability of an event defined in terms of $\X$ to that of the same event defined in terms of $X$. The first tool is the Vervaat transform of the following proposition, which allows us to compare to a stable bridge $\Xb$ as an intermediate step. This is particularly useful as we will at times consider our looptrees to be rooted at a uniform point.

\begin{theorem}\cite[Th\'eor\`eme 4]{Chaumont}. Vervaat Transform.\label{thm:Vervaat}
\begin{enumerate}
\item Let $\X$ be as above, and take $U \sim$ \textsf{Uniform}$([0,1])$. Then the process $(\Xb_t)_{0 \leq t \leq 1}$ defined by
\[
\Xb_t = \begin{cases} \X_{U+t} & \text{ if } U+t \leq 1,\\
\X_{U+t-1} & \text{ if } U+t > 1.
\end{cases}
\]
has the law of a spectrally positive stable \Levy bridge on $[0,1]$.
\item Now let $\Xb$ be a spectrally positive stable \Levy bridge on $[0,1]$, and let $m$ be the (almost surely unique) time at which it attains its minimum. Define an excursion $\X$ by
\[
\X_t = \begin{cases} \Xb_{m+t} & \text{ if } m+t \leq 1,\\
\Xb_{m+t-1} & \text{ if } m+t > 1.
\end{cases}
\]
Then $\X$ has the law of a spectrally positive stable \Levy excursion.
\end{enumerate}
\end{theorem}

An event defined for the stable bridge on the interval $[0,T]$ can then be transferred to the unconditioned process using the fact that the law of the bridge is absolutely continuous with respect to the law of the process, with Radon-Nikodym derivative
\begin{equation}\label{eqn:RN deriv levy bridge}
\frac{p_{1-{T}}(-X_{T})}{p_{1}(0)}
\end{equation}
for $T \in (0,1)$ (see \cite[Section VIII.3, Equation (8)]{BertoinLevy}). Here the transition density $p_t(\cdot, \cdot)$ for the \Levy process $X$ is defined analogously to that in (\ref{eqn:trans dens def}), but with respect to Lebesgue measure on the real line. We note here that $||p_1||_{\infty} < \infty$ for all $\alpha \in (1,2)$ since $p_1(\cdot)$ is continuous and vanishes at infinity (e.g. see \cite[Section VIII.1]{BertoinLevy}).

\subsubsection{Descents}\label{sctn:Useful results}
Next, we introduce the notion of a descent of a \Levy process, following the presentation of \cite[Section 3.1.3]{RSLTCurKort}. Let $X^1$ and $X^2$ be two independent spectrally positive $\alpha$-stable \Levy processes as defined above, and define a two-sided process $X$ by setting
\[
X_t = \begin{cases} X^1_t & \text{ if } t \geq 0 \\
-X^2_{-t^-} & \text{ if } t < 0.
\end{cases}
\]
For every $s, t \in \R$, we write $s \preceq t$ if and only if $s \leq t$ and $X_{s^-} \leq \inf_{[s,t]}X$, and in this case we set
\begin{align*}
\Delta X_s = X_s - X_{s^-}, \hspace{2mm} x_s^t(X) = \inf_{[s,t]}X - X_{s^-}, \text{     and      } u_s^t(X) = \frac{x_s^t(X)}{\Delta X_s}.
\end{align*}
We write $s \prec t$ if $s \preceq t$ and $s \neq t$. As in \cite{RSLTCurKort}, for any $t \in \R$, we will call the collection $\{x_s^t(X), u_s^t(X): s \preceq t \}$ the \textit{descent} of $t$ in $X$.

The next proposition describes the law of descents from a typical point of $X$, and will be useful in the proofs of the limit theorems. We let $\overline{X}_t = \sup\{X_s: 0 \leq s \leq t\}$ denote the running supremum process of $X$. The process $\overline{X} - X$ is strong Markov and $0$ is regular for itself, allowing the use of excursion theory. Let $(L_t)_{t \geq 0}$ denote the local time of $\overline{X} - X$ at 0. Note that, by \cite[Chapter VIII, Lemma 1]{BertoinLevy}, $L^{-1}$ is a $(1 - \frac{1}{\alpha})$-stable subordinator, and $( \overline{X}_{L^{-1}(t)})_{t \geq 0}$ is an $(\alpha - 1)$-stable subordinator, so we can normalise local time so that $\E{\exp (-\lambda \overline{X}_{L^{-1}(t)}} = \exp (-t\lambda^{\alpha - 1})$ for all $\lambda > 0$. Finally, if $\overline{X}_s > \overline{X}_{s^{-}}$, set
\begin{align*}
x_s = \overline{X}_s - \overline{X}_{s^{-}}, \hspace{10mm} u_s = \frac{\overline{X}_s - \overline{X}_{s^{-}}}{\overline{X}_s - {X}_{s^{-}}}.
\end{align*}

\begin{proposition}(\cite[Proposition 3.1]{RSLTCurKort}, \cite[Corollary 1]{BertoinPitmanExt}).\label{prop:descent PP}
Let $X$ be a two-sided spectrally positive $\alpha$-stable process as above. Then
\begin{enumerate}[(i)]
\item
\begin{align*}
\{(-s, x_s^0(X), u_s^0(X)): s \preceq 0 \} \overset{(d)}{=} \{s,x_s, u_s: s \geq 0 \text{ such that } \overline{X}_s > \overline{X}_{s^-} \}.
\end{align*}
\item The point measure
\begin{align*}
\sum_{\overline{X}_s > \overline{X}_{s^-}} \delta \big(L_s, \frac{x_s}{u_s}, u_s \big)
\end{align*}
is a Poisson point measure with intensity $dl \cdot x\Pi(dx) \cdot \mathbb{1}_{[0,1]}(u) du$.
\end{enumerate}
\end{proposition}

We also give a technical lemma which will be used at various points in the paper. This appeared previously in \cite[Section 3.3.1]{RSLTCurKort} and uses an argument from \cite{BertoinLevy}. The final claim follows by bounded convergence.

First recall that for a function $f: [0, \infty) \rightarrow \R$ and $[a,b] \subset [0, \infty)$, we define
\[
\osc_{[a,b]} f := \sup_{s, t \in [a,b]} |f(t) - f(s)|.
\]

\begin{lemma}\label{lem:osc}
Let $\mathcal{E}$ be an exponential random variable with parameter $1$, and let $X$ be a spectrally positive $\alpha$-stable \Levy process conditioned to have no jumps of size greater than $1$ on $[0, \mathcal{E}]$. Let $\tilde{\osc} = \osc_{[0, \mathcal{E}]} X$. Then there exists $\theta > 0$ such that $\E{e^{\theta \tilde{\osc}}} < \infty$. Moreover, $\E{e^{\theta \tilde{\osc}}} \downarrow 1$ as $\theta \downarrow 0$.
\end{lemma}

\begin{remark}\label{rmk:osc deterministic}
The same results holds if $\EE$ is set to be deterministically equal to $1$ rather than an exponential random variable. The proof is almost identical to the proof of the result above, with one minor modification.
\end{remark}

\section{Background on stable trees and looptrees}\label{sctn:tree looptree def}

\subsection{Discrete Trees}\label{sctn:trees background discrete}
Before defining stable trees and looptrees, we briefly recap some notation for discrete trees, following the formalism of \cite{Neveu}. Firstly, let
\[
\mathcal{U}=\bigcup_{n=0}^{\infty} {\N}^n
\]
be the Ulam-Harris tree. By convention, ${\N}^0=\{ \emptyset \}$. If $u=(u_1, \ldots, u_n)$ and $v=(v_1, \ldots, v_m) \in \mathcal{U}$, we let $uv= (u_1, \ldots, u_n, v_1, \ldots, v_m)$ be the concatenation of $u$ and $v$.

\begin{definition} A plane tree $T$ is a finite subset of $\mathcal{U}$ such that
\begin{enumerate}[(i)]
\item $\emptyset \in T$,
\item If $v \in T$ and $v=uj$ for some $j \in \N$, then $u \in T$,
\item For every $u \in T$, there exists a number $k_u(T) \geq 0$ such that $uj \in T$ if and only if $1 \leq j \leq k_u(T)$.
\end{enumerate}
\end{definition}

We let $\mathbb{T}$ denote the set of all plane trees. A plane tree $T \in \mathbb{T}$ with $n+1$ vertices labelled according to the lexicographical order as $u_0, u_1, \ldots, u_n$ can be coded by its \textit{height function}, \textit{contour function}, or \textit{Lukasiewicz path}, defined as follows.
\begin{itemize}
\item The height function $(H^{T}_i)_{0 \leq i \leq n}$ is defined by considering the vertices $u_0, u_1, \ldots, u_n$ in lexicographical order, and then setting $H^{T}_i$ to be the generation of vertex $u_i$.
\item The contour function $(C^{T}_t)_{0 \leq t \leq 2n}$ is defined by considering a particle that starts at the root $\emptyset$ at time zero, and then continuously traverses the boundary of ${T}$ at speed one, respecting the lexicographical order where possible, until returning to the root. $C^{T}(t)$ is equal to the height of the particle at time $t$.
\item The Lukasiewicz path $(W^{T}_m)_{0 \leq m \leq n}$ is defined by setting $W^{T}_0 = 0$, then by considering the vertices $u_0, u_1, \ldots, u_n$ in lexicographical order and setting $W^T_{m+1} = W^T_m + k_{u_m}(T)-1$.
\end{itemize}

These are illustrated in Figure \ref{fig:contourheightfns}, together with points corresponding to specific vertices in the tree, and the part of each excursion coding the subtree rooted at the red vertex, which we denote by $\theta_1(T)$. For further details, see \cite[Section 0.1]{LeGDuqMono}.

\begin{figure}[h]
\includegraphics[width=15cm, height=5.6cm]{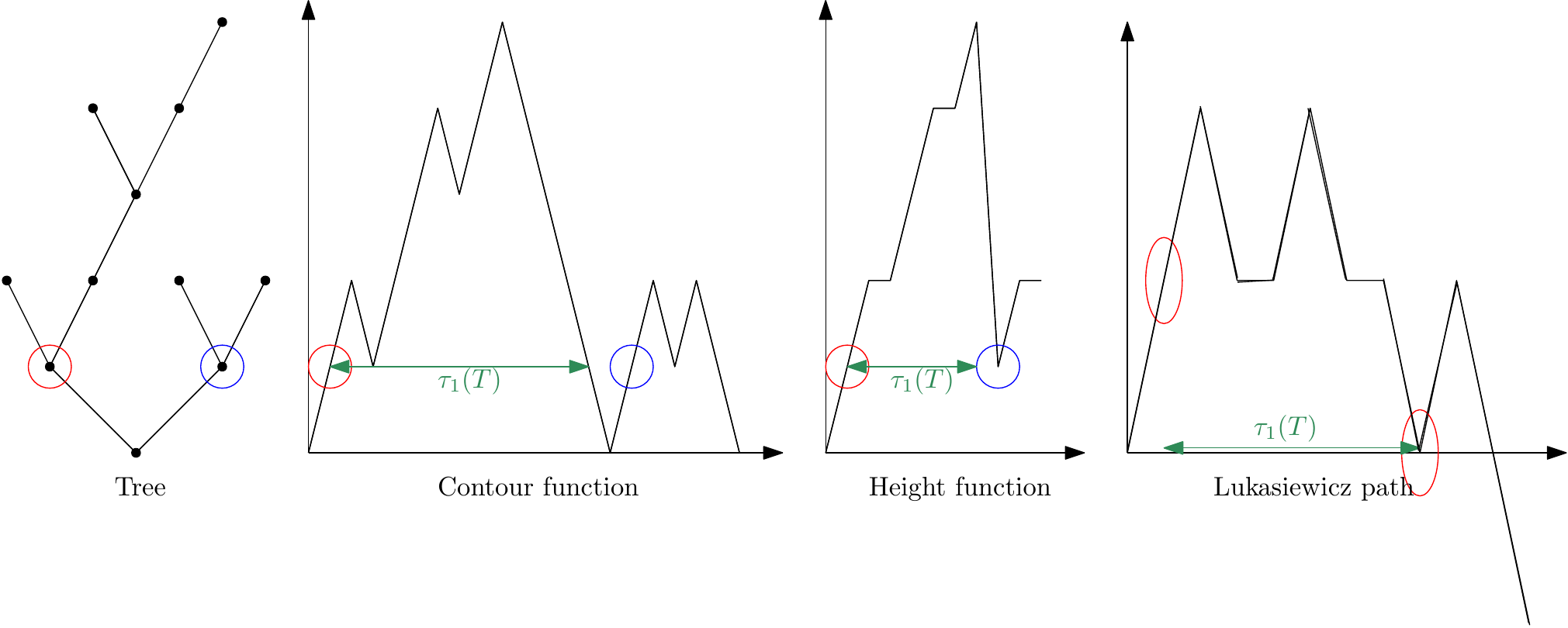}
\centering
\caption{Example of contour function, height function and Lukasiewicz path for the given tree.}\label{fig:contourheightfns}
\end{figure}

These functions all uniquely define the tree $T$. This can be written particularly conveniently in the case of the contour function, since for any $s, t \in \{0, \ldots, 2(n-1)\}$, we can write the tree distance as a function on $\{0, \ldots, 2(n-1)\} \times \{0, \ldots, 2(n-1)\}$ by setting
\[
d^T(s,t) = C^T(s) + C^T(t) - 2\inf_{s \leq r \leq t} C^T(r).
\]

We will work mainly with the Lukasiewicz path $(W^{T}_m)_{0 \leq m \leq n}$ in this paper. It is not too hard to see that $W^{T}_m \geq 0$ for all $0 \leq m \leq n-1$, and $W^T_n = -1$. Moreover, the height function can be defined as a function of the Lukasiewicz path (see \cite[Equation (1)]{LeGDuqMono}) by setting
\begin{equation}\label{eqn:height Luk def}
H^T(m) = \Big|\Big\{ k \in \{0, 1, \ldots, m-1\}: W^T_k = \inf_{k \leq l \leq m} W^T_l \Big\}\Big|.
\end{equation}

\subsubsection{Multi-type Galton-Watson trees}\label{sctn:GW multi}
We will consider scaling limits of looptrees defined from both one and two-type Galton-Watson trees in Section \ref{sctn:LLT}. Accordingly, let $\xi, \xi_{\circ}$ and $\xi_{\bullet}$ be probability distributions on $\Z^{\geq 0}$.

\begin{definition} A Galton-Watson tree with offspring distribution $\xi$ is a random plane tree $\mathcal{T}$ with law $\mathbb{P}_{\xi}$ satisfying the following properties.
\begin{enumerate}[(i)]
\item $\prstart{k_{\emptyset}=j}{\xi}=\xi(j)$ for all $j \in \Z^{\geq 0}$,
\item For every $j \geq 1$ with $\xi(j)>0$, the shifted trees $\theta_1(\mathcal{T}), \ldots, \theta_j(\mathcal{T})$ are independent under the conditional probability $\prcond{\cdot}{k_{\emptyset}=j}{\xi}$, with law $\bP_{\xi}$, where $\theta_i(\T) = \{v \in \mathcal{U}: iv \in \T \}$.
\end{enumerate}
\end{definition}
We say that $\T$ is \textit{critical} if $\E{\xi}=1$. Additionally, we say a random plane tree is an alternating two-type Galton-Watson tree with offspring distribution $(\xi_{\circ}, \xi_{\bullet})$ if all vertices at even (respectively odd) height have offspring distribution $\xi_{\circ}$ (respectively $\xi_{\bullet}$). We say that the tree is \textit{critical} if $\E{\xi_{\circ}}\E{\xi_{\bullet}}=1$.

\subsection{Stable trees}
We now introduce stable trees. These are closely related to stable looptrees, and were introduced by Le Gall and Le Jan in \cite{LeGLeJanExploration} then further developed by Duquesne and Le Gall in \cite{LeGDuqMono,DuqLeGPFALT}. For $\alpha \in (1,2)$ we define the stable tree $\Ta$ from a spectrally positive $\alpha$-stable \Levy excursion, which plays the role of the Lukasiewicz path introduced above. By analogy with (\ref{eqn:height Luk def}), given such an excursion $\X$, we define the height function $\HH$ to be the continuous modification of the process satisfying
\begin{equation*}\label{eqn:height def}
\HH(t) = \lim_{\epsilon \rightarrow 0} \frac{1}{\epsilon} \int_0^t \mathbb{1} \{X^{\exc}_s < I_s^t + \epsilon \} ds,
\end{equation*}
where $I_s^t = \inf_{r \in [s,t]} \X_r$ for $s \leq t$, and the limit exists in probability (e.g. see \cite[Lemma 1.1.3]{LeGDuqMono}). We define a distance function on $[0,1]$ by
\[
d(s,t) = \HH(s) + \HH(t) - 2 \inf_{s \leq r \leq t} \HH(r),
\]
and an equivalence relation on $[0,1]$ by setting $s \sim t$ if and only if $d(s,t) = 0$. $\T_{\alpha}$ is the quotient space $([0,1]/ \sim, d)$, and we let $\pi$ denote the canonical projection from $[0,1]$ to $\Ta$. If $u, v \in \Ta$, we let $[[u,v]]$ denote the unique geodesic between $u$ and $v$ in $\Ta$.

This construction also provides a natural way to define a measure $\mu$ on $\T_{\alpha}$ as the image of Lebesgue measure on $[0,1]$ under the quotient operation.

Stable trees arise naturally as scaling limits of discrete plane trees with appropriate offspring distributions. More specifically, let $T_n$ be a discrete tree conditioned to have $n$ vertices and with critical offspring distribution $\xi$ in the domain of attraction of an $\alpha$-stable law, and such that $\xi$ is aperiodic. It is shown in \cite[Theorem 3.1]{DuqContourLimit} that
\begin{equation}\label{eqn:stable tree scaling limit def}
a_n n^{-1} T_n \rightarrow \T_{\alpha}
\end{equation}
in the Gromov-Hausdorff topology as $n \rightarrow \infty$, where $a_n$ is as defined in (\ref{eqn:dom of att def}).

\subsection{Random looptrees}\label{sctn:looptree def}
Discrete looptrees are best described by Figure \ref{fig:disc looptree intro} in the introduction. Moreover, as outlined there, stable looptrees can be defined as scaling limits of their discrete counterparts. That is, if $T_n$ is a Galton Watson tree conditioned to have $n$ vertices with critical offspring distribution $\xi$ in the domain of attraction of an $\alpha$-stable law, then
\[
a_n^{-1} {\Loop}(T_n) \overset{(d)}{\rightarrow} \L_{\alpha}
\]
with respect to the Gromov-Hausdorff topology as $n \rightarrow \infty$ \cite[Theorem 4.1]{RSLTCurKort}, where again in the case that $\xi([n, \infty)) \sim cn^{-\alpha}$, we can take $a_n = (c |\Gamma(-\alpha)|n)^{\frac{1}{\alpha}}$.

By comparison with (\ref{eqn:stable tree scaling limit def}), $\L_{\alpha}$ can therefore be thought of as the looptree version of the \Levy tree $\T_{\alpha}$. We now explain how this intuition can be used to code $\La$ from a stable \Levy excursion, in such a way that $\La$ can be heuristically obtained from the corresponding stable tree $\T_{\alpha}$ by replacing each branch point by a loop with length proportional to the size of the branch point, gluing these loops together along the tree structure of $\T_{\alpha}$, and then taking the closure of the resulting metric space.

The following construction was introduced in \cite[Section 2.3]{RSLTCurKort}. The \Levy excursion itself plays the role of a continuum Lukasiewicz path. It was shown in \cite[Proposition 2]{MiermontSplittingNodes} that if we define the width of a branch point in $\T_{\alpha}$, coded by a jump at $t \in [0,1]$ of size $\Delta_t$, by 
\[
\lim_{\epsilon \downarrow 0} \frac{1}{\epsilon} \mu (\{v \in \Ta, d(\pi(t), v) \leq \epsilon\}),
\]
then the limit almost surely exists and is equal to $\Delta_t$. It is therefore natural that a jump of size $\Delta$ in $\X$ should code a loop of length $\Delta$ in $\La$.

Accordingly, using the notation of Section \ref{sctn:Useful results}, for every $t \in [0,1]$ with $\Delta_t > 0$, the authors in \cite[Section 2.3]{RSLTCurKort} equip the segment $[0, \Delta_t]$ with the pseudodistance
\begin{equation}\label{eqn:delta def}
\delta_t(a,b) = \min \{|a-b|,(\Delta_t - |a-b|) \}, \ \ \ \ \ \ \ \ \ \text{for} \ a, b \in [0, \Delta_t],
\end{equation}
and define a distance function on $[0,1]$ by first setting 
\begin{equation*}
d_0(s,t) = \sum_{s \prec u \preceq t} \delta_u(0, x_u^t)
\end{equation*}
whenever $s \preceq t$, and
\begin{equation} \label{eqn:d}
d(s,t) = \delta_{s \wedge t}(x_{s \wedge t}^s,x_{s \wedge t}^t) + d_0(s \wedge t, s) + d_0(s \wedge t, t)
\end{equation}
for arbitrary $s, t \in [0,1]$.

They show that $d$ as defined above is almost surely a continuous pseudodistance on $[0,1]$, and define an equivalence relation $\sim$ on $[0, 1]$ by setting $s \sim t$ if $d(s,t)=0$. They then define the stable looptree $\La$ as the quotient space
\[
\mathcal{L}_{\alpha} = ([0,1]/ \sim, d)
\]
in \cite[Definition 2.3]{RSLTCurKort}. We let $p:[0,1] \rightarrow \La$ denote the canonical projection under the quotient operation, and let $\nu$ denote the image of Lebesgue measure on $[0,1]$ under $p$. $\nu$ therefore denotes the natural analogue of uniform measure on $\La$.

In \cite{ArchBMCompactLooptrees}, we also define a resistance metric $R$ on stable looptrees. By analogy with the construction above, this is done by first replacing $\delta_t$ with the quantity $r_t$ defined by
\begin{equation}\label{eqn:r def}
r_t(a,b) = {\Big( \frac{1}{|a-b|}+\frac{1}{\Delta_t - |a-b|} \Big) }^{-1}= \frac{|a-b|(\Delta_t - |a-b|)}{\Delta_t}, \ \ \ \ \ \ \ \ \ \text{for} \ a, b \in [0, \Delta_t].
\end{equation}

Note that this corresponds to the effective resistance across two parallel edges of lengths $|a-b|$ and $\Delta_t - |a-b|$. For $s, t \in [0,1]$ with $s \preceq t$, we then set
\begin{equation}\label{eqn:R0}
R_0(s,t) = \sum_{s \prec u \preceq t} r_u(0, x_u^t).
\end{equation}
For arbitrary $s, t \in [0,1]$, we set 
\begin{equation} \label{eqn:R}
R(s,t) = r_{s \wedge t}(x_{s \wedge t}^s,x_{s \wedge t}^t) + R_0(s \wedge t, s) + R_0(s \wedge t, t).
\end{equation}

We show in \cite[Proposition 4.4]{ArchBMCompactLooptrees} that $R$ defined in this way is a resistance metric on $\La$ in the sense of Definition \ref{def:eff resistance metric}. Moreover, in \cite[Lemma 4.1]{ArchBMCompactLooptrees} we show that for any  $s,t \in [0,1]$, we have that $\frac{1}{2}d(s,t) \leq R(s,t) \leq d(s,t)$, and define the resistance looptree $\La^R$ (which we will often denote $(\La, R)$) as
\[
\mathcal{L}^R_{\alpha} = ([0,1]/ \sim, R).
\]

As a consequence, we also show in \cite[Corollary 4.2]{ArchBMCompactLooptrees} that the looptrees $(\La, d)$ and $(\La, R)$ are homeomorphic.

The construction above is such that a jump of size $\Delta$ corresponds naturally to a cycle of length $\Delta$ in $\La$, which we will call a ``loop".

A key result of \cite{RSLTCurKort} is a Gromov-Hausdorff invariance principle. We extended the result to include convergence of measures in \cite[Proposition 4.6]{ArchBMCompactLooptrees}. Moreover, the Gromov-Hausdorff convergence of \cite[Theorem 4.1]{RSLTCurKort} was originally stated with the geodesic metric $d$ in place of the resistance metric $R$, but equally holds for $R$. This results in the following proposition.

\begin{proposition}(cf \cite[Theorem 4.1]{RSLTCurKort}, \cite[Proposition 4.6]{ArchBMCompactLooptrees}).\label{thm:compact disc inv princ res}
Let $(\tau_n)_{n=1}^{\infty}$ be a sequence of trees with $|\tau_n| \rightarrow \infty$ and corresponding Lukasiewicz paths $(W^n)_{n = 1}^{\infty}$, and let $R_n$ denote the effective resistance metric on $\Loop(\tau_n)$ obtained via (\ref{eqn:resistance def variational}) by letting an edge between any two adjacent vertices have conductance $1$. Additionally let $\nu_n$ be the uniform measure that gives mass $1$ to each vertex of $\Loop(\tau_n)$, and let $\rho_n$ be the root of ${\Loop}(\tau_n)$, defined to be the vertex representing the edge joining the root of $\tau_n$ to its first child. Suppose that $(C_n)_{n=1}^{\infty}$ is a sequence of positive real numbers such that  
\begin{enumerate}[(i)]
\item $\Big( \frac{1}{C_n} W^n_{\lfloor |\tau_n| t \rfloor} (\tau_n) \Big)_{0 \leq t \leq 1} \overset{(d)}{\rightarrow} \X$ as $n \rightarrow \infty$,
\item $\frac{1}{C_n} \textsf{Height}(\tau_n) \ \overset{\mathbb{P}}{\rightarrow} \ 0$ as $n \rightarrow \infty$.
\end{enumerate}
Then
\[
\Big({\Loop}(\tau_n), \frac{1}{C_n}R_n, \frac{1}{|\tau_n|} \nu_n, \rho_n \Big) \overset{(d)}{\rightarrow} \Big( \La, R, \nu, \rho \Big)
\]
as $n \rightarrow \infty$ with respect to the Gromov-Hausdorff-Prohorov topology.
\end{proposition}

We now state a continuous version of this convergence. More generally, if $f$ is a function in $D^{\text{exc}}([0,\l])$ for some $\l \in (0, \infty)$, with only positive jumps, we can replace $\X$ with $f$ in the construction above to define the associated continuum looptree $\L_f$. Moreover, if $f_n$ is a sequence in $D^{\text{exc}}([0,\l])$ converging to $f$, also all with only positive jumps, then we can prove a similar invariance principle for the sequence of corresponding continuum looptrees.

There are minor differences in the assumptions required for the continuum convergence. In particular, note that the second condition of Proposition \ref{thm:compact disc inv princ res} that $\frac{1}{C_n} \textsf{Height}(\tau_n) \ {\rightarrow} \ 0$ in probability as $n \rightarrow \infty$ is important there because it ensures that in the limit, distances in the rescaled discrete looptrees come from the loop structure and not from distances in the corresponding tree. More formally, in the proof of \cite[Theorem 4.1]{RSLTCurKort} it is used to make a comparison between the expressions $\frac{1}{C_n} \sum_{u_n \preceq v_n} x_{u_n}^{v_n}$ and $\sum_{u \preceq v} x_u^v$ for the discrete and continuum trees respectively, where $x_{u_n}^{v_n}$ is the discrete analogue of $x_{u}^{v}$. For a sequence of trees $\tau_n$ with $\frac{1}{C_n} W^n \rightarrow f$ in the setting of Proposition \ref{thm:compact disc inv princ res}, we have for any $v_n \in \Loop (\tau_n)$ and $v \in \L_f$ that
\begin{align}\label{eqn:height correction exp}
\sum_{u_n \preceq v_n} x_{u_n}^{v_n} = \textsf{Height}(v_n) + W^n(v_n), \hspace{10mm}
\sum_{u \preceq v} x_{u}^{v} = f(v).
\end{align}
If $v$ and $v_n$ are in correspondence with each other, after being careful with left and right limits we can essentially apply the result that $\frac{1}{C_n} W^n(v_n) \rightarrow f(v)$ to deduce that the $\frac{1}{C_n} \sum_{u_n \preceq v_n} x_{u_n}^{v_n}$ also converges to $\sum_{u \preceq v} x_u^v$ in the limit to prove the invariance. To obtain this result, it is therefore crucial that the contribution from the rescaled height function goes to zero.

If, however, we replace the sequence of rescaled discrete looptrees with a sequence of continuum looptrees, say coded by the functions $(f_n)_{n=1}^{\infty}$ each with support $[0,1]$ and such that $f_n \rightarrow f$ in the Skorohod-$J_1$ topology as $n \rightarrow \infty$, then the height function won't appear in any of the new terms in (\ref{eqn:height correction exp}) and so the continuum analogue of condition $(ii)$ of Proposition \ref{thm:compact disc inv princ res} is not required for convergence of the corresponding looptrees.

In this sense, condition $(ii)$ reflects the fact the looptree $\Loop (\tau_n)$ isn't quite the same as the looptree $\L_{W^n}$. Condition $(ii)$ is precisely what is required to say that the difference between $\Loop (\tau_n)$ and $\L_{W^n}$ becomes negligible in the limit.

Hence, in the continuum, the same proof gives the following result.

\begin{proposition}\label{thm:compact cont inv princ}
Let $(f_n)_{n \geq 1}$ be a sequence in $D^{\exc}([0,1], \R^{\geq 0})$, and $f \in D^{\exc}([0,1], \R^{\geq 0})$ be such that $f_n \rightarrow f$ as $n \rightarrow \infty$ with respect to the Skorohod-$J_1$ topology. Additionally let $\nu$ and $\nu_n$ be the projections of Lebesgue measure via $p_f$ and $p_{f_n}$ onto the spaces $\L_f$ and $\L_{f_n}$ respectively. Then
\[
d_{\text{GHP}} \Bigg( \Big(\L_{f_n}, \d_n, \nu_n, \rho_n \Big), \Big( \mathcal{L}_f, \d_f, \nu_f, \rho_f \Big) \Bigg) \rightarrow 0
\]
as $n \rightarrow \infty$.

Here $\d$ can denote either the shortest-distance metric of \cite{RSLTCurKort}, or the resistance metric of (\ref{eqn:R}), but defined using the function $f$ in place of $\X$. Similarly for $\d_n$ and $f_n$.
\end{proposition}

The result follows exactly as in the proof of \cite[Theorem 4.1]{RSLTCurKort} by defining a correspondence between $\L_f$ and $\L_{f_n}$ to consist of all pairs $(t, \lambda_n (t))$, where $\lambda_n$ is the Skorohod homeomorphism that minimises the Skorohod distance between $f_n$ and $f$. The extension to include convergence of measures can be obtained exactly as in \cite[Proposition 4.6]{ArchBMCompactLooptrees}.

Clearly the result of the proposition will hold for functions defined on any compact time interval, not just $[0,1]$. We will use this in Section \ref{sctn:LLT} to prove Theorem \ref{thm:LLT}. Moreover, by extending the coding functions to be constant beyond endpoints where necessary, the result also holds providing the supports of the functions $f_n$ converge to that of $f$.

At some points in this paper, we will refer to the ``corresponding" or ``underlying" stable tree of $\La$, by which we mean the stable tree $\Ta$ coded by the same excursion that codes $\La$. We let $\La$ denote a compact stable looptree conditioned on $\nu(\La)=1$, but at various points we will let $\tilde{\La}$ denote a generic stable looptree coded by an excursion under the \Ito measure but without any conditioning on its total mass. We will also let $\La^{1}$ denote a stable looptree but conditioned so that its underlying tree has height $1$. However, we will make this notation explicit at the time of writing.

The height of a stable tree $\tilde{\Ta}$ is defined as $H_{\text{max}} = \sup_{u \in \tilde{\Ta}} d_{\tilde{\Ta}}(\rho, u)$. As the height process is almost surely continuous, this maximum is almost surely realised by at least one $u \in \tilde{\Ta}$. Moreover, we see from \cite[Equation (23)]{DuqWangDiameter} (and references therein) that there is almost surely a unique $u \in \tilde{\Ta}$ that attains this maximum, which we denote by $u_H$. If $\tilde{\La}$ is the corresponding stable looptree, we define two notions of its height:
\begin{enumerate}[(i)]
\item We define its $L^W$-Height to be the looptree distance from $\rho$ to $u_H$,
\item We define its $L$-Height to be $\sup_{u \in \tilde{\La}} d_{\tilde{\La}}(\rho, u)$.
\item We define its $L^m$-Height to be $\max \tilde{X}_s^{\text{exc}}$, where $\tilde{X}^{\text{exc}}$ is the \Levy excursion coding $\tilde{\La}$.
\end{enumerate}
In general, these are not the same. Note however that the $L^m$-Height is at least as big as the $L$-Height, since $\tilde{X}_s^{\text{exc}}$ gives the distance to the point in $\tilde{\La}$ represented by $s$ but going ``clockwise" around all loops. At times, we will also use the notation $T^W$-Height and $T^m$-Height to denote the length of the corresponding spine in the underlying tree, which we respectively denote by W-spine or m-spine.

\subsubsection{Uniform re-rooting invariance for stable trees and looptrees}
We will also use re-rooting invariance properties of stable trees and looptrees in our arguments. In particular, Duquesne and Le Gall proved in \cite[Proposition 4.8]{DuqLeGPFALT} that stable \Levy trees are invariant under re-rooting at a uniform point. Following on from this, they also proved the stronger result of invariance under re-rooting at a deterministic point $u \in [0,1]$ in \cite[Theorem 2.2]{DuqLeGRerooting}.

In \cite{HPWSpinPart}, the authors provide an alternative proof of uniform re-rooting invariance by considering a spinal decomposition of stable trees and using exchangeability properties of the resulting mass partition. This additionally allows them to show that stable trees are the only fragmentation trees for which this property holds. As a result, we obtain a similar uniform re-rooting invariance property for stable looptrees. This is stated precisely as \cite[Remark 4.6]{RSLTCurKort}, and the principles there show that looptrees are invariant under re-rooting at a uniform leaf, which is an equivalent statement in the limiting case.

We will exploit this in the proof of Theorem \ref{thm:LLT} where we will in fact prove the convergence result for compact stable looptrees rooted at a uniform point.

\subsubsection{Williams' decomposition of stable looptrees}\label{sctn:Williams Decomp}
The Williams' Decomposition for stable trees was given in \cite{AbDelWilliamsDecomp}. There, the authors show that if we define the W-spine of a stable \Levy tree $\tilde{\Ta}$ to be the unique path from its root to $u_H$, then $\tilde{\Ta}$ can be broken along this W-spine and that the resulting fragments form a collection of smaller \Levy trees. As a consequence, we immediately have a similar decomposition result for looptrees.

The Williams' Decomposition for stable trees given in \cite{AbDelWilliamsDecomp} encodes this  decomposition of $\tilde{\Ta}$ along its W-spine in a Poisson process. In the Brownian case of $\alpha = 2$, this corresponds to Williams' decomposition of Brownian motion.  Letting $H_{\text{max}}$ and $u_H$ be as above, we define the Williams' spine (or W-spine) of $\tilde{\Ta}$ to be the segment $[[\rho, u_H]]$, and define the Williams' loopspine (or W-loopspine) in the corresponding looptree $\tilde{\La}$ to be the closure of the set of loops coded by points in $[[\rho, u_H]]$. One of the main results of \cite{AbDelWilliamsDecomp} is a theorem which firstly gives the distribution of the loop lengths along the W-loopspine, and additionally the distribution of the fragments obtained by decomposing along it.

Given the spine from $\rho$ to $u_H$, and conditional on $\Hm = H$, the loops along the W-loopspine can be represented by a Poisson point measure $\sum_{j \in J} \delta (l_j, t_j, u_j)$ on $\R^+ \times [0, H] \times [0,1]$ with a certain intensity. A point $(l,t,u)$ corresponds to a loop of length $l$ in the W-loopspine, occurring on the W-spine at distance $t$ from the root in the underlying tree $\tilde{\Ta}$, and such that a proportion $u$ of the loop is on the ``left" of the W-loopspine, and a proportion $1-u$ is on the ``right". In \cite{AbDelWilliamsDecomp}, this is written in terms of the exploration process on $\tilde{\Ta}$, but we interpret their result below in the context of looptrees.

We note that when stating this result, we are not conditioning on the total mass of $\tilde{\Ta}$: only the maximal height. The mass of $\tilde{\Ta}$ will depend on its height via the joint laws for these under the \Ito excursion measure.

\begin{theorem}(Follows directly from \cite[Lemma 3.1 and Theorem 3.2]{AbDelWilliamsDecomp}).\label{thm:AbDel Williams Decomp}
\begin{enumerate}[(i)]
\item Conditionally on $H_{\text{max}}=H$, the set of loops in the W-loopspine forms a Poisson point process $\mu_{\textsf{W-loopspine}} = \sum_{j \in \mathcal{J}} \delta (l_j, t_j, u_j)$ on the W-spine in the underlying tree with intensity
\[
\mathbb{1}_{\{ [0,1] \}}(u) \mathbb{1}_{\{[0,H]\}}(t) l \exp \{ -l (H-t)^{\frac{-1}{\alpha - 1}} \} du \ dt \ \Pi (dl),
\]
where $\Pi$ is the underlying \Levy measure, with $\Pi(dl) = \frac{1}{|\Gamma(-\alpha)|} l^{-\alpha - 1} \mathbb{1}_{(0, \infty)}(l) dl$ in the stable case. We will denote the atom $\delta (l_j, t_j, u_j)$ by $\textsf{Loop}_j$.
\item Let $\delta (l, t, u)$ be an atom of the Poisson process described above. The set of sublooptrees grafted to the W-loopspine at a point on the corresponding loop can be described by a random measure $M^{(l)} = \sum_{i \in I} \delta^{(l)} (\EE_i, D_i)$, where $\EE_i$ is a \Levy excursion that codes a looptree in the usual way, and $D_i$ represents the distance going clockwise around the loop from the point at which this sublooptree is grafted to the loop, to the point in the loop that is closest to $\rho$. This measure has intensity
\begin{align*}
N( \cdot, \Hm \leq H-t) \times \mathbb{1}_{\{[0,l]\}} dD.
\end{align*}
In particular, since the sublooptrees are coded by the \Ito excursion measure, they are just rescaled copies of our usual normalised compact stable looptrees, and each of these is grafted to the loop on the W-loopspine at a uniform point around the loop lengths.
\end{enumerate}
\end{theorem}

\begin{remark}
Point $(ii)$ is a slight extension of the results of \cite{AbDelWilliamsDecomp} since the authors of that paper are only concerned with stable trees, and consequently are not interested in how the sublooptrees are distributed around each loop in the W-loopspine. Instead they write that the subtrees incident to the W-spine at the node corresponding to the atom $\delta(l,t,u)$ are described by a Poisson random measure with intensity $l N( \cdot, \Hm \leq H-t)$. In fact, in our proofs we will only be counting sublooptrees grafted to entire loops so the distribution of these around each individual loop will not matter. However, it should be clear from equation (11) and the paragraph following it in \cite{DuqLeGPFALT} that the sublooptrees are actually distributed uniformly around each loop.
\end{remark}


%
In Proposition \ref{prop:resistance results infinite}, we will have to decompose along the loopspine from the root to a point attaining the distance of the $L^m$-Height from the root. By analogy with the notation above, we will call this the $m$-loopspine, and the corresponding spine in the underlying tree the $m$-spine. We do not prove a specific distribution for the decomposition along this m-loopspine, but note that by similar principles to the Williams' case, the Poisson measure describing the loop lengths along the m-loopspine (analogous to that in Theorem \ref{thm:AbDel Williams Decomp}(i)) will have the form
\begin{equation*}
C_{\alpha} \mathbb{1}_{\{ [0,1] \}}(u) \mathbb{1}_{\{[0,H^m]\}}(t) l^{-\alpha} \textsf{pen}(l,H^m,t) du \ dt \ dl,
\end{equation*}
where $C_{\alpha} = \frac{\alpha (\alpha - 1)}{\Gamma (2-\alpha)}$, as before, $H^m = T^m\text{-Height}(\tilde{\La})$, and \textsf{pen} is a lower order penalty term. In particular, by considering only loops on incident on the first half of the m-spine, it can be bounded above and below by a constant. Moreover, the sublooptrees grafted to the m-loopspine will be coded by a thinned version of the \Ito excursion measure. This can be proved rigorously by applying Proposition \ref{prop:descent PP} for an unconditioned \Levy process and transferring to the excursion via the Vervaat transform (Theorem \ref{thm:Vervaat}) and absolute continuity relation (\ref{eqn:RN deriv levy bridge}).

\subsection{Infinite critical trees and looptrees}\label{sctn:infinite trees bckgrnd}
In this section we introduce Kesten's tree $T_{\infty}$ for a given critical offspring distribution $\xi$. In light of Theorem \ref{thm:Kesten LLT}, it is the natural way to construct such an infinite tree.

\begin{definition}\label{def:Kesten's tree}(\cite[Definition 2.9]{AbDelGWIntro}, adapted from \cite{KestenIICtree}).
Let $\xi$ be a critical offspring distribution, and define its size biased version $\xi^*$ by
\[
\xi^*(n) = {n \xi (n)}.
\]
The \textbf{Kesten's tree} $T_{\infty}$ associated to the probability distribution $\xi$ is a two-type Galton-Watson tree distributed as follows:
\begin{itemize}
\item Individuals are either normal or special.
\item The root of $T_{\infty}$ is special.
\item A normal individual produces only normal individuals according to $\xi$.
\item A special individual produces individuals according to the size-biased distribution $\xi^*$. Of these, one of them is chosen uniformly at random to be special, and the rest are normal.
\end{itemize}
\end{definition}
Almost surely, the special vertices form a unique infinite backbone of $T_{\infty}$. Note that this is one-ended. Aldous in \cite{AldousFringeSinTree} coined the term \textit{sin-trees} for such trees, since they have a single infinite spine.

The following local limit theorem was originally proved by Kesten in \cite{KestenIICtree} under a second moment condition, but was proved with the stated assumptions in \cite[Theorem 7.1]{JansonSurvey}, and demonstrates that this construction is the right one to take.

\begin{theorem}\label{thm:Kesten LLT}(\cite[Lemma 1.14]{KestenIICtree}, \cite[Theorem 2.1.1]{AbDelGWIntro}, \cite[Theorem 7.1]{JansonSurvey}).
Let $\xi$ be a critical offspring distribution with $\xi(0) + \xi(1) < 1$ and define $T_{\infty}$ as in Definition \ref{def:Kesten's tree}. Let $T_n$ be a Galton-Watson tree with offspring distribution $\xi$ conditioned on having height at least $n$. Then
\[
T_n \overset{(d)}{\rightarrow} T_{\infty}
\]
with respect to the Gromov-Hausdorff-vague topology as $n \rightarrow \infty$.
\end{theorem}

The convergence is actually stated in a stronger topology in the original literature, but we are mainly interested in Gromov-Hausdorff-vague convergence in this paper.

Kesten's construction has been imitated in the continuum by Duquesne in \cite{DuqSinTree}, who constructs continuum sin-trees and shows that these arise as the appropriate local limit of compact continuum trees conditioned on being large. By analogy with the compact continuum case, Duquesne's construction involves defining two height functions from two independent \Levy processes in the same way as done with the excursion in (\ref{eqn:height def}). These respectively code the tree structure on the left and right sides of the spine in the usual way.

The construction was further extended to infinite discrete looptrees in \cite{BjornStef}, where the authors define the infinite looptree associated with a critical offspring distribution $\xi$ to simply be $\textsf{Loop}'(T_{\infty})$, where $T_{\infty}$ is constructed as in Definition \ref{def:Kesten's tree}, and $\textsf{Loop}'$ is an operation very close to \textsf{Loop}, as defined in \cite[Section 4]{RSLTCurKort} and which we will introduce later in Section \ref{sctn:scaling lims}. This infinite looptree inherits the structure of having a loopspine with loop sizes determined by a size-biased version of $\xi$, to which usual compact discrete looptrees are grafted. The local limit theorem of Theorem \ref{thm:Kesten LLT} thus passes directly to the looptree case by continuity of the $\textsf{Loop}$ operation (see \cite[Corollary 2.3]{BjornStef}, the proof of which can easily be adapted to \textsf{Loop} rather than $\textsf{Loop}'$).

Finally, Kesten's construction of Definition \ref{def:Kesten's tree} was extended to critical multi-type Galton Watson trees in \cite[Theorem 3.1]{StephensonLocal} along with an analogous local limit theorem. Richier in \cite{RichierIICUIHPT} then used this to define an infinite two-type looptree and showed in \cite[Lemma 5.5]{RichierMapBoundaryLimit} that this arises as a similar local limit under appropriate conditions.

The concept of an infinite stable looptree has thus left a gap in the literature and the purpose of this paper is to fill that gap. The construction is the one suggested in \cite[Section 6]{RichierIICUIHPT} and extends the construction of infinite discrete looptrees in the same way that Duquesne's continuum sin-trees extend the construction of their discrete counterparts. The resulting local limit theorem allows us to prove various volume and heat kernel convergence results for compact stable looptrees in \cite{ArchBMCompactLooptrees}.

\section{Construction of infinite stable looptrees}\label{sctn:construction of infinite stable looptrees}
Our construction uses two stable \Levy processes to code each side of the loopspine, in place of the excursion. This is the approach suggested in \cite[Section 6]{RichierIICUIHPT} and our construction is merely the continuum version of the discrete construction of \cite[Section 3]{RichierIICUIHPT}, except that we have essentially turned this construction ``upside down" to match the original coding mechanism for compact looptrees.

We start by giving an equivalent construction of compact stable looptrees. We give the construction for a looptree of mass $\l$.

\begin{tcolorbox}[colback=white]
\textbf{Two-sided Construction of Compact Stable Looptrees}
\begin{enumerate}
\item Let $\Xbl$ be a spectrally positive, $\alpha$-stable \Levy bridge of lifetime $\l$. Let $m = m_{\l}$ be the (almost surely unique) time at which $\Xbl$ attains its infimum.
\item Let $(X^{(2, \l)}_t)_{t \geq 0}$ be the pre-infimum process, and $(X^{(1, \l)}_t)_{t \geq 0}$ be the time-reversed post-infimum process, extended to stay constant after times $m$ and $1-m$ respectively. That is,
\begin{align*}
X^{(2, \l)}_t = \begin{cases} \Xb_{t} \text{ for } t \in [0,m], \\ \Xb_{m} \text{ for } t > m; \end{cases}
\hspace{10mm}
X^{(1, \l)}_t = \begin{cases} \Xb_{\l-t} \text{ for } t \in [0, 1-m], \\  \Xb_{m} \text{ for } t > \l-m. \end{cases}
\end{align*}
\item Define a function ${X}^{\l}: \R \rightarrow \R$ by
\[
{X}^{\l}_t =
\begin{cases} X^{(2, \l)}_{t} & \text{ if } t \geq 0, \\
X^{(1, \l)}_{-t} & \text{ if } t < 0.
\end{cases}
\]
It should be clear from the Vervaat transform that ${X}^{\l}$ is just a shifted \Levy excursion.
\item For $s, t \in \R$, we define resistances $r^{\l}, R^{\l}_0$ and $R^{\l}$ from ${X}^{\l}$ exactly as in (\ref{eqn:r def}), (\ref{eqn:R0}) and (\ref{eqn:R}), but with the superscript $\l$ on all the quantities involved. We can similarly define distances $\delta^{\l}$, $d^{\l}_0$ and $d^{\l}$ exactly as in (\ref{eqn:d}). Analogously to the normalised case, we then set $\Lal = (\R / \sim, d^{\l})$, and ${\Lal}^R = (\R / \sim, R^{\l})$, and let $p^{\l}: \R \rightarrow \Lal$ denote the canonical projection.
\end{enumerate}
\end{tcolorbox}

Before giving the infinite construction, we give a brief outline of the strategy for proving Theorem \ref{thm:LLT}, which exploits uniform rerooting invariance of stable looptrees. By taking a stable looptree coded by an excursion $X^{\text{exc},\l}$ of length $\l$, and taking the root to be a uniform point in $U \in [0,\l]$, it follows from the Vervaat transformation that the processes $(X^{\text{exc},\l}_t)_{0 \leq t \leq U}$ and $(X^{\text{exc},\l}_t)_{U \leq t \leq \l}$ are distributed respectively as the post- and pre-minimum parts of a stable \Levy bridge. Standard convergence results then imply that on any compact interval, these converge in distribution to stable \Levy processes as $\l \rightarrow \infty$. Moreover, if we think of the loopspine as the sequence of loops coded by jump points at times $0 \preceq t \preceq U$, then $(X^{\text{exc},\l}_t)_{0 \leq t \leq U}$ codes for the loopspine along with everything grafted to the left hand side of it, and $(X^{\text{exc},\l}_t)_{U \leq t \leq \l}$ codes for everything grafted to the right hand side of it. It is therefore natural to replace each of these by unconditioned \Levy process in the infinite volume limit.

Due to the Vervaat transformation, this construction is entirely equivalent to the original construction of looptrees using the \Levy excursion, but we have now split the coding into two functions which define each side of the loopspine. To code the infinite looptree, we will take limits of each of these functions and use these to code each side of the infinite loopspine.

We first give the construction, and then prove Theorem \ref{thm:LLT} in Section \ref{sctn:LLT}.

\begin{tcolorbox}[colback=white]
\textbf{Construction of Infinite Stable Looptrees}
\begin{enumerate}
\item Let $X$ be an $\alpha$-stable, spectrally positive \Levy process, and let $X'$ be an $\alpha$-stable, spectrally negative \Levy process.
\item Define a function $\Xi: \R \rightarrow \R$ by
\[
\Xi_t =
\begin{cases} X_{t} & \text{ if } t \geq 0, \\
X'_{-t^-} & \text{ if } t<0.
\end{cases}
\]
\item Analogously to the compact construction above, if $t$ is a jump point of $\Xi$ with jump size $\Delta_t$ and $a, b \in [0, \Delta_t]$, set
\begin{align*}
{\delta}^{\infty}_t(a,b) &= \text{min} \{|a-b|, \Delta_t - |a-b| \}, \\
{r}^{\infty}_t(a,b) &= {\Big( \frac{1}{|a-b|}+\frac{1}{\Delta_t - |a-b|} \Big) }^{-1}= \frac{|a-b|(\Delta_t - |a-b|)}{\Delta_t}. 
\end{align*}

Additionally, as before, for $s, t \in \R$ with $s \leq t$ set ${I}^{\infty}_{s,t} = \inf_{r \in [s,t]} \Xi_r$, and ${x}^{\infty}_{s,t} = {I}^{\infty}_{s,t} - \Xi_{s^-}$. For $s, t \in \R$ we again write $s \prec t$ if $s \preceq t$ (meaning that ${x}^{\infty}_{s,t} \geq 0$) and $s \neq t$. Then, if $s \preceq t$ set
\begin{align*}
{d}^{\infty}_0(s,t) &= \sum_{s \prec u \preceq t} {\delta}^{\infty}_u(0, x_u^t), \\
{R}^{\infty}_0(s,t) &= \sum_{s \prec u \preceq t} {r}^{\infty}_u(0, x_u^t).
\end{align*}
Then, for general $s, t \in \R$, set 
\begin{align}\label{eqn:infinite d R def}
\begin{split}
{d}^{\infty}(s,t) &= {\delta}^{\infty}_{s \wedge t}(x^{\infty}_{s \wedge t,s},x^{\infty}_{s \wedge t,t}) + {d}^{\infty}_0(s \wedge t, s) + {d}^{\infty}_0(s \wedge t, t),\\
{R}^{\infty}(s,t) &= {r}^{\infty}_{s \wedge t}(x^{\infty}_{s \wedge t,s},x^{\infty}_{s \wedge t,t}) + {R}^{\infty}_0(s \wedge t, s) + {R}^{\infty}_0(s \wedge t, t).
\end{split}
\end{align}
Finally, define an equivalence relation $\sim$ on $\R$ by setting $s \sim t$ if and only if ${d}^{\infty}(s,t)=0$. We define the infinite looptrees $\L^{\infty}_{\alpha}$ and $\L^{\infty, R}_{\alpha}$ by
\begin{align*}
\L^{\infty}_{\alpha} &= (\R / \sim, {d}^{\infty}), \\
\L^{\infty, R}_{\alpha} &= (\R / \sim, {R}^{\infty}).
\end{align*}
\end{enumerate}
\end{tcolorbox}

For ease of notation and intuition, we will focus on
$\L^{\infty}_{\alpha}$ rather than $\L^{\infty, R}_{\alpha}$ in Sections \ref{sctn:vol bounds and spectral dim infinite} and \ref{sctn:LLT}, but the results will hold in the resistance setting by exactly the same arguments.

As in the compact case, we can define the projection $p^{\infty}: \R \rightarrow \La^{\infty}$, which is almost surely continuous, and endow $\La^{\infty}$ with the measure $\nu^{\infty}$ which is defined to be the pushforward of Lebesgue measure on the real line to $\La^{\infty}$ via $p^{\infty}$.

We also have the following proposition, as a direct consequence of the scale invariance of the stable \Levy process.

\begin{proposition}[Scale invariance of $\Lai$]\label{prop:Lai scale inv}
For any $c>0$,
\[
(\Lai, c\d, \rho^{\infty}, c^{\alpha} \nu^{\infty}) \overset{(d)}{=} (\Lai, \d, \rho^{\infty}, \nu^{\infty}),
\]
where $\d$ here can be equal to either $d^{\infty}$ or $R^{\infty}$.
\end{proposition}


We also record the following result, which arises as a direct consequence of Theorem \ref{thm:LLT}, \cite[Corollary 4.4]{RSLTCurKort} (which gives the same result in the compact case), and \cite[Theorem 8.1.9]{Burago} (which implies that this property is preserved in the limit).

\begin{corollary}\label{cor:length space}
Almost surely, $\Lai$ is a length space.
\end{corollary}

\section{Limit theorems}\label{sctn:LLT}
In this Section we prove Theorems \ref{thm:LLT} and \ref{thm:main scaling lim}, and other similar results.
\subsection{Proof of Theorem \ref{thm:LLT}}


Theorem \ref{thm:LLT} is proved by applying Proposition \ref{thm:compact cont inv princ} to the following convergence result. The \Levy processes are all normalised as in Section \ref{sctn:Levy background}. 

\begin{proposition}\label{prop:bridge conv}
Let $\Xbl$ be a spectrally positive, $\alpha$-stable \Levy bridge of lifetime $\l$, let $X$ be an $\alpha$-stable, spectrally positive \Levy process, and let $X'$ be an independent $\alpha$-stable, spectrally negative \Levy process. Also let $m_{\l}$ be the (almost surely unique) time at which $\Xbl$ attains its minimum. Then, for any $T_1, T_2>0$, letting $f$ and $g$ be any bounded continuous functions $D([0,T_i], \R) \rightarrow \R$, we have that
\begin{align*}
\E{ f\Big( (\Xbl_{t \wedge m_{\l}})_{t \in [0,T_1]} \Big) g\Big( (\Xbl_{((\l-t) \vee m_{\l})^-})_{t \in [0,T_2]}\Big)} &\rightarrow \E{f\Big((X_t)_{t \in [0,T_1]}\big)} \E{g\Big((X_t')_{t \in [0,T_2]}\Big)}
\end{align*}
as $\l \rightarrow \infty$.
\end{proposition}

Before we prove the proposition, we show how we can apply Proposition \ref{thm:compact cont inv princ} to the functions $X$ and $X'$ on compact time intervals to prove Theorem \ref{thm:LLT}.

\begin{proof}[Proof of Theorem \ref{thm:LLT}, assuming Proposition \ref{prop:bridge conv}]
We need to show that for Lebesgue almost every $r>0$, 
\begin{equation}\label{eqn:R ball LLT}
\mathcal{B}_r(\Lal) \overset{(d)}{\rightarrow} \mathcal{B}_r(\La^{\infty}).
\end{equation}
To this end, take some $r>0$. We define two times $t_g(r)$ and $t_d(r)$ by
\begin{align*}
t_g(r) = \inf \{ s \geq 0: \Delta_{-s} \geq 4r, \delta^{\infty}_{-s}(x^{\infty}_{-s,0}) \geq r \}, \hspace{10mm} t_d(r) = \inf \{s \geq 0: \Xi_s \leq \Xi_{{-t_g(r)}^-} \}.
\end{align*}
The purpose of defining $t_g(r)$ and $t_d(r)$ like this is that $X^{\infty}$ codes a compact looptree on the interval $[-t_g(r), t_d(r)]$, and that $\mathcal{B}_r(\La^{\infty})$ is contained in this.

Note that $t_g(r)$ is $\bPb$-almost surely finite, since letting $L_s$ denote the local time spent by $(\Xi_{-t^+})_{t \geq 0}$ at its infimum by time $s$, normalised so that $\E{e^{\lambda \Xi_{L^{-1}(t)}}} = e^{-\lambda^{\alpha - 1}t}$, we have from Proposition \ref{prop:descent PP} that the measure
\[
\sum_{s \in J} \delta_{(L_s, \Delta_s)}
\]
is a Poisson point measure of intensity $dl \cdot x \mathbb{1} \{ x^{-\alpha} \geq 4r \} dx$, where $J$ is the set $\{ s \geq 0: \Delta_{-s} \geq 4r, \delta^{\infty}_{-s}(x^{\infty}_{-s,0}) \geq r \}$. Moreover, by \cite[Chapter VIII, Lemma 1]{BertoinLevy} we know that $L^{-1}$ is a stable subordinator of parameter $1-\frac{1}{\alpha}$, and hence $L_t \rightarrow \infty$ $\bPb$-almost surely as $t \rightarrow \infty$. It follows that $t_g(r)$ is $\bPb$-almost surely finite for all $r>0$. Similarly, since $\liminf_{t \rightarrow \infty} \Xi_t = -\infty$ $\bPb$-almost surely, $t_d(r)$ is also $\bPb$-almost surely finite for all $r>0$.

For notational convenience, we write $t_g = t_g(r)$ and $t_d = t_d(r)$ from now on.

The compact looptree $\Lal$ is coded by an excursion $\Xl$ of length $\l$. To write this as a two-sided construction as described in the previous section, choose $U_{\l}$ uniform on $[0, \l]$, and define a function $\Xbl: [-U_{\l}, \l-U_{\l}]$ by 
\[
\Xbl_t = \Xl_{t+U_{\l}} - \Xl_{U_{\l}}
\]
for all $t \in [-U_{\l}, \l-U_{\l}]$. Then $\Xbl$ codes $\Lal$. Moreover, we can extend $\Xbl$ to $\R$ by taking it to be constant outside of $[-U_{\l}, \l-U_{\l}]$, and by Proposition \ref{prop:bridge conv}, it is then the case that $(X^{\br, \l}_t)_{t \in [-t_g - 1, t_d + 1]} \overset{(d)}{\rightarrow} (X^{\infty}_t)_{t \in [-t_g - 1, t_d + 1]}$.

Since the interval $[-t_d-1, t_g+1]$ is $\bPb$-almost surely compact, and the space of \cadlag functions with compact support endowed with the Skorohod-$J_1$ topology is separable, it follows by the Skorohod Representation Theorem and Proposition \ref{prop:bridge conv} that there exists a probability space $(\Omega, \F, \mathbb{P})$ on which $(X^{\br, \l}_t)_{t \in [-t_g - 1, t_d + 1]} \rightarrow (X^{\infty}_t)_{t \in [-t_g - 1, t_d + 1]}$ almost surely. We henceforth work in this space.

For each $\l>0$, let $\lambda_{\l}$ be the Skorohod homeomorphism (defined pointwise on $\Omega$) from $[-t_g - 1, t_d + 1] \rightarrow [-t_g - 1, t_d + 1]$ that minimises the Skorohod distance between these $X^{\br, \l}$ and $X^{\infty}$ on this interval. Then set $t_d^{\l} = \lambda_{\l}(t_d)$, and similarly $t_g^{\l} = \lambda_{\l}(t_g)$.

The correspondence consisting of all pairs $[t, \lambda_{\l}(t)]$ for $t \in [-t_g, t_d]$ is a subset of the correspondence used to minimise the Gromov-Hausdorff distance in the proof of Proposition \ref{thm:compact cont inv princ}, so letting $\L_{\alpha}^{\l,r} = p^{\l}((X^{\br, {\l}}_t)_{t \in [-t^{\l}_g, t^{\l}_d]})$ for each $\l>0$ and $\L^{\infty, r}_{\alpha} = p^{\infty}((X_t)_{t \in [-t_g, t_d]})$, it follows from Proposition \ref{thm:compact cont inv princ} that $d_{GHP} (\L_{\alpha}^{\l,r}, \L^{\infty, r}_{\alpha}) \rightarrow 0$ as $\l \rightarrow \infty$. Since $\mathcal{B}_r(\Lal) \subset \L_{\alpha}^{\l,r}$ and $\mathcal{B}_r(\La^{\infty}) \subset \L^{\infty, r}_{\alpha}$, it thus follows that $\mathcal{B}_r(\Lal) \overset{(d)}{\rightarrow} \mathcal{B}_{r'}(\La^{\infty})$ for Lebesgue almost every $r'<r$. By taking a countable sequence $r_n \rightarrow \infty$ we therefore deduce the result for Lebesgue almost-every $r>0$, and the theorem follows.
\end{proof}

We now conclude the proof of Theorem \ref{thm:LLT} by proving Proposition \ref{prop:bridge conv}.

\begin{proof}[Proof of Proposition \ref{prop:bridge conv}]
The key point is that the two sides of the bridge have a density with respect to the laws of $X$ and $X'$, in that for any $f, g$ as in the statement of the proposition, and any $\l > T_1 + T_2$, it follows from a minor modification of (\ref{eqn:RN deriv levy bridge}) that
\begin{align}\label{eqn:density bridge}
\begin{split}
&\E{ f\Big( (\Xbl_{t})_{t \in [0,T_1]} \Big) g\Big( (\Xbl_{(\l-t)^-})_{t \in [0,T_2]}\Big)} \\
&\hspace{2cm}= \E{f\Big((X_t)_{t \in [0,T_1]}\Big) g\Big((X_{t}')_{t \in [0,T_2]}\Big) \frac{p_{\l - T_1 - T_2}(X_{T_2^-}'  -X_{T_1})}{p_{\l}(0)}},
\end{split}
\end{align}
where $p_t(\cdot)$ here denotes the transition density of $X$.
The proof then essentially just uses the fact that $m_{\l}$ and $\l - m_{\l}$ tend to infinity in probability as $\l \rightarrow \infty$, and then the fact that with high probability, $X_{T_1}$ and $X'_{T_2}$ will also not be too large. There are two main steps. We first note that the quantity
\begin{align*}
\E{ f\Big( (\Xbl_{t \wedge m_{\l}})_{t \in [0,T_1]} \Big) g\Big( (\Xbl_{((\l-t) \vee m_{\l})^-})_{t \in [0,T_2]}\Big)} - \E{ f\Big( (\Xbl_{t})_{t \in [0,T_1]} \Big) g\Big( (\Xbl_{(\l-t)^-})_{t \in [0,T_2]}\Big)}
\end{align*}
is upper bounded by
\begin{align*}
&\ \ \ 2 ||f||_{\infty} ||g||_{\infty} \Bigg( \pr{m_1 < \frac{T_1}{\l}} + \pr{m_1 > 1 - \frac{T_2}{\l}} \Bigg),
\end{align*}
which converges to $0$ as $\l \rightarrow \infty$. This allows us to apply (\ref{eqn:density bridge}) as follows. First, note that it follows from the scaling relation $p_t(x) = t^{\frac{-1}{\alpha}}p_1(xt^{\frac{-1}{\alpha}})$ that
\begin{align*}
\frac{p_{\l - T_1 - T_2}(X_{T_2}'  -X_{T_1})}{p_{\l}(0)} = \Bigg( \frac{\l}{\l - T_1 - T_2}\Bigg)^{\frac{1}{\alpha}} \frac{p_1\big( (\l - T_1 - T_2)^{\frac{-1}{\alpha}}(X_{T_2}'  -X_{T_1}) \big) }{p_1(0)}.
\end{align*}
We denote this latter quantity by $p(\l, X, X', T_1, T_2)$, so that
\begin{align*}
\E{ f\Big( (\Xbl_{t})_{t \in [0,T_1]} \big) g\Big( (\Xbl_{(\l-t)^-})_{t \in [0,T_2]}\big)} - \E{f\Big((X_t)_{t \in [0,T_1]}\Big) g\Big((X_t')_{t \in [0,T_2]}\Big)} \hspace{2cm} & \\
=\E{f\Big((X_t)_{t \in [0,T_1]}\Big) g\Big((X_{t}')_{t \in [0,T_2]}\Big) \Bigg( p(\l, X, X', T_1, T_2) - 1 \Bigg)}&.
\end{align*}
Taking some $0 < \epsilon \ll \frac{1}{\alpha}$, we then decompose on the event $\{ |X_{T_1}| \vee |X_{T_2}'| \leq (\l - T_1 - T_2)^{\frac{1}{\alpha} - \epsilon}\}$ and its complement by writing the latter quantity as the sum
\begin{small}
\begin{align}\label{eqn:bridge exp two terms}
\begin{split}
&\E{f\Big((X_t)_{t \in [0,T_1]}\Big) g\Big((X_{t}')_{t \in [0,T_2]}\Big)\Bigg(p(\l, X, X', T_1, T_2) - 1 \Bigg)\mathbb{1} \{ |X_{T_1}| \vee |X_{T_2}'| \leq (\l - T_1 - T_2)^{\frac{1}{\alpha} - \epsilon}\}} \\
&+ \E{f\Big((X_t)_{t \in [0,T_1]}\Big) g\Big((X_{t}')_{t \in [0,T_2]}\Big)\Bigg( p(\l, X, X', T_1, T_2) - 1 \Bigg) \mathbb{1} \{ |X_{T_1}| \vee |X_{T_2}'| > (\l - T_1 - T_2)^{\frac{1}{\alpha} - \epsilon}\} }.
\end{split}
\end{align}
\end{small}
We deal with each of these two terms separately. For the first term, note that by continuity of the transition density \cite[Section VIII.1]{BertoinLevy}, 
\[
\sup_{|x| \leq 2(\l - T_1 - T_2)^{\frac{1}{\alpha} - \epsilon}} \Big\{ p_1\Big(x(\l - T_1 - T_2)^{\frac{-1}{\alpha}}\Big) \Big\} \rightarrow p_1(0)
\]
as $\l \rightarrow \infty$. We apply this by writing:
\begin{align*}
&\Big|\Big|\Bigg(p(\l, X, X', T_1, T_2) - 1 \Bigg)\mathbb{1} \{ |X_{T_1}| \vee |X_{T_2}'| \leq (\l - T_1 - T_2)^{\frac{1}{\alpha} - \epsilon}\} \Big|\Big|_{\infty} \\
&\hspace{1.8cm} \leq \frac{1}{p_1(0)} \Bigg( \Bigg| \Big( \Big( \frac{\l}{\l - T_1 - T_2}\Big)^{\frac{1}{\alpha}} - 1 \Big) \sup_{|x| \leq 2(\l - T_1 - T_2)^{\frac{1}{\alpha} - \epsilon}} \Big\{ p_1\Big({x}{(\l - T_1 - T_2)^{\frac{-1}{\alpha}}}\Big) \Big\} \Bigg| \\
&\hspace{4.8cm} \ \ \ \ + \Bigg| \sup_{|x| \leq 2(\l - T_1 - T_2)^{\frac{1}{\alpha} - \epsilon}} \Big\{ p_1\Big({x}{(\l - T_1 - T_2)^{\frac{-1}{\alpha}}}\Big) \Big\} - p_1(0) \Bigg| \Bigg),
\end{align*}
from which we deduce that the first term in (\ref{eqn:bridge exp two terms}) converges to zero as $\l \rightarrow \infty$, since $f$ and $g$ are also bounded. To deal with the second term, we upper bound it by
\begin{align*}
||f||_{\infty} ||g||_{\infty} \frac{1}{p_1(0)} ||p_1||_{\infty} \pr{|X_{T_1}| \vee |X'_{T_2}| > (\l - T_1 - T_2)^{\frac{1}{\alpha} - \epsilon}},
\end{align*}
which also vanishes as $\l \rightarrow \infty$.

It therefore follows by an application of the triangle inequality and the bounds above that

\begin{footnotesize}
\begin{align*}
&\E{ f\Big( (\Xbl_{t \wedge m_{\l}})_{t \in [0,T_1]} \big) g\Big( (\Xbl_{((\l-t) \vee m_{\l})^-})_{t \in [0,T_2]}\Big)} - \E{f\Big((X_t)_{t \in [0,T_1]}\big) g\Big((X'_{t})_{t \in [0,T_2]}\Big)}\\
&\hspace{5mm}\leq \E{ f\Big( (\Xbl_{t \wedge m_{\l}})_{t \in [0,T_1]} \Big) g\Big( (\Xbl_{((\l-t) \vee m_{\l})^-})_{t \in [0,T_2]}\Big)} - \E{ f\Big( (\Xbl_{t})_{t \in [0,T_1]} \Big) g\Big( (\Xbl_{(\l-t)^-})_{t \in [0,T_2]}\Big)} \\
&\hspace{8mm} + \E{ f\Big( (\Xbl_{t})_{t \in [0,T_1]} \big) g\Big( (\Xbl_{(\l-t)^-})_{t \in [0,T_2]}\big)} - \E{f\Big((X_t)_{t \in [0,T_1]}\Big) g\Big((X_t')_{t \in [0,T_2]}\Big)} \hspace{2cm} \\
&\hspace{5mm} \rightarrow 0
\end{align*}
\end{footnotesize}
as $\l \rightarrow \infty$, as claimed. We can then factorise the final term by independence of $X$ and $X'$.
\end{proof}

\subsection{Scaling limits of infinite discrete looptrees}\label{sctn:scaling lims}
In this section, we prove that infinite stable looptrees are scaling limits of infinite discrete looptrees. We start by proving the following proposition, from which Theorem \ref{thm:main scaling lim} will follow. Note the analogy with Proposition \ref{thm:compact disc inv princ res}, and \cite[Theorem 4.1]{RSLTCurKort}.

Given an infinite critical discrete tree $T_{\infty}$, we note that it can be coded by a two-sided Lukasiewicz path indexed by $\Z$ in the same way that an infinite critical continuum tree can be coded by a two-sided \Levy process.

As introduced in Section \ref{sctn:infinite trees bckgrnd}, the infinite discrete looptrees defined by Bj\"ornberg and Stef\'ansson in \cite{BjornStef} are formed by first taking a critical offspring distribution $\xi$ in the domain of attraction of an $\alpha$-stable law, and then forming Kesten's tree $\Tai$ as outlined in Section \ref{sctn:infinite trees bckgrnd}. This tree has a unique infinite spine of vertices with a size-biased version of the offspring distribution. The authors define their looptree as $\Loopp (\Tai)$. Here $\Loopp$ is an operation very similar to $\Loop$, obtained as in Figure \ref{fig:BSLoop}, and $d_{GH}(\Loop (\Tai), \Loopp (\Tai)) \leq 2$ (see \cite[Proof of Theorem 4.1]{RSLTCurKort}). We let $L_{\alpha}^{\infty, 1} = \Loopp(\Tai)$.

\begin{figure}[h]
\includegraphics[width=14cm, height=4cm]{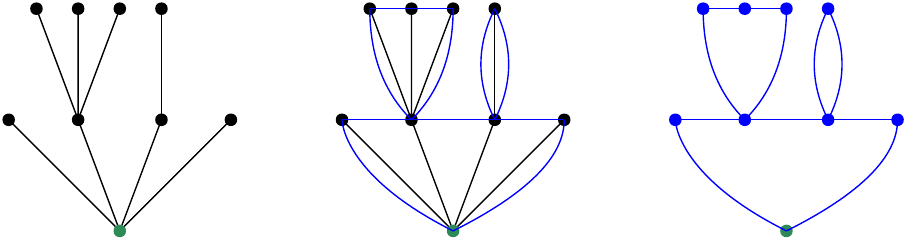}
\centering
\caption{A tree $T$ and $\Loop ' (T)$, for the same underlying tree as in Figure \ref{fig:disc looptree intro}.}\label{fig:BSLoop}
\end{figure}

\begin{remark}
In various places in other literature, the notation for $\Loop$ and $\Loopp$ is interchanged. We have used the notation of \cite{RSLTCurKort} since our paper follows on more naturally from the results there.
\end{remark}

We also make one further definition. Given an infinite critical tree $T_{\infty}$ and $R>0$, we define $\Loop (T_{\infty})^R$ to be the sublooptree of $\Loop (T_{\infty})$ obtained by letting $L$ be the first loop on the infinite loopspine that is of length greater than $4R$, and such that if we let $l_1$ and $l_2$ be the lengths of the two segments of this loop obtained by splitting the loop at the two points where it intersects its neighbouring loops in the infinite loopspine, we have that $\frac{l_1}{l_1 + l_2} \in [\frac{1}{4}, \frac{3}{4}]$. We then let $\Loop (T_{\infty})^R$ be the subset of $\Loop (T_{\infty})$ obtained by removing all descendants of all points in $L$ (but not removing $L$ itself). This definition is the discrete analogue to that of $\La^{\infty,R}$ given in the proof of Theorem \ref{thm:LLT}, and is useful since $\mathcal{B}_R(\Loop (T_{\infty})) \subset \Loop (T_{\infty})^R$, but $\Loop (T_{\infty})^R$ has the advantage of being a full looptree, whereas $\mathcal{B}_R(\Loop (T_{\infty}))$ may contain incomplete loops.

\begin{proposition}\label{prop:inf conv disc loop}
Let $(\tau_n)_{n=1}^{\infty}$ be a sequence of infinite critical trees (in the sense of Kesten) with corresponding two-sided Lukasiewicz paths $(W^n)_{n = 1}^{\infty}$, and let $\d_n$ denote either the shortest-distance or effective resistance metric on $\Loop(\tau_n)$. Additionally let $\nu_n$ be the measure that gives mass $1$ to each vertex in $\Loop(\tau_n)$, and let $\rho_n$ be the root of ${\Loop}(\tau_n)$, defined to be the vertex representing the edge joining the root of $\tau_n$ to its first child. Suppose that $(C_n)_{n=1}^{\infty}$ is a sequence of positive real numbers such that  
\begin{enumerate}[(i)]
\item For any compact interval $K \subset \R$, $\Big( \frac{1}{C_n} W^n_{\lfloor n t \rfloor} \Big)_{t \in K} \overset{(d)}{\rightarrow} (X^{\infty}_t)_{t \in K}$ as $n \rightarrow \infty$,
\item $\frac{1}{C_n} \textsf{Height}(\Tree (\Loop (\tau_n)^{rC_n})) \overset{\mathbb{P}}{\rightarrow} \ 0$ as $n \rightarrow \infty$, for all $r > 0$, where \Tree \ is the inverse operation of \Loop, and $\Loop (\tau_n)^R$ is defined above.
\end{enumerate}
Then
\[
\Big({\Loop}(\tau_n), \frac{1}{C_n}\d_n, \frac{1}{n}\nu_n, \rho_n \Big) \overset{(d)}{\rightarrow} \Big( \Lai, \d^{\infty}, \nu^{\infty}, \rho^{\infty} \Big)
\]
as $n \rightarrow \infty$ with respect to the Gromov-Hausdorff vague topology, where $\d^{\infty}$ can denote either the shortest-distance or effective resistance metric on $\Lai$, as appropriate. Moreover, the result also holds on replacing $\Loop$ by $\Loopp$ in all the statements above.
\end{proposition}
\begin{proof}
We start by proving the result for \Loop. We will prove the result with $\d=d$ and note that the corresponding result for $\d=R$ follows by the same arguments. The proof is again a consequence of Proposition \ref{thm:compact disc inv princ res}, given which, the proof is almost identical to the proof of Theorem \ref{thm:LLT} (i.e. by defining an increasing sequence of sublooptrees that exhaust the whole space, to each of which we then apply Proposition \ref{thm:compact disc inv princ res}), so we omit the details. As we did there, take $r>0$, and define two times $t_g(r)$ and $t_d(r)$ by
\begin{align*}
t_g(r) &= \inf \{ s \geq 0: \Delta_{-s} \geq 4r, \delta_{-s}(x_{-s}^0) \geq r \}, \\
t_d(r) &= \inf \{s \geq 0: {X}^{\infty}_s \leq {X}^{\infty}_{{-t_g(r)}^-} \}.
\end{align*}
It then follows by the Skorohod Representation Theorem that there exists a probability space $(\Omega, \mathcal{F}, \mathbb{P})$ upon which $(\frac{1}{C_n}W^n_{nt})_{-(t_g+1) \leq t \leq t_d+1} \rightarrow (X^{\infty})_{-(t_g+1) \leq t \leq t_d+1}$ almost surely with respect to the Skorohod-$J_1$ topology. As in the proof of Theorem \ref{thm:LLT}, for each $n \in \N$ let $\lambda_{n}$ be the Skorohod homeomorphism $[-t_g - 1, t_d + 1] \rightarrow [-t_g - 1, t_d + 1]$ that minimises the Skorohod-$J_1$ distance between these two functions, and set $t_d^{n} = \lambda_{n}(t_d)$, and similarly $t_g^{n} = \lambda_{n}(t_g)$.

By repeating the arguments of the proof of Theorem \ref{thm:LLT}, and noting that condition $(ii)$ above ensures that condition $(ii)$ of Proposition \ref{thm:compact disc inv princ res} is satisfied, we deduce that the looptrees coded by $(\frac{1}{C_n}W^n_{nt})_{-t^n_g \leq t \leq t^n_d}$ converge to the looptree coded by $(X^{\infty})_{t \geq 0}$. The result then follows as in the proof of Theorem \ref{thm:LLT}.

To prove the same result for \Loop$'$ in place of \Loop, note that since $d_{GH}(\Loop (\Tai), \Loopp (T^{\alpha}_\infty)) \leq 2$, the Gromov-Hausdorff convergence of Proposition \ref{thm:compact disc inv princ res} holds with $\Loop (\tau_n)$ replaced by $\Loopp (\tau_n)$, and the Prohorov convergence of measures of that proposition holds by the exactly the same arguments. As a consequence, we can just repeat exactly the same proof for \Loop$'$.
\end{proof}

In particular, the result applies taking $\tau_n = \Tai$ for all $n$, and $C_n = a_n$. In this case, $\frac{1}{C_n} \textsf{Height}(\Tree (\Loop (\tau_n) (rC_n)))$ will be of order $r^{\alpha - 1}n^{-\frac{2-\alpha}{\alpha}}L(n)$ for some slowly-varying function $L$, so point (ii) of Proposition \ref{prop:inf conv disc loop} holds by an appplication of Markov's inequality. We therefore deduce both Theorem \ref{thm:main scaling lim}, and Theorem \ref{thm:BS looptree conv} below, as a corollary.

\begin{theorem}\label{thm:BS looptree conv}
Take $\Loopp(\Tai)$ as above, with $\nu'$ the measure on $\Loopp(\Tai)$ such that $\nu'(x) = 1$ for all $x \in \Loopp(\Tai)$. Then
\[
(\Loopp(\Tai), a_n^{-1} \d, n^{-1} \nu', \rho) \overset{(d)}{\rightarrow} (\Lai, \d^{\infty}, \nu^{\infty}, \rho^{\infty})
\]
with respect to the Gromov-Hausdorff vague topology as $n \rightarrow \infty$. Here $\d$ (respectively $\d^{\infty}$) can denote either the geodesic metric $d$ (respectively $d^{\infty}$), or the effective resistance metric R (respectively $R^{\infty}$).
\end{theorem}

\subsubsection{Looptrees defined from two-type Galton Watson trees}
In practice in the context of random planar maps, it is often convenient to define discrete looptrees from alternating two-type Galton-Watson trees. In particular, Richier in \cite[Section 3]{RichierIICUIHPT} gives the following definition, illustrated in Figure \ref{fig:RichLoop}. Given an infinite alternating two-type Galton-Watson tree $T$ (as defined in Section \ref{sctn:GW multi}), say with white vertices at even height and black vertices at odd height, draw a loop around each black vertex by connecting its $i^{\text{th}}$ white child to its $(i+1)^{\text{th}}$ white child for all $i$, and join its parent to both its first and last white child. Then delete the black vertices and their incident edges; we denote the resulting structure by $\Loop^2(T)$.

\begin{figure}[h]
\includegraphics[width=13cm]{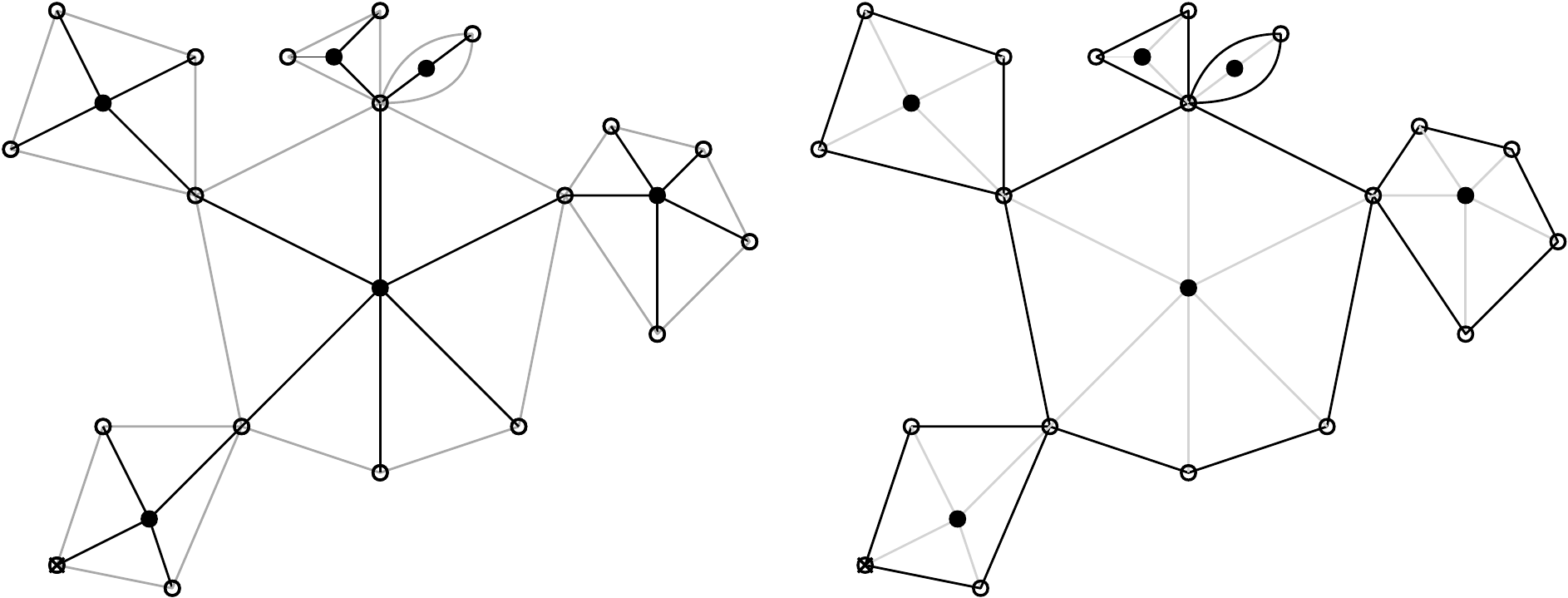}
\centering
\caption{A two-type tree and its looptree.}\label{fig:RichLoop}
\end{figure}

We now take a two-type tree $T_{\alpha}^{\infty, 2}$ with offspring distribution $(\xi_{\circ}, \xi_{\bullet})$ such that:
\begin{itemize}
\item $(\xi_{\circ}, \xi_{\bullet})$ is critical, i.e. $\E{\xi_{\circ}}\E{\xi_{\bullet}}=1$.
\item $\xi_{\circ}$ is shifted geometric with parameter $1-p \in (0,1)$, i.e. $\xi_{\circ}(k) = (1-p)p^k$ for all $k \geq 0$.
\item $\xi_{\bullet}$ is in the domain of attraction of an $\alpha$-stable law.
\end{itemize}

Before stating the scaling result, we briefly introduce two related concepts. One of these is the Janson-Stef\'ansson bijection of \cite{JanStefLargeFace}, which gives a bijection between alternating two-type Galton-Watson trees and one-type Galton-Watson trees. Given an alternating two-type Galton-Watson tree $T$, we denote its image under this bijection by $\Phi_{\text{JS}}(T)$. $\PhiJS$ has the same vertex set as $T$, but different edges, and is constructed as follows: for every white vertex that is not equal to the root, label its offspring as $u_1, \ldots, u_k$ in lexicographical order, and label its parent $u_0$. Then draw an edge joining $u_i$ to $u_{i+1}$ for each $i \in \{0, \ldots, k-1\}$, and draw an edge joining $u_k$ to $u$. See Figure \ref{fig:JS}.

The bijection is such that each white vertex in $T$ is therefore mapped to a leaf in $\PhiJS$, and each black vertex in $T$ with $k$ offspring is mapped to a vertex in $\PhiJS$ with $k+1$ offspring.

The second concept is a (final) related loop operation $\overline{\Loop}$. Given a (one-type) tree $T$, $\overline{\Loop}(T)$ is obtained by first forming $\Loopp(T)$, and then for each vertex $u \in \Loopp(T)$, contracting each edge joining $u$ to its rightmost child. $\overline{\Loop}(T)$ therefore has the property that multiple loops can be grafted at the same vertex, which is not the case with $\Loop(T)$ and $\Loopp(T)$ (but is the case with the two-type operation $\Loop^2$).

\begin{figure}[h]
\begin{subfigure}{.45\textwidth}
\includegraphics[width=6cm]{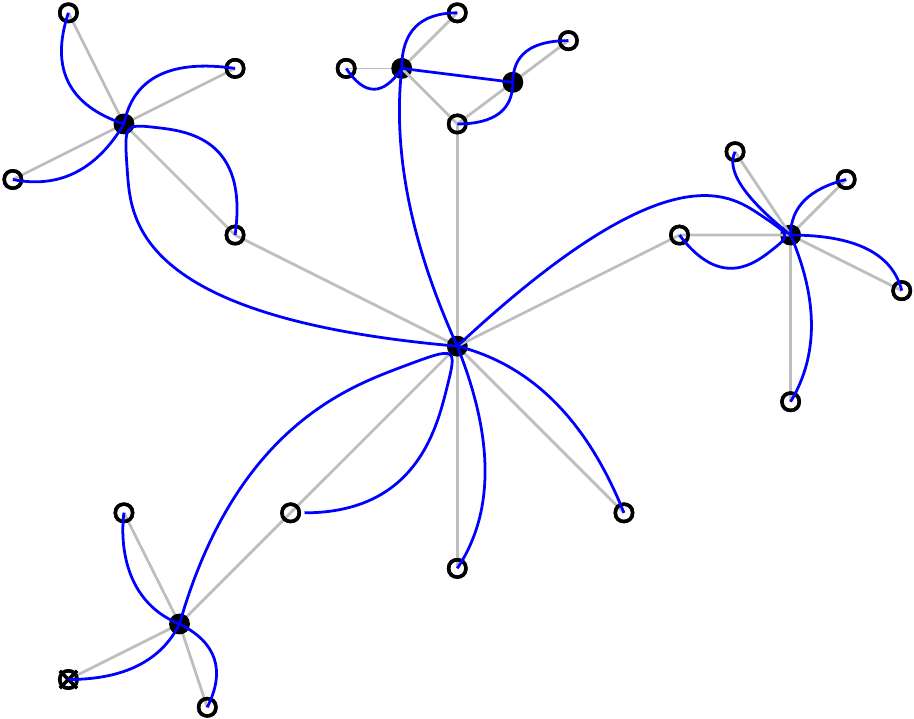}
\centering
\subcaption{$\Phi_{\text{JS}}(T)$.}
\end{subfigure}
\begin{subfigure}{.55\textwidth}
\includegraphics[width=7cm]{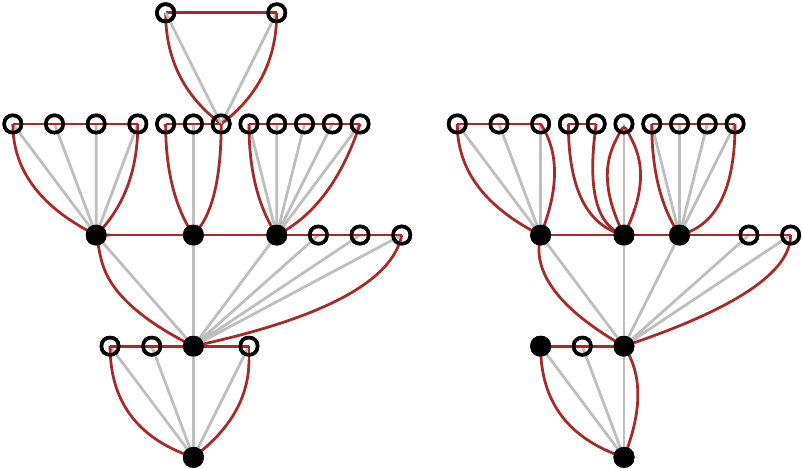}
\centering
\subcaption{$\Loopp(\Phi_{\text{JS}}(T))$ and $\overline{\Loop}(\Phi_{\text{JS}}(T))$.}
\end{subfigure}
\caption{Illustrations for the two-type tree $T$ in Figure \ref{fig:RichLoop}.}\label{fig:JS}
\end{figure}

The proof of the two-type scaling result then proceeds by applying the Janson-Stef\'ansson bijection to the two-type tree, and using the following facts, which we state without proof, but which should be plausible from looking at Figure \ref{fig:JS}.

\begin{enumerate}[(i)]
\item For any plane tree $T$ endowed with a measure giving mass $1$ to every vertex, $\dGHP (\Loopp(T), \overline{\Loop}(T)) \leq 4 \Height(T)$ (see \cite[Equation (48)]{RichierMapBoundaryLimit} for Gromov-Hausdorff version, then the Prohorov bound on measures follows by same reasoning).
\item If $T$ is an alternating two-type tree, then $\Loop^2(T) = \overline{\Loop}(\Phi_{\text{JS}}(T))$ ) (see \cite[Lemma 4.3]{CurKortUIPTPerc}).
\item Let $T$ be an alternating two-type Galton-Watson tree with offspring distributions $\xi_{\circ}$ and $\xi_{\bullet}$ such that $\xi_{\circ}$ is shifted geometric with parameter $1-p \in (0,1)$, i.e. $\xi_{\circ}(k) = (1-p)p^k$ for all $k \geq 0$, and $\E{\xi_{\circ}}\E{\xi_{\bullet}} \leq 1$. Then $\Phi_{\text{JS}}(T)$ is a one-type Galton-Watson tree with offspring distribution $\xi$, where $\xi$ is such that $\xi(0) = 1-p$ and $\xi(k) = p \xi_{\bullet}(k-1)$ for all $k \geq 1$ (see \cite[Appendix A]{JanStefLargeFace}).

Moreover, under the criticality assumption, this implies that
\begin{equation}\label{eqn:dom of att 2 type}
\frac{\sum_{i=1}^n \xi^{(i)} - n}{a_n} \overset{(d)}{\rightarrow} Z_{\alpha}
\hspace{1cm}
\text{if and only if} \hspace{1cm}
\frac{\sum_{i=1}^n \xi^{(i)}_{\bullet} - \frac{1-p}{p}n}{p^{\frac{-1}{\alpha}}a_n} \overset{(d)}{\rightarrow} Z_{\alpha}.
\end{equation}
\end{enumerate}

We are now ready to state and prove the convergence result.

\begin{theorem}\label{thm:Rich looptree conv}
Let $\Loop^2(T_{\alpha}^{\infty, 2})$ be above, with $(a_n)_{n \geq 1}$ as in (\ref{eqn:dom of att 2 type}), and let $\nu^2$ be the measure on $\Loop^2(T_{\alpha}^{\infty, 2})$ such that $\nu^2(x) = 1$ for all $x \in \Loop^2(T_{\alpha}^{\infty, 2})$. Then
\[
(\Loop^2(T_{\alpha}^{\infty, 2}), a_n^{-1} \d, n^{-1} \nu^2, \rho) \overset{(d)}{\rightarrow} (\Lai, \d^{\infty}, \nu^{\infty}, \rho^{\infty})
\]
with respect to the Gromov-Hausdorff vague topology as $n \rightarrow \infty$. Again, here $\d$ (respectively $\d^{\infty}$) can denote either the geodesic metric $d$ (respectively $d^{\infty}$), or the effective resistance metric R (respectively $R^{\infty}$).
\end{theorem}
\begin{proof}[Proof of Theorem \ref{thm:Rich looptree conv}]
Using the points above, we will show that there exists a probability space on which we can define both $T_{\alpha}^{\infty, 2}$ and a one-type Galton Watson tree $\tilde{T}_{\alpha}$ satisfying the assumptions of Proposition \ref{prop:inf conv disc loop} such that, for all $r>0$, 
\begin{equation}\label{eqn:two type loop comp r}
\dGHP (\mathcal{B}_r\big((\Loop^2 (T_{\alpha}^{\infty, 2}), a_n^{-1} \d, n^{-1} \nu', \rho)\big), \mathcal{B}_r\big((\Loopp (\tilde{T}_{\alpha}), a_n^{-1} \d, n^{-1} \nu', \rho)\big) \rightarrow 0
\end{equation}
almost surely as $n \rightarrow \infty$. As a result, we deduce that these two looptrees have the same Gromov-Hausdorff-Prohorov vague limit.

To do this, we first make a definition. As in the one-type case, it follows that $T_{\alpha}^{\infty, 2}$ almost surely has a unique infinite spine on which vertices instead have a size-biased offspring distribution (see \cite[Section 3.1]{StephLocalLim}). Analogously to previous definitions, for any $R>0$ we say that a loop on the corresponding loopspine is $R$-good if it has length at least $4R$ and if the two points at which it is connected to adjacent loops on the loopspine are separated by distance at least $R$. We then let $L_{\alpha}^{2}(R)$ denote the subspace obtained by taking the union of all the loops up to and including the first $R$-good loop on the loopspine, along with any sublooptrees grafted to them. The reason for this definition is that $\mathcal{B}_R(\Loop^2 (T_{\alpha}^{\infty, 2})) \subset L_{\alpha}^{2}(R)$, and $L_{\alpha}^{2}(R)$ is a full looptree (i.e. does not contain partial loops). We also let $T_{\alpha}^{2}(R)$ denote the (two-type) tree such that $\Loop^2(T_{\alpha}^{2}(R)) = L_{\alpha}^{2}(R)$ (this is well-defined since $\Loop^2$ is a bijection).

Set $\tilde{T}_{\alpha}^{r,n} = \Phi_{\text{JS}} (T_{\alpha}^{2}(ra_n))$. We make the following observations, based on the facts above.

\begin{enumerate}
\item By Fact (ii) above, $\overline{\Loop} \Big( \tilde{T}_{\alpha}^{r,n}\Big) = L_{\alpha}^2(ra_n)$.
\item By Fact (i) above, $\dGHP \Big(\overline{\Loop} \Big(\tilde{T}_{\alpha}^{r,n}\Big), \Loopp \Big( \tilde{T}_{\alpha}^{r,n}\Big) \Big) \leq 4 \Height \Big(\tilde{T}_{\alpha}^{r,n}\Big)$.

Moreover, $n^{\frac{-1}{\alpha}} \Height \Big(\tilde{T}_{\alpha}^{r,n}\Big) \rightarrow 0$ in probability as $n \rightarrow \infty$
since:
\begin{align*}
\prb{\Height \Big(\tilde{T}_{\alpha}^{r,n}\Big) \geq \epsilon n^{\frac{1}{\alpha}}+1} &\leq 
\prb{\Height \Big(T_{\alpha}^{2}(rn^{\frac{1}{\alpha}})\Big) \geq \epsilon n^{\frac{1}{\alpha}}+1} \\
&= (1-p_{r,n})^{\epsilon n^{\frac{1}{\alpha}}} \\
&\leq \exp \{-Cr^{-\alpha}n^{-\frac{\alpha - 1}{\alpha}} \epsilon n^{\frac{1}{\alpha}}\},
\end{align*}
where $p_{r,n} = \frac{1}{2} \pr{\hat{\xi}_{\bullet} \geq rn^{\frac{1}{\alpha}}} \sim Cr^{\alpha}n^{\frac{\alpha-1}{\alpha}}$ as $n \rightarrow \infty$ by assumption, since $\hat{\xi}_{\bullet}$ is a size-biased version of $\xi_{\bullet}$.
\item By construction and Fact (iii) above, $\mathcal{B}_r \big(\Loopp \big( \tilde{T}_{\alpha}^{r,n}\big)\big) = \mathcal{B}_r \big( \Loopp \big( \tilde{T}_{\alpha}\big)\big)$, where $\tilde{T}_{\alpha} = \lim_{n \rightarrow \infty}\tilde{T}_{\alpha}^{r,n}$ (the Janson-Stef\'ansson bijection is such that this is well-defined). Moreover, $\tilde{T}_{\alpha}$ is distributed as Kesten's critical tree with offspring distribution $\xi$.
%
\end{enumerate}
These three points imply that (\ref{eqn:two type loop comp r}) holds with $\tilde{T}_{\alpha}$ as in Point 3 above. Then, $\tilde{T}_{\alpha}$ satisfies the conditions of Proposition \ref{prop:inf conv disc loop} (in particular, condition (ii) of the Proposition holds by similar arguments to those in Point 2 above), so $(\Loopp (\tilde{T}_{\alpha}), a_n^{-1} \d, n^{-1} \nu', \rho) \overset{(d)}{\rightarrow} \Lai$ as $n \rightarrow \infty$. Since these $T_{\alpha}^{\infty, 2}$ and $\tilde{T}_{\alpha}$ are defined on a common probability space, (\ref{eqn:two type loop comp r}) therefore implies the same distributional result for $(\Loop^2 (T_{\alpha}^{\infty, 2}), a_n^{-1} \d, n^{-1} \nu', \rho)$.
\end{proof}

\begin{remark}
In \cite{RichierIICUIHPT}, these two-type looptrees are coded by upward skip-free random walks in a similar way to the one-type case. It is also possible to write an analogous result to Proposition \ref{prop:inf conv disc loop} in this case, under more general assumptions on the coding functions.
\end{remark}

\section{Volume bounds and resistance estimates for infinite stable looptrees}\label{sctn:vol bounds and spectral dim infinite}

In this section, we prove precise estimates on the volume and resistance growth properties of infinite stable looptrees. These are of interest in their own right but in Section \ref{sctn:RW consequences} we also use these to obtain bounds on the heat kernel, and use the resistance estimate to verify that the non-explosion conditions of Theorems \ref{thm:scaling lim RW resistance} and \ref{thm:scaling lim RW resistance annealed} are satisfied when we prove Theorems \ref{thm:main RW LLT conv} and \ref{thm:RW conv infinite intro}, along with their annealed counterparts.

In \cite[Section 5]{ArchBMCompactLooptrees}, we conduct a much more detailed study of the volume growth properties of compact stable looptrees, including proving similar results to those in Theorem \ref{thm:vol bounds} below. For this reason we will therefore skip some technical proof details when they are the same as in \cite{ArchBMCompactLooptrees}.

The full results are as follows. The result holds regardless of whether we define the balls in terms of $R^{\infty}$ or $d^{\infty}$, since the two metrics are equivalent. In particular, it is sufficient to prove the result for $d^{\infty}$ only, which is easier to handle. We do this below.

\begin{theorem}(cf \cite[Theorem 1.4]{ArchBMCompactLooptrees}).\label{thm:vol bounds}
$\mathbf{P}$-almost surely, we have:
\begin{align*}
&\limsup_{r \uparrow \infty} \Bigg( \frac{\nu^{\infty}(B^{\infty}(\rho^{\infty}, r))}{r^{\alpha} (\log \log r)^{\frac{4\alpha - 3}{{\alpha - 1}}}} \Bigg) < \infty, &&\limsup_{r \uparrow \infty} \Bigg( \frac{\nu^{\infty} (B^{\infty}(\rho^{\infty}, r))}{r^{\alpha} \log \log r} \Bigg) > 0, \\
&\liminf_{r \uparrow \infty} \Bigg( \frac{\nu^{\infty}(B^{\infty}(\rho^{\infty}, r))}{r^{\alpha}(\log \log r)^{-\alpha}} \Bigg) > 0, &&\liminf_{r \uparrow \infty} \Bigg( \frac{\nu^{\infty} (B^{\infty}(\rho^{\infty}, r))}{r^{\alpha} (\log \log r)^{-(\alpha - 1)}} \Bigg) < \infty.
\end{align*}
Moreover, $\mathbf{P}$-almost surely, for $\nu^{\infty}$-almost every $u \in \Lai$ we have
\begin{align*}
&\limsup_{r \downarrow 0} \Bigg( \frac{\nu^{\infty}(B^{\infty}(u, r))}{r^{\alpha} (\log \log r^{-1})^{\frac{4\alpha - 3}{{\alpha - 1}}}} \Bigg) < \infty, &&\limsup_{r \downarrow 0} \Bigg( \frac{\nu^{\infty} (B^{\infty}(u, r))}{r^{\alpha} \log \log r^{-1}} \Bigg) > 0, \\
&\liminf_{r \downarrow 0} \Bigg( \frac{\nu^{\infty}(B^{\infty}(u, r))}{r^{\alpha}(\log \log r^{-1})^{-\alpha}} \Bigg) > 0, &&\liminf_{r \downarrow 0} \Bigg( \frac{\nu^{\infty} (B^{\infty}(u, r))}{r^{\alpha} (\log \log r^{-1})^{-(\alpha - 1)}} \Bigg) < \infty.
\end{align*}
\end{theorem}

\begin{theorem}\label{thm:res bounds}
$\bPb$-almost surely, there exists a constant $c>0$ such that for all $r>0$,
\[
cr (\log \log (r \vee r^{-1}))^{\frac{-(3\alpha - 2)}{\alpha - 1}} \leq R^{\infty}(\rho^{\infty}, B^{\infty}(\rho^{\infty}, r)^c) \leq r.
\]
\end{theorem}

These results are obtained as a consequence of the following propositions.

\begin{proposition}\label{prop:vol results infinite}
There exist constants $c, c', C, C' \in (0, \infty)$ such that for all $r>0$, $\lambda >1$:
\begin{align*}
C\exp \{-c\lambda^{\frac{1}{\alpha - 1}}\} \leq \prb{ \nu^{\infty} (B^{\infty}(\rho^{\infty}, r)) < r^{\alpha} \lambda^{-1}} &\leq  C'\exp \{-c'\lambda^{\frac{1}{\alpha}}\} \\
C e^{-c \lambda} \leq \prb{ \nu^{\infty} (B^{\infty}(\rho^{\infty}, r)) \geq r^{\alpha} \lambda} &\leq C' \lambda^{\frac{\alpha - 1}{4\alpha - 3}} e^{-c'\lambda^{\frac{\alpha - 1}{4\alpha - 3}}}.
\end{align*}
\end{proposition}

\begin{proposition}\label{prop:resistance results infinite}
There exist constants $C, c \in (0, \infty)$ such that for all $r>0, \lambda > 1$:
\[
\prb{\Refi(\rho^{\infty}, B^{\infty}(\rho^{\infty}, r)^c) \leq r\lambda^{-1}} \leq 
Ce^{-c\lambda^{\frac{1}{4}}}.
\]
\end{proposition}

By applying Borel-Cantelli arguments along the sequence $r_n = 2^n$ (respectively $r_n=2^{-n}$) in Propositions \ref{prop:vol results infinite} and \ref{prop:resistance results infinite},  we obtain the results of Theorems \ref{thm:vol bounds} and \ref{thm:res bounds} for the regime $r \uparrow \infty$ (respectively $r \downarrow 0$). For any $R \in (0, \infty)$, the local results can then be extended to $\nu^{\infty}$-almost every $u \in \L_{\alpha}^{\infty,R}$ by uniform re-rooting invariance (recall that $(\L_{\alpha}^{\infty, R})_{R \geq 0}$ is a sequence of nested compact looptrees that exhaust $\Lai$). Taking $R \rightarrow \infty$ then gives the result.

Before outlining the proofs of Propositions \ref{prop:vol results infinite} and \ref{prop:resistance results infinite}, we briefly explain how the fractal structure of $\Lai$ can be encoded using the Ulam-Harris tree. This will be useful in the proofs of both propositions. This representation is very similar to the one described for compact looptrees in \cite[Section 5.2.1]{ArchBMCompactLooptrees}, except that at the first level we will decompose along the infinite loopspine rather than the W-loopspine.

\subsection{Encoding the looptree structure in a branching process}
The Williams' decomposition of Section \ref{sctn:Williams Decomp} suggests a natural way to encode the fractal structure of $\Lai$ in a branching process, which we will label using the Ulam-Harris numbering convention of Section \ref{sctn:trees background discrete}. Although the Williams' decomposition is defined along the maximal spine from the root of a compact tree, it follows from uniform rerooting invariance of stable trees that we can apply the same procedure from a uniform point instead, without changing the distribution of the decomposition.

Specifically, we let $\emptyset$ denote the root vertex of our branching process. This will represent the whole looptree $\Lai$ (in particular, $\emptyset$ should not be confused with $\rho^{\infty}$, which is the root of $\Lai$). We decompose $\Lai$ by removing the infinite loopspine, and denote the resulting fragments by $(\La^{(i,o)})_{i=1}^{\infty}$. Moreover, we let ${\La^{(i)}}$ denote the closure of $\La^{(i,o)}$ in $\Lai$, and remark that it follows from standard properties of the \Ito excursion measure that $\bPb$-almost surely, $\La^{(i)} = \La^{(i,o)} \cup \{\rho_i\}$ for each $i$. We call $\rho_i$ the root of $\La(i)$ as it is the point at which $\La^{(i)}$ is grafted to the infinite loopspine. It again follows from standard properties of the \Ito excursion measure that each fragment $\La^{(i)}$ is an independent (unconditioned) copy of a compact stable looptree, coded by an instance of the \Ito measure. We will view the set $(\La^{(i)})_{i=1}^{\infty}$ as the children of $\emptyset$ in our branching process, and we will index them by $\N$. Moreover, to each edge joining $\emptyset$ to one of its offspring $i$, we associate a random variable $m_i = m (\emptyset,i)$ which gives the mass of the sublooptree corresponding to index $i$.

We then repeat this decomposition along each of the sublooptrees $\La^{(i)}$, with the minor modification that we decompose along the W-loopspine rather than the infinite loopspine. More precisely, if $i$ is a child of $\emptyset$, we can decompose along its W-loopspine from its root to its point of maximal tree height to obtain a countable collection of fragments. By taking the appropriate closures, these fragments are sublooptrees and will form the offspring of $i$ in our branching process. We label the offspring as $(ij)_{j \geq 1}$. By repeating this procedure again and again on the resulting subsublooptrees, we can keep iterating to obtain an infinite branching process.

\begin{remark}
The spinal decomposition of \cite{HPWSpinPart} obtained by taking the loopspine to be from $p(U)$ (or the root) to an independent uniform point $p(V)$ is perhaps the most natural candidate to use as the basis of this iterative procedure, but when using this to bound the mass of small balls in $\La$ this leads to technical difficulties in the case when $V$ is chosen so that $p(V)$ is a point too close to $p(U)$. This difficulty is avoided by instead picking the maximal spine in the underlying tree.
\end{remark}

We index this process using the Ulam-Harris tree
\[
\mathcal{U} = \bigcup_{n=0}^{\infty}\N^{n}
\]
defined in Section \ref{sctn:trees background discrete}. Using the notation of \cite{Neveu}, an element of our branching process will be denoted by $u = u_1 u_2 u_3 \ldots u_j$, and corresponds to a sublooptree which we denote by $\La^{(u)} \subset \Lai$. Its offspring will all be of the form $(u i)_{i \in \N}$, where $ui$ here abbreviates the concatenation $u_1 u_2 u_3 \ldots u_j i$, and each will correspond to one of the further sublooptrees obtained on performing a Williams' decomposition of $\La^{(u)}$.

For each element $u \in \mathcal{U}$, we set
\[
M_u := \nu^{\infty} (\La^{(u)}),
\]
by viewing $\La^{(u)}$ as a subset of $\Lai$.

In the proofs of Propositions \ref{prop:vol results infinite} and \ref{prop:resistance results infinite}, we will select subtrees $T_{\text{vol}}, T_{\text{res}} \subset \mathcal{U}$ which index sublooptrees of large mass or large diameter. We make this more precise in the box below, where we describe the procedure used to obtain $T_{\text{vol}}$.

\subsection{Volume bounds}
To maintain consistency with the notation of \cite{ArchBMCompactLooptrees}, we take:
\begin{align*}
\bF = \frac{\alpha - 1}{4\alpha - 3}, \hspace{5mm} \bg = \frac{\alpha - 1}{4\alpha - 3}, \hspace{5mm} \bE = \frac{2\alpha - 1}{2\alpha (4\alpha - 3)}, \hspace{5mm} \bd = \frac{1}{4\alpha - 3}.
\end{align*}
The main point to remember is that $\beta_i \in (0,1)$ for all $i$. These precise values have been chosen to optimise the final exponent on $\lambda$, but are otherwise not important.

\begin{tcolorbox}[colback=white]
\textbf{Iterative Algorithm}\\

Start by taking $\emptyset$ to be the root of $T_{\text{vol}}$. Recall this represents the whole looptree $\Lai$.
\begin{enumerate}
\item Perform a decomposition of $\Lai$ along its infinite loopspine.
\item Consider the resulting fragments. To choose the offspring of $\emptyset$, select the fragments that have mass at least $r^{\alpha} \lambda^{1-\bF-\bg}$, and such that the roots of the corresponding sublooptrees are within distance $r$ of the root of $\emptyset$.
\item Repeat this process to construct $T_{\text{vol}}$ in the usual Galton-Watson way. Given an element $u = u_1 u_2 \ldots u_j \in T_{\text{vol}}$, there is a corresponding sublooptree $\La^{(u)}$ in $\Lai$ with root $\rho_u$ and $M_u \geq r^{\alpha} \lambda^{1-\bF-\bg}$. Consider the fragments obtained in a Williams' decomposition of $\La^{(u)}$, and select those that correspond to further sublooptrees that are within distance $r$ of $\rho_u$, and also such that $M_{u_1 u_2 \ldots u_j u_{j+1}} \geq r^{\alpha} \lambda^{1-\bF-\bg}$, to be the offspring of $u$.
\item For each $u = u_1 u_2 \ldots u_j \in T_{\text{vol}}$, set 
\[
S_u = \sum_{i=1}^{\infty} M_{ui} \mathbb{1} \Big\{ \rho_{ui} \in B (\rho_u, r) \Big\} \mathbb{1} \Big\{ M_{ui} < r^{\alpha} \lambda^{1-\bF-\bg} \Big\}.
\]
\end{enumerate}
\end{tcolorbox}

By the discussion above, this algorithm is $\bPb$-almost surely well defined, and is very similar to the decomposition of compact stable looptrees used in \cite[Section 5.2.2]{ArchBMCompactLooptrees}. As explained there, in the event that $T_{\text{vol}}$ is finite we then have that:
\begin{equation}\label{eqn:vol bound iter}
\nu^{\infty}(B^{\infty}(\rho^{\infty}, r)) \leq \sum_{u \in T_{\text{vol}}} S_u.
\end{equation}
Using this, we can now prove Theorem \ref{thm:vol bounds}. We skip some technical details since they are quite lengthy and can be carried out exactly as in the compact case, which is explained fully in \cite[Section 5]{ArchBMCompactLooptrees}, but comment on any necessary modifications for the infinite case.

\begin{proof}[Proof of Theorem \ref{thm:vol bounds}, outline only]
We start by proving the volume lower bounds, since the proof strategy is simpler than for the upper bounds. We use the \Levy coding mechanism of Section \ref{sctn:looptree def} and known fluctuation results for stable \Levy processes. It is not to hard to see (perhaps with the help of a picture, though this is proved formally in \cite[Lemma 2.1(ii)]{RSLTCurKort} in the compact case), that for any $[s,t]$ in $[0,1]$, 
\begin{equation}\label{eqn:osc dist relation}
d^{\infty}(s,t) \leq X^{\infty}_s + X^{\infty}_t -2\inf_{r \in [s,t]} X^{\infty}_r.
\end{equation}
Recall also from Section \ref{sctn:Useful results} that 
\[
\osc_{[a,b]} X^{\infty} := \sup_{s, t \in [a,b]} |X^{\infty}_t - X^{\infty}_s|.
\]
We deduce from (\ref{eqn:osc dist relation}) that if $\osc_{[0, r^{\alpha}\kappa]} X^{\infty} \leq \frac{1}{2}r$, then $B^{\infty}(\rho^{\infty}, r) \geq r^{\alpha}\kappa$. By applying the Vervaat transform and absolute continuity relation of (\ref{eqn:RN deriv levy bridge}), and taking either $\kappa = \lambda$, or $\kappa = \lambda^{-1}$, we are then able to use standard results for fluctuations of unconditioned \Levy processes to control the behaviour of \osc, and obtain the volume lower bounds. This is done rigorously in \cite[Sections 5.1 and 5.3]{ArchBMCompactLooptrees}. The only difference in the arguments used there is that in the compact place, we have to replace $X^{\infty}$ with $\X$ in (\ref{eqn:osc dist relation}). However, all the proofs of \cite{ArchBMCompactLooptrees} proceed by using the Vervaat transform and absolute continuity relation to compare $\X$ with an unconditioned \Levy process $X$. In the infinite case the proof is therefore simpler since we are already working with the unconditioned process.

The \Levy process picture is not so useful for proving precise volume upper bounds since the relation (\ref{eqn:osc dist relation}) is not an equality. In fact, the upper bound it gives on the distance is quite rough since any single jump in $\X$ contributes quite heavily to \osc, but does not immediately contribute to distances in looptrees. In particular, an entire jump corresponds to traversing an entire loop and therefore (initially) contributes zero overall distance in the looptree.

Set $p(\lambda) = \lambda^{\frac{\alpha - 1}{4\alpha - 3}} e^{-c'\lambda^{\frac{\alpha - 1}{4\alpha - 3}}}$. To obtain the volume upper bounds, or, more precisely, to show that $\prb{ \nu^{\infty} (B^{\infty}(\rho^{\infty}, r)) \geq r^{\alpha} \lambda} \leq p(\lambda)$, we therefore use the approach indicated by (\ref{eqn:vol bound iter}) above. The proof consists of two main steps:
\begin{enumerate}[(i)]
\item Bounding the progeny of $T_{\text{vol}}$;
\item Bounding each of the terms $(S_u)_{u \in T_{\text{vol}}}$.
\end{enumerate}

Again, these can be broken down into smaller steps. For $(i)$, we first show that the length of loopspine (or W-loopspine) contained in $B^{\infty}(\rho_u, r)$ is upper bounded by $r\lambda^{\bE}$ with probability at least $1-C'p(\lambda)$ (cf \cite[Lemma 5.5]{ArchBMCompactLooptrees}). Conditional on this, using the Poisson property of successive \Ito excursions, the number of offspring of $\La^{(u)}$ can essentially be stochastically dominated by a \textsf{Poisson}$(K_{\alpha}\lambda^{2\bE - \frac{1}{\alpha}({1-\bF-\bg})})$ random variable, where $K_{\alpha}$ is just a constant (cf \cite[Lemma 5.6]{ArchBMCompactLooptrees}). This is a subcritical offspring distribution, and by applying the main theorem of \cite{DwassProg} we deduce that, with probability at least $1-C'p(\lambda)$, $|T_{\text{vol}}| \leq \lambda^{\bF}$.

We now discuss a bound for a single term of the form $S_u$, as in point $(ii)$. We use the fact that the sum of the lifetimes of successive \Ito excursions (recall that these represent the volumes of successive sublooptrees arranged around the loopspine) can be represented as an $\alpha^{-1}$-stable subordinator with jump sizes corresponding to the original excursion lengths (e.g. see \cite[proof of Proposition 5.6]{GoldHaasExtinctionStable}), which we denote by $\textsf{Sub}$. In particular, since (as above) the relevant length of loopspine (or W-loopspine) contained in $B^{\infty}(\rho_u, r)$ is upper bounded by $r\lambda^{\bE}$, we can upper bound $S_u$ by $\textsf{Sub}_{r\lambda^{\bE}}$. Moreover, all jumps greater than $r\lambda^{1-\bF-\bg}$ have been removed from $S$ as a result of the construction of $T_{\text{vol}}$, which allows us to apply Lemma \ref{lem:osc} to deduce that, with probability at least $1-C'p(\lambda)$, for all $u \in T_{\text{vol}}$:
\[
S_u \leq \textsf{Sub}_{r\lambda^{\bE}} \leq r^{\alpha}\lambda^{1-\bF}.
\]
By taking a union bound and summing up, we therefore deduce that, with probability at least $1-C'p(\lambda)$,
\begin{equation*}
\nu^{\infty}(B^{\infty}(\rho^{\infty}, r)) \leq \sum_{u \in T_{\text{vol}}} S_u \leq |T_{\text{vol}}| \sup_{u \in T_{\text{vol}}} S_u \leq \lambda^{\bF}r^{\alpha}\lambda^{1-\bF} = r^{\alpha}\lambda.
\end{equation*}

The method to obtain the infimal volume upper bound is simpler and does not require reiterating around subsequent levels. We will say that a radius $r \in (0, \infty)$ is ``short" if the length of loopspine contained within $B^{\infty}(\rho^{\infty}, r)$ is at most $3r$. By scaling invariance of $\Lai$, the probability that $r$ is short is a (non-zero) constant that is independent of $r$ (or more usefully for an application of a generalised version of the second Borel-Cantelli Lemma, $\prcondb{r \text{ short}}{2r \text{ not short}}{}$ and $\prcondb{r \text{ short}}{\frac{1}{2}r \text{ not short}}{}$ are independent of $r$). On the event that $r$ is short, and using the same logic as above, we can bound the sum of the volumes of all the incident sublooptrees by $\textsf{Sub}_{3r}$, which is independent of the loopspine structure. Therefore, by repeating this argument along a subsequence $r_n \downarrow 0$ or $r_n \uparrow \infty$ of short radii, the infimal volumes will be upper bounded by the infimal behaviour of \textsf{Sub}, i.e. with fluctuations at least of order $(\log \log r^{-1})^{-(\alpha - 1)}$ as $r \downarrow 0$, and  $(\log \log r)^{-(\alpha - 1)}$ as $r \uparrow \infty$.
\end{proof}

\subsection{Applications to volume limits in compact stable looptrees}\label{sctn:infinite looptrees unit balls}
As a result of Theorem \ref{thm:LLT}, we are able to prove various volume convergence results that are exploited in \cite{ArchBMCompactLooptrees} to study Brownian motion on compact stable looptrees. The main applicable result is the following theorem. Here we let $\nu$ denote the intrinsic measure on a compact stable looptree $\La$ as defined in Section \ref{sctn:looptree def}, conditioned so that $\nu(\La)=1$. We also let $B(\rho, r)$ denote the open ball of radius $r$ around the root in $\La$, and $\bar{B}(\rho, r)$ its closure.

\begin{theorem}\label{thm:main vol conv}
There exists a random variable $(V_t)_{t \geq 0}: \Omega \rightarrow D([0, \infty), [0, \infty))$ such that the finite dimensional distributions of the process
\begin{align*}
\big( r^{-\alpha} \nu(\bar{B}(\rho, rt)) \big)_{t \geq 0}
\end{align*}
converge to those of $\big( V_t \big)_{t \geq 0}$ as $r \downarrow 0$, and $V_t$ denotes the volume of a closed ball of radius $t$ around the root in $\Lai$. Moreover, for any $p \in [1,\infty)$, setting $V:=V_1$ we have that $\Eb{V^p} < \infty$, and that
\[
r^{-\alpha p} \Eb{\nu(\bar{B}(\rho, r))^p} {\rightarrow} \Eb{V^p}
\]
as $r \downarrow 0$.
\begin{remark}
We have taken closed balls rather than open ones simply so that $V$ is c\`adl\`ag. We conjecture that the volume processes are in fact continuous, and that the convergence of the theorem can be extended to hold uniformly on compacts. However, due to the complex nature of looptrees, this is not straightforward to prove. In particular it is difficult to replicate the argument used to prove a similar result for stable trees, since looptrees do not have such a straightforward regeneration structure around the boundary of a ball of radius $r$.
\end{remark}
\begin{proof}

By the separability of Proposition \ref{thm:GH vague separable HB}, we can work on a probability space on which $\Lal \rightarrow \Lai$ almost surely as $\l \rightarrow \infty$. By standard results on metric space convergence, it follows that almost surely on this space, $\nu^{\l}(B^{\l}(\rho^{\l}, t)) \rightarrow \nu^{\infty}(B^{\infty}(\rho^{\infty}, t))$ for all $t$ such that $\nu^{\infty}(\partial B^{\infty}(\rho^{\infty}, t)) = 0$ (e.g. see \cite[Lemma 2.11]{GwynneMillerUIHPQscaling}), and therefore for Lebesgue almost every $t$. Moreover, by scaling invariance of $\Lai$, there are no ``special" values of $t$, so we deduce that for any fixed sequence $0<t_0<t_1 < \ldots < t_n < \infty$, the convergence almost surely holds simultaneously for all of the points $t_i, 0 \leq i \leq n$.

Since $(\nu^{\l}(B^{\l}(\rho, t)))_{t \geq 0} \overset{(d)}{=} ({\l} \nu B(\rho, {\l}^{\frac{-1}{\alpha}} t))_{t \geq 0}$, by writing $\l = r^{-\alpha}$ we therefore deduce the result as stated. In particular, it follows that $\nu^{\l} ({B^{\l}(\rho^{\l}, 1)}) \overset{(d)}{\rightarrow} V$ as $\l \rightarrow \infty$.

We claim that $V \in (0, \infty)$ almost surely, with all moments finite. This follows immediately from the exponential upper tails of Proposition \ref{prop:vol results infinite}, namely that
\begin{align}\label{eqn:V exp tails}
\prb{ V \geq \lambda} \leq C \lambda^{\frac{\alpha - 1}{4\alpha - 3}} e^{-c\lambda^{\frac{\alpha - 1}{4\alpha - 3}}}.
\end{align}
We now prove that the moments of $r^{-\alpha} \nu_1(B(\rho_1, r))$ converge to those of $V$. To see this, we observe that the arguments used to prove (\ref{eqn:V exp tails}) and the compact analogue in \cite[Proposition 5.4]{ArchBMCompactLooptrees} can be applied uniformly along the sequence $\Lal$ to give constants $c, C \in (0, \infty)$ such that 
\begin{align*}
\prbl{ \nu^{\l} (B^{\l}(\rho, r)) \geq r^{\alpha} \lambda} \leq C \lambda^{\frac{\alpha - 1}{4\alpha - 3}} e^{-c\lambda^{\frac{\alpha - 1}{4\alpha - 3}}}
\end{align*}
for all $\l \geq 1$. It follows that the sequence $(r^{-\alpha p} (\nu^{\l} (B^{\l}(\rho, r)))^p )_{\l \geq 1}$ is uniformly integrable
for all $p \geq 1$ and so setting $C_p = \Eb{V^p}$ we deduce that
\begin{align*}
r^{-\alpha p} \Eb{(\nu_1(B(\rho_1, r)))^p} \rightarrow C_p
\end{align*}
for all $p \geq 1$.
\end{proof}
\end{theorem}

\subsection{Resistance bounds}
We now turn to proving the resistance bounds. We use a version of the iterative procedure described above, which we again index by a subcritical branching process, to count the number of sublooptrees intersecting the boundary of a ball of radius $r$. More formally, we will define another subtree $T_{\text{res}} \subset \mathcal{U}$, but this time selecting sublooptrees of large diameter, rather than of large volume, to form the offspring at each step. Since this argument is not given in \cite{ArchBMCompactLooptrees}, we write it more carefully.

We first recall from Section \ref{sctn:looptree def} that the $L$-Height of a compact looptree $\tilde{\La}$ is given by $\sup_{u \in \tilde{\La}} d_{\tilde{\La}}(\rho, u)$, and the $L^m$-Height is given by $\max \tilde{X}^{\text{exc}}$. The $L^m$-Height is $\bPb$-almost surely realised by a unique point in $\tilde{\La}$, which we denote $u_m$. We refer to (the closure of) the set of loops coded by the ancestors of $u_m$ as the $m$-loopspine. As described in Section \ref{sctn:Williams Decomp}, the Poisson measure describing the loop lengths along the loopspine will have the form
\begin{equation}\label{eqn:m loopspine loop measure}
C_{\alpha} \mathbb{1}_{\{ [0,1] \}}(u) \mathbb{1}_{\{[0,H^m]\}}(t) l^{-\alpha} \textsf{pen}(l,H^m,t) du \ dt \ dl,
\end{equation}
where $C_{\alpha} = \frac{\alpha (\alpha - 1)}{\Gamma (2-\alpha)}$, as before, $H^m = T^m\text{-Height}(\tilde{\La})$, and \textsf{pen} is a penalty term that is bounded above and below by a constant on the first half of the m-spine. Moreover, the sublooptrees grafted to the m-loopspine will be coded by a thinned version of the \Ito excursion measure.

We now define some terminology, in keeping with that used in \cite[Section 5.2]{ArchBMCompactLooptrees} wherever possible.

Firstly, given $R>0$, we say that a loop on the m-loopspine is ``good" if it has length at least $4R$, and if the associated uniform random variable (that dictates the ratio of the two segments it splits into on either side of the loopspine) is in the interval $[\frac{1}{4}, \frac{3}{4}]$. We say the a loop is ``goodish" if it just has length at least $4R$. Additionally, for any $R>0$, and any (unconditioned) compact looptree $\tilde{\La}$ (respectively any infinite looptree $\Lai$), we let $I^m_R$ be the closure in $\tilde{\L_{\alpha}}$ (respectively $\Lai$) of the union of all the loops in the m-loopspine (respectively infinite loopspine) that intersect $\tilde{B}(\tilde{\rho}, R)$ (respectively $B^{\infty}(\rho^{\infty}, R)$). Additionally, we let $|I^m_R|$ be the sum of the lengths of these loops.

We start by giving a technical lemma, the proof of which may be skipped on a first reading.

\begin{lemma}(cf \cite[Lemma 5.5]{ArchBMCompactLooptrees}).\label{lem:segment length bound infinite}
For any $h>0, \lambda > 1, R< \lambda^{-1-\frac{h}{\alpha - 1}}$,
\begin{equation*}
\prcondb{|I^m_{R}| \geq 3R \lambda}{L^m\textup{-Height}(\tilde{\La}) \geq \frac{1}{2}}{} \leq Ce^{-c\lambda^{h \wedge 1}}.
\end{equation*}
\end{lemma}
\begin{proof}
We use a similar strategy to \cite[Lemma 5.5]{ArchBMCompactLooptrees}. Indeed, we first condition on existence of a good loop in the m-loopspine. We then select the closest good loop to $\rho$. Given such a loop, the number of goodish loops between $\rho$ and the first good loop is stochastically dominated by $N-1$, where $N$ is a Geometric($\frac{1}{2}$) random variable. $|I^m_R|$ can then be upper bounded by the random variable
\begin{equation}\label{eqn:loopspine length decomp}
2RN + \sum_{i=1}^{N} Q^{(i)},
\end{equation}
where $Q^{(i)}$ denotes the sum of the lengths of all the smaller loops on the m-loopspine that are between the $(i-1)^{\text{th}}$ and $i^{\text{th}}$ goodish loops, and the term $2RN$ comes from selecting a segment of length at most $R$ in each direction round each of the goodish loops. Each $Q^{(i)}$ can be independently approximated by an $(\alpha - 1)$-stable subordinator run up until an exponential time and conditioned not to have any jumps greater than $4R$.

Since we model the loop lengths by a subordinator indexed by the m-spine of the underlying tree, we upper bound the probability in question by:
\begin{align}\label{eqn:length height decomp}
\begin{split}
&\prcondb{|I^m_{R}| \geq 3R \lambda, T^m\text{-Height}(\tilde{\La}) \geq R^{\alpha -1}\lambda^{h}}{L^m\text{-Height}(\tilde{\La}) \geq \frac{1}{2}}{} \\
&+ \prcondb{|I^m_{R}| \geq 3R \lambda, T^m\text{-Height}(\tilde{\La}) \leq R^{\alpha -1}\lambda^{h}}{L^m\text{-Height}(\tilde{\La}) \geq \frac{1}{2}}{}.
\end{split}
\end{align}
The first of these terms can be upper bounded by $Ce^{-c\lambda}$ using exactly the same arguments as in \cite[Lemma 5.5]{ArchBMCompactLooptrees}, the point being that if the m-spine in the underlying tree is long enough, then there is plenty of time for a good loop to occur in the corresponding subordinator (though note that to do this formally, we have to deal with the penalty term of (\ref{eqn:m loopspine loop measure}), but this is minor and can be treated as in \cite[Lemma 5.5]{ArchBMCompactLooptrees}). To summarise more concretely:
\begin{itemize}
\item The number of good loops on the m-loopspine is stochastically dominated by a \textsf{Poisson}(c$\lambda^h$) random variable, so $\prb{\nexists \text{ a good loop }} \leq e^{-c \lambda^h}$.
\item $N$ is \textsf{Geometric}($\frac{1}{2}$), so $\prb{N \geq \lambda} \leq Ce^{-c\lambda}$.
\item $\prb{\sum_{i=1}^{N} Q^{(i)} \geq R \lambda} \leq C e^{-c \lambda}$. Indeed, by (\ref{eqn:m loopspine loop measure}), we can (independently for each $i$) stochastically dominate each term $Q^{(i)}$ by an $(\alpha - 1)$-stable subordinator $\textsf{Sub}^{(i)}$ with all jumps greater than $4R$ removed, run up until a time $\EE_R \sim \textsf{exp}(cR^{\frac{-1}{\alpha - 1}}$). We also let $\textsf{Sub}^{(i)'}$ denote a rescaled version of $\textsf{Sub}^{(i)}$, instead with all jumps greater than $4$ removed, and let $\EE \sim \textsf{exp}(c)$. By rescaling $\textsf{Sub}^{(i)}$ and choosing $\theta$ so that $\Eb{e^{\theta \textsf{Sub}^{(i)'}}} < \frac{3}{2}$ (which we can do by Lemma \ref{lem:osc}), we then have that
\begin{align}\label{eqn:Qi bound}
\begin{split}
\prb{\sum_{i=1}^{N} Q^{(i)} \geq R \lambda} &= \sum_{n=1}^{\infty}\prcondb{\sum_{i=1}^{N} \textsf{Sub}^{(i)'}_{\mathcal{E}} \geq \lambda}{N=n}{} \prb{N=n} \\
&\leq \sum_{n=1}^{\infty} \Big(\frac{3}{2}\Big)^n e^{-\theta \lambda} \Big(\frac{1}{2}\Big)^n \\
&= C_{\theta} e^{-\theta \lambda}.
\end{split}
\end{align}
\end{itemize}

This deals with the first term in (\ref{eqn:length height decomp}). If the m-spine is prohibitively short, then this logic cannot be applied, however we can remedy this by noting that if the $T^m$-Height is unusually small in relation to the $L^m$-Height, then this essentially forces the loop sizes to be large compared to what we would normally expect.

More concretely, in this case, let $M'$ be the total number of goodish loops on the m-loopspine (i.e. the total number of loops of length at least $4R$). Using the subordinator representation of the loop lengths, we then have that 
\begin{align*}
&\prcondb{M' \leq \lambda, T^m\text{-Height}(\tilde{\La}) \leq  R^{\alpha -1}\lambda^{h}}{L^m\text{-Height}(\tilde{\La}) \geq \frac{1}{2}}{} \\
&\hspace{20mm}\leq c\prb{M' \leq \lambda, L^m\text{-Height}(\tilde{\La}) \geq \frac{1}{2},T^m\text{-Height}(\tilde{\La}) \leq  R^{\alpha -1}\lambda^{h}} \\
&\hspace{20mm}\leq c\prcondb{\textsf{Sub}_{ R^{\alpha -1}\lambda^{h}} \geq \frac{1}{2} - 4R\lambda}{\text{no jumps of size at least } 4R}{},
\end{align*}
where the third line follows by removing any jumps corresponding to goodish loops from \textsf{Sub}, and \textsf{Sub} is a subordinator with (time-dependent) jump measure 
\[
C_{\alpha} \mathbb{1}_{\{ [0,1] \}}(u) \mathbb{1}_{\{[0,H^m]\}}(t) l^{-\alpha} \textsf{pen}(l,H^m,t) du \ dt \ dl,
\]
as in (\ref{eqn:m loopspine loop measure}). Note that $\textsf{Sub}$ is almost an $(\alpha - 1)$-stable subordinator, but with the extra penalty against larger jumps. We therefore let $\textsf{Sub}^{\alpha - 1}$ denote an $(\alpha - 1)$-stable subordinator. It follows that for any $k>0$, and any $t, x, y > 0$:
\begin{align*}
\prcondb{\textsf{Sub}_{t} \geq x}{\text{no jumps of size at least } y}{} &\leq \prcondb{\textsf{Sub}^{\alpha - 1}_{t} \geq x}{\text{no jumps of size at least } y}{} \\
&= \prcondb{\textsf{Sub}^{\alpha-1}_{k^{\alpha - 1}t} \geq kx}{\text{no jumps of size at least } ky}{}.
\end{align*}
Taking $k=R^{-1}\lambda^{\frac{-h}{\alpha-1}}$, we therefore see that 
\begin{align*}
&\prcondb{M' \leq \lambda, T^m\text{-Height}(\tilde{\La}) \leq  R^{\alpha -1}\lambda^{h}}{L^m\text{-Height}(\tilde{\La}) \geq \frac{1}{2}}{} \\
&\hspace{20mm}\leq \prcondb{\textsf{Sub}^{\alpha-1}_{1} \geq \frac{1}{2} R^{-1}\lambda^{\frac{-h}{\alpha-1}} - \lambda^{1 - \frac{h}{\alpha-1}} }{\text{no jumps at least } 4\lambda^{-\frac{h}{\alpha-1}}}{} \\
&\hspace{20mm}\leq \Eb{e^{\theta \textsf{Sub}^{\alpha-1}_1}} e^{-\theta \lambda}
\end{align*}
for sufficiently small $\theta > 0$, where the existence of the exponential moment in the last line follows from Remark \ref{rmk:osc deterministic}, and we recall that $R< \lambda^{-1-\frac{h}{\alpha - 1}}$ by assumption.

We can then proceed exactly as in the second and third bullet points above to deduce that the second term in (\ref{eqn:length height decomp}) is upper bounded by $Ce^{-c\lambda}$. This completes the proof.
\end{proof}

\begin{proof}[Proof of Proposition \ref{prop:resistance results infinite}]
By scaling invariance of $\Lai$, it is sufficient to prove the result for $r=1$.

Take $R=\lambda^{-2t}$, for some positive constant $t$ that will be specified later. The aim will be to bound the cardinality of a set $A \subset \Lai$ such that any path from $B^{\infty}(\rho^{\infty}, R)$ to $B^{\infty}(\rho^{\infty}, 1)^c$ must pass through at least one point in $A$. Do to this, we will define a tree $T_{\text{res}} \subset \mathcal{U}$, obtained similarly to $T_{\text{vol}}$ in the box above, but with two important differences:
\begin{itemize}
\item Rather than decomposing along the W-loopspine in the second and subsequent steps, we decompose along the m-loopspine.
\item Rather than reiterating around sublooptrees of larger mass, we reiterate around those with large $L$-Height: specifically, those that are grafted to the m-loopspine within distance $R$ of the root, and with $L$-Height at least $\frac{1}{2}$. We decompose along the m-loopspine rather than the loopspine to the point achieving the $L$-Height purely because it is more convenient to write down an expression of the form (\ref{eqn:m loopspine loop measure}) in this case. However, an expression of the form of (\ref{eqn:m loopspine loop measure}) should also be true in the case of this loopspine.
\end{itemize}

We will show that, with sufficiently high probability, the total progeny of $T_{\text{res}}$ is at most $\frac{1}{2}\lambda^{t}$, and that, on this event, we can pick a set $A$ of cardinality at most $\lambda^{2t}$. In this case we are done: since $A$ is a cutset, we then have that
\begin{equation}\label{eqn:Reff decomp}
\Refi(\rho^{\infty}, B^{\infty}(\rho^{\infty}, 1)^c) \geq \Refi(\rho^{\infty}, A),
\end{equation}
and due to the underlying tree structure this latter quantity is lower bounded by the resistance of $2|A|$ edges connected in parallel, each of resistance $\lambda^{-2t}$. More precisely:
\[
\Refi(\rho^{\infty}, A) \geq (|A| \lambda^{2t})^{-1} \geq \frac{1}{2}\lambda^{-4t}.
\]
We will then optimise over $t$ to obtain the result.

To this end, we now turn to bounding $|T_{\text{res}}|$. As commented under (\ref{eqn:m loopspine loop measure}), the sequence of sublooptrees incident to the m-loopspine at a point in $I^m_{R}$ can be stochastically dominated by those coded by the classical (unthinned) \Ito excursion measure along this segment, so the offspring distribution of a particular $u \in T_{\text{res}}$ will be \textsf{Poisson}($\tilde{C}|I^{m,u}_{R}|$), where $\tilde{C} = N(L^m \text{-Height} \geq \frac{1}{2})$, and we have added an extra superscript $u$ to denote the dependence on $u$. By applying Lemma \ref{lem:segment length bound infinite} with $h = (\alpha - 1)(2t-1)$, it then follows exactly as in \cite[Lemma 5.7]{ArchBMCompactLooptrees} that
\begin{align*}
\prb{|T_{\text{res}}| \geq \lambda^t} &\leq \lambda^t \prcondb{|I^m_{R}| \geq R \lambda^{t}}{L^m\text{-Height}(\tilde{\La}) \geq \frac{1}{2}}{} + \prb{|\hat{T}| \geq \frac{1}{2}\lambda^{t}} \\
&\leq C \lambda^{t}Ce^{-c\lambda^{t(h \wedge 1)}} + Ce^{-c\lambda^t},
\end{align*}
where $\hat{T}$ is a Galton-Watson tree with \textsf{Poisson}($\tilde{C}\lambda^{-t}$) offspring distribution.

Assuming now that $|T_{\text{res}}| < \frac{1}{2}\lambda^{t}$, we claim that we can pick a set $A$ of cardinality at most $\lambda^{2t}$. In fact, rather than just assuming that $|T_{\text{res}}| < \frac{1}{2}\lambda^{t}$, we can assume that all of the events we conditioned on in order to construct the event $\{|T_{\text{res}}| < \frac{1}{2}\lambda^{t}\}$ do indeed occur. In particular, we can assume that:
\begin{enumerate}[(i)]
\item For each $u \in T_{\text{res}}$, letting $N_u$ be the number of goodish loops on the m-loopspine between $\rho_u$ and the first good loop, we have that $N_u < \lambda^{t}$.
\item For each $u \in T_{\text{res}}$, letting $Q^{(i)}_u$ denote the sum of the length of the shorter loops between successive goodish loops on the m-loopspine,
\[
\sum_{i=1}^{N_u} Q_u^{(i)} < R\lambda^t = \lambda^{-t}.
\]
\item $|T_{\text{res}}| < \frac{1}{2}\lambda^{t}$.
\end{enumerate}

Assuming this, we now describe how we select the set $A$. This is illustrated in Figure \ref{fig:loopspine cutset} below which represents the m-loopspine of some $u \in T_{\text{res}}$. In particular, on this m-loopspine, we can pick two points on each of the goodish loops, and two points on the first good loop, to be in $A$. Moreover, these points can be chosen so that they are within distance $R+\lambda^{-t}$ of the ``base point" of the loop (see Figure \ref{fig:loopspine cutset}). If one of the goodish loops violates the condition that the length of its shorter segment is less than $R$, we can instead treat it as the first good loop.

From the assumptions above, we deduce the following:
\begin{enumerate}[(i)$'$]
\item For all $u \in T_{\text{res}}$, the number of points of $A$ contained in $\La^{(u)}$ is at most $2N_u$ which by $(i)$ above is in turn at most $2 \lambda^{t}$.
\item $|A| \leq |T_{\text{res}}|2 \lambda^{t} =  \lambda^{2t}$.
\item Points in $A$ that are selected as points in the looptree corresponding to $u$ are within distance $|I_R^m| + \lambda^{-t}$ of $\rho_u$, i.e. distance $2\lambda^{-t}$ of $\rho_u$.
\item All points in $A$ are within distance $|T_{\text{res}}|\lambda^{-t} + \lambda^{-t}$ of $\rho^{\infty}$, which is at most $\frac{1}{2}$ by $(iii)$ above.
\item Therefore, any sublooptree grafted to the m-loopspine of $\La^{(u)}$ for some $u \in T_{\text{res}}$ that has $L$-Height less than $\frac{1}{2}$, will not intersect $B(\rho, 1)^c$. In other words, $A$ is really a cutset.
\end{enumerate}
%

\begin{figure}[h]
\includegraphics[width=14cm]{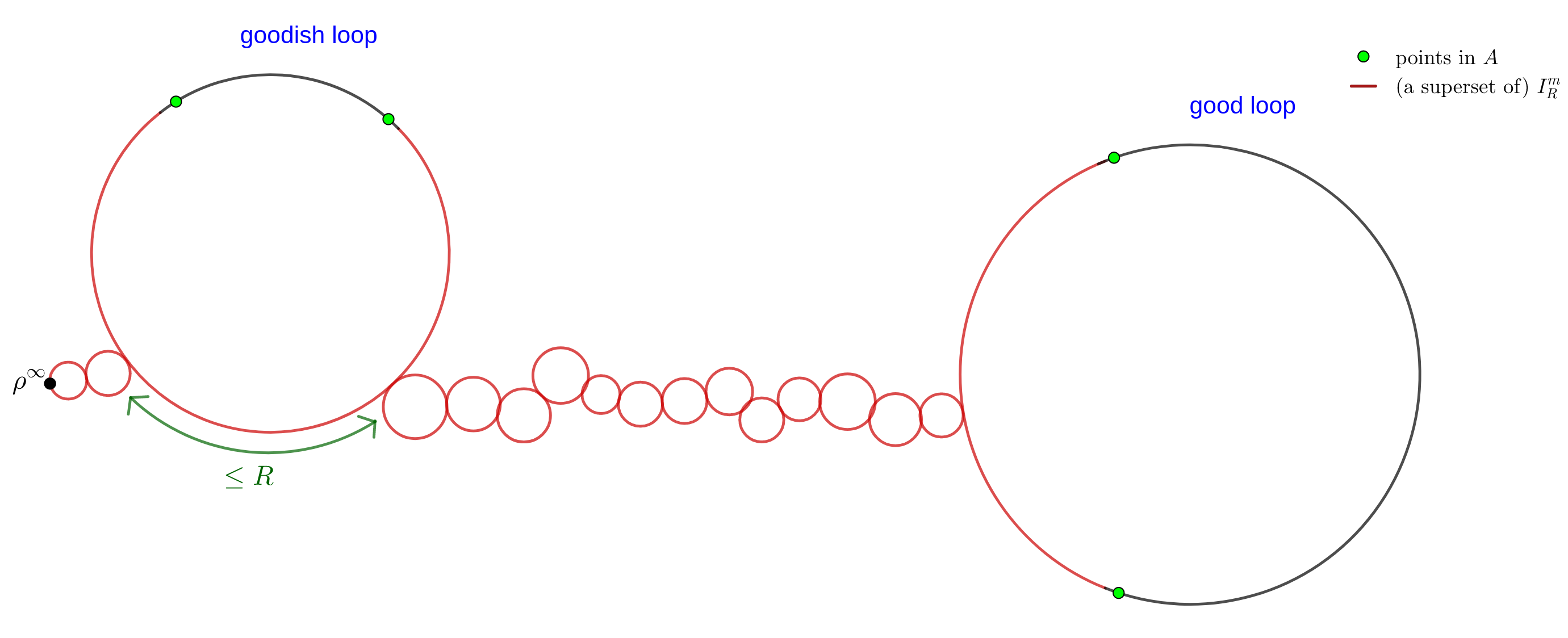}
\centering
\caption{How to select $A$. The red segment contains the portion of $B(\rho^{\infty}, R)$ intersecting the m-loopspine.}\label{fig:loopspine cutset}
\end{figure}

From the probabilistic bounds above, and since we set $h = (\alpha - 1)(2t-1)$, we therefore deduce that
\[
\prb{\Refi(\rho^{\infty}, B^{\infty}(\rho^{\infty}, 1)^c) \leq \frac{1}{2}\lambda^{-4t}} \leq 
C \lambda^{t}Ce^{-c\lambda^{t(h \wedge 1)}} + e^{-c\lambda^t} \leq C \lambda^{t}Ce^{-c\lambda^{t(2t-1)(\alpha -1)}} + Ce^{-c\lambda^t}.
\]
In particular, choosing $t> \frac{\alpha}{2(\alpha - 1)}$, we obtain
\[
\prb{\Refi(\rho^{\infty}, B^{\infty}(\rho^{\infty}, 1)^c) \leq \frac{1}{2}\lambda^{-4t}} \leq 
Ce^{-c\lambda^t},
\]
or equivalently, 
\[
\prb{\Refi(\rho^{\infty}, B^{\infty}(\rho^{\infty}, 1)^c) \leq \lambda^{-1}} \leq 
Ce^{-c\lambda^{\frac{1}{4}}}.
\]
\end{proof}

\section{Random walk limits}\label{sctn:RW consequences}
\subsection{Brownian motion and spectral dimension of $\Lai$}

As in the case of compact looptrees, the looptree convergence results can be used to give a collection of limit results for random walks and Brownian motion on sequences of looptrees. Before we do this, we have to show that $R^{\infty}$ is in fact a resistance metric, and that the resistance form associated with the metric space $(\Lai, R^{\infty})$ is regular, which implies that it is also a regular Dirichlet form on the space $L^2(\Lai, \nu)$ and so is naturally associated with a stochastic process. This is done in the following two propositions.

\begin{proposition}\label{prop:res metric}
$\bPb$-almost surely, $R^{\infty}$ is a resistance metric in the sense of Definition \ref{def:eff resistance metric}.
\end{proposition}
\begin{proof}
This follows from \cite[Proposition 4.4]{ArchBMCompactLooptrees}, in which we prove the same result for compact stable looptrees. In particular, any finite set of points $V$ in $\Lai$ is contained in $B(\rho^{\infty}, r)$ for some $r>0$. Taking such an $r$, we then define $t_g(r)$ and $t_d(r)$ exactly as we did in the proof of Theorem \ref{thm:LLT}; that is, we set
\begin{align*}
t_g(r) = \inf \{ s \geq 0: \Delta_{-s} \geq 4r, \delta^{\infty}_{-s}(x^{\infty}_{-s,0}) \geq r \}, \hspace{10mm} t_d(r) = \inf \{s \geq 0: \Xi_s \leq \Xi_{{-t_g(r)}^-} \}.
\end{align*}
As in previous proofs, it then follows that $B(\rho_{\infty}, r) \subset p^{\infty}([-t_g(r), t_d(r)])$, and $p^{\infty}(-t_g(r)) = p^{\infty}(t_d(r))$. Moreover, $p^{\infty}([-t_g(r), t_d(r)])$ codes a compact stable looptree, which, in keeping with earlier notation, we denote by $\La (r)$. We endow it with a metric and a measure by restricting $R^{\infty}$ and $\nu^{\infty}$ to $\La (r)$.

It then follows exactly as in \cite[Proposition 4.4]{ArchBMCompactLooptrees} that $R^{\infty}$ restricted to $\La(r)$ is a resistance metric on $\La (r)$, and that we can therefore construct a weighted network with vertex set $V$ with matching effective resistance. The same network will therefore work for $\Lai$.
\end{proof}

\begin{proposition}\label{prop:regular DF}
$\bPb$-almost surely, the resistance form associated with the metric space $(\Lai, R^{\infty})$ is regular.
\begin{proof}
We let $(\Ei, \Fi)$ denote the resistance form on $\Lai$ associated with the resistance metric $R^{\infty}$ as in (\ref{eqn:resistance def variational}). According to Definition \ref{def:reg res form}, we need to show that for any $f \in C_0(\Lai)$ and any $\epsilon > 0$, we can find $g' \in \Fi \cap C_0 (\Lai)$ such that $||f - g'||_{\infty} \leq \epsilon$. The key point is that by cutting off the infinite loopspine of $\Lai$ at an appropriate cutpoint, any such $f$ is also a compactly supported function on a compact stable looptree, and therefore approximable on this compact looptree, since all resistance forms on compact spaces are regular. Formally, we proceed as follows.

First, note that since $f$ is compactly supported, then its support must be contained in $B(\rho^{\infty}, r)$ for some $r>0$. Taking such an $r$, we then define $t_g(r)$ and $t_d(r)$ exactly as we did in the proof of Theorem \ref{thm:LLT}; that is, we set
\begin{align*}
t_g(r) = \inf \{ s \geq 0: \Delta_{-s} \geq 4r, \delta^{\infty}_{-s}(x^{\infty}_{-s,0}) \geq r \}, \hspace{10mm} t_d(r) = \inf \{s \geq 0: \Xi_s \leq \Xi_{{-t_g(r)}^-} \}.
\end{align*}
As in previous proofs, it then follows that $B(\rho_{\infty}, r) \subset p^{\infty}([-t_g(r), t_d(r)])$, and $p^{\infty}(-t_g(r)) = p^{\infty}(t_d(r))$. We denote this projected point by $v_r$. Moreover, $p^{\infty}([-t_g(r), t_d(r)])$ codes a compact stable looptree, which, in keeping with earlier notation, we denote by $\La (r)$. We endow it with a metric and a measure by restricting $R^{\infty}$ and $\nu^{\infty}$ to $\La (r)$, and denote the associated resistance form by $(\Er, \Fr)$.

The key point is the following: by \cite[Theorem 8.4]{KigamiResistanceFormsMono}, and the one-to-one correspondence given by (\ref{eqn:resistance def variational}) and its continuum extension on compact spaces, $(\Er, \Fr)$ is obtained as the trace of $(\Ei, \Fi)$ on $\La (r)$, and is such that for any $f \in \Fr$, $\Er (f, f) = \Ei (h(f), h(f))$, where $h(f)$ is the unique harmonic extension of $f$ to $\Lai$.

Now take $f \in \Fi$. Note that, necessarily, $f(v_r)=0$, since $f$ is continuous. Moreover, $v_r$ is a point on the infinite loopspine that cuts $\rho^{\infty}$ off from $\infty$. Arbitrarily, we now choose a new point $v_r'$ on the loopspine, coded by a jump point of $X^{\infty}$, that also separates $\rho^{\infty}$ from $\infty$, but such that $R^{\infty} (\rho^{\infty}, v_r') > R^{\infty}(\rho^{\infty}, v_r)$. It follows that $v_r'$ is coded by jump point of $X^{\infty}$ at a time that we denote by $-t_{g,2}(r)$, where $t_{g,2}(r) > t_g(r)$ and $-t_{g,2}(r) \preceq 0$. For any $s$ with $-t_{g,2}(r) \preceq s \prec -t_g(r)$, set $a_s = \delta_s ( x_{s,0}^{\infty})$, and $b_s = \Delta_s - \delta_s ( x_{s,0}^{\infty})$, so that $a_s$ gives the length of the ``shorter" segment of the corresponding loop in the loopspine, and $b_s$ gives the length of the ``longer" segment (see Figure \ref{fig:regularRF}). Set
\[
d_{\text{min}} = \sum_{-t_{g,2}(r) \preceq s \prec -t_g(r)} a_s, \hspace{10mm} d_{\text{max}} = \sum_{-t_{g,2}(r) \preceq s \prec -t_g(r)} b_s.
\]
These are defined so that $d_{\text{min}}$ gives the looptree distance between $v_r$ and $v_r'$, and $d_{\text{max}}$ gives the ``longer distance" between them, which is the length of the path between them that traverses the longer side of all the loops in the loopspine that lie between $v_r$ and $v_r'$ (see Figure \ref{fig:regularRF}).

\begin{figure}[h]
\includegraphics[width=14cm, height=5.3cm]{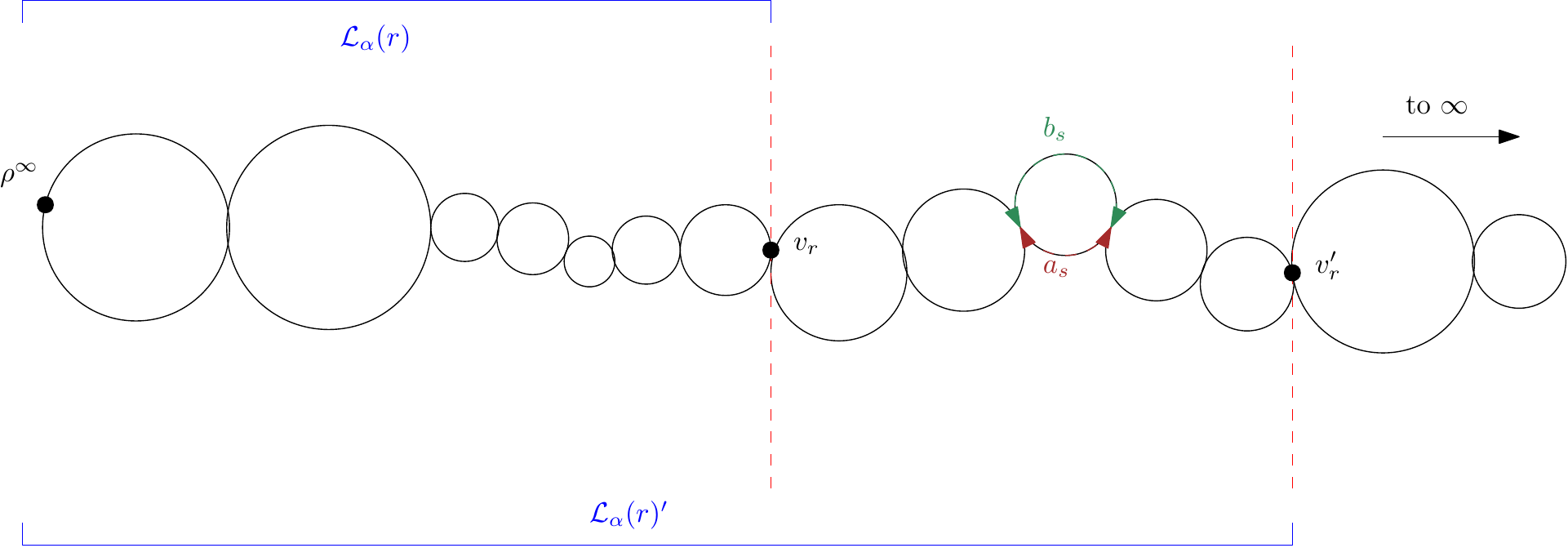}
\centering
\caption{Illustration of how we cut the infinite loopspine.}\label{fig:regularRF}
\end{figure}

Additionally, let $t_{d,2}(r) = \inf \{s \geq 0: \Xi_s \leq \Xi_{{-t_{g,2}(r)}^-} \}$. Then $p^{\infty}([ -t_{g,2}(r), t_{d,2}(r)])$ codes another compact stable looptree which we denote by $\La (r)'$, satisfying $\La (r) \subset \La (r)' \subset \Lai$.

Since $\La (r)$ is compact, it follows that $(\Er, \Fr)$ is regular, so there exists $g \in \Fr \cap C_0(\La(r))$ with $||f|_{\La (r)}-g||_{\infty, \La (r)} \leq \epsilon$. We therefore define a function $g' \in C_0(\Lai)$ by setting $g' = g$ on $\La (r)$, $g'=0$ on $\Lai \setminus \La (r)'$, and extending harmonically on $\La (r)' \setminus \La (r)$.

Since $g$ approximates $f_{\La (r)}$ in the supremum norm, it follows that $|g(v_r)| \leq \epsilon$, and moreover it then follows by the maximum principle for harmonic functions that $||g'_{\La (r)' \setminus \La (r)}||_{\infty} \leq \epsilon$. Consequently, $||f-g'||_{\infty} \leq \epsilon$. It therefore just remains to show that $\Ei (g', g') < \infty$.

Let $(\Err, \Frr)$ denote the restriction of $(\Ei, \Fi)$ to $\La(r)'$. Since the spaces $\La (r), \La (r)' \setminus \La (r)$ and $\Lai \setminus \La (r)'$ are disjoint, and $g'$ is the harmonic extension of $g'|_{\La (r)'}$ to $\Lai$, it follows by bilinearity and from consistency properties of resistance forms and their traces given in \cite[Section 8]{KigamiResistanceFormsMono} that
\begin{equation}\label{eqn:energy decomp}
\Ei (g', g') = \Err (g'|_{\La (r)'}, g'|_{\La (r)'}).
\end{equation}
However, since $\La(r)'$ is simply a compact looptree, this is automatically finite.

\end{proof}
\end{proposition}

As a result, we deduce that the resistance metric space is naturally associated with a Hunt process on $(\Lai, R^{\infty})$, which we call Brownian motion on $\Lai$ and denote by $B^{\infty}$.

\subsection{Quenched results}
We can apply Theorem \ref{thm:scaling lim RW resistance} to the results of Theorems \ref{thm:LLT} and \ref{thm:main scaling lim} to deduce convergence results for stochastic processes on the corresponding spaces. The only additional detail in the proofs of these results is that we have to check that the non-explosion condition at (\ref{eqn:nonexplosion}) is satisfied, i.e. that
\begin{equation*}
\lim_{r \rightarrow \infty} \liminf_{{\l} \rightarrow \infty} R^{\l} (\rho^{\l}, B^{\l} (\rho^{\l}, r)^c) = \infty
\end{equation*}
almost surely, where $R^{\l}$ here denotes the resistance metric on $\Lal$.

\subsubsection{Local limits}\label{sctn:RW limit local infinite}
The local limit theorem of Theorem \ref{thm:LLT} immediately allows us to apply Theorem \ref{thm:scaling lim RW resistance} to deduce that Brownian motion on $\Lal$ converges in distribution to Brownian motion on $\Lai$ as $\l \rightarrow \infty$ on compact time intervals. Indeed, it follows from Theorem \ref{thm:GH vague separable HB} and the Skorohod Representation Theorem that there exists a probability space on which the convergence on Theorem \ref{thm:LLT} holds almost surely. Moreover, the explosion condition is satisfied as an immediate consequence of Proposition \ref{prop:resistance results infinite}. In particular, the arguments used to prove Proposition \ref{prop:resistance results infinite} are also valid for compact stable looptrees, so we deduce that the resistance bounds of Proposition \ref{prop:resistance results infinite} almost surely hold along the sequence $(\Lal)_{\l \in \N}$.

Theorem \ref{thm:main RW LLT conv} then follows by a direct application of Theorem \ref{thm:scaling lim RW resistance}.

\subsubsection{Scaling limits}
We can also deduce similar results from Theorems \ref{thm:RW conv infinite intro}, \ref{thm:BS looptree conv} and \ref{thm:Rich looptree conv}. In this case, the non-explosion condition is satisfied as a result of \cite[Lemma 3.5]{BjornStef}, which says that for $\Loopp(\Tai)$, there exist $q, C \in (0, \infty)$ such that
\begin{equation}\label{eqn:BS nonexp}
\prb{\Reff (\rho, B(\rho, r)^c) \leq r\lambda^{-1}} \leq C\lambda^{-q}.
\end{equation}
In light of Proposition \ref{prop:resistance results infinite}, we conjecture that there should in actual fact be exponential tail decay, but polynomial decay is sufficient for our purposes here. Indeed, to verify (\ref{eqn:nonexplosion}), we need to show that 
\[
\lim_{r \rightarrow \infty} \liminf_{n \rightarrow \infty} n^{\frac{-1}{\alpha}} \Reff (\rho, B(\rho, rn^{\frac{1}{\alpha}})^c) = \infty
\]
$\bPb$-almost surely. This follows directly from applying a Borel-Cantelli argument along a suitable subsequence using the probabilistic bound (\ref{eqn:BS nonexp}). Moreover, the same arguments apply for $\Loop(\Tai)$ since $\Reff (\rho, \mathcal{B}_r(\Loop(\Tai))^c) \geq \Reff (\rho, B_{r-1}(\Loopp(\Tai))^c)$.

Similarly, the result also holds for the two-type looptree $\Loop^2(T_{\alpha}^{\infty, 2})$, since $\Reff (\rho, \mathcal{B}_{r}(\overline{\Loop}(\Tai))^c) \geq \Reff (\rho, B_{r - \Height(\textsf{Tree}(\Loopp (\Tai)^r))}(\Loopp(\Tai))^c)$, and $r^{-1}\Height(\textsf{Tree}(\Loopp (\Tai)^r)) \rightarrow 0$ in probability, with exponential tail decay (as in Point 2 of the proof of Theorem \ref{thm:Rich looptree conv}), allowing further Borel-Cantelli arguments.

In all the different versions of infinite looptrees that we have considered, the Gromov-Hausdorff-Prohorov convergence holds with the uniform measure on vertices of the looptree, and the associated stochastic process is therefore a variable speed random walk.

In the case of $\Loop(\Tai)$, all vertices have degree $4$, so in this case the stochastic process is actually a constant speed random walk, with \textsf{exp}($4$) waiting times at each vertex.
However, by applying Kolmogorov's Maximal Inequality to the time index of this stochastic process (as in the proof of \cite[Theorem 1.1]{ArchBMCompactLooptrees}) we can show that the waiting times average out sufficiently well over time so the scaling limit result will also hold for a simple random walk on $\textsf{Loop}(\Tai)$ (although sped up deterministically by a factor of $4$). 

Theorem \ref{thm:RW conv infinite intro} therefore follows by an immediate application of Theorem \ref{thm:scaling lim RW resistance} to Proposition \ref{prop:inf conv disc loop}.

In the case of $\Loopp(\Tai)$, all internal vertices have degree $4$, and all leaf vertices have degree $2$. This corresponds to the fact the the only significant difference between $\Loop(\Tai)$ and $\Loopp(\Tai)$ is that in $\Loopp(\Tai)$ the loops corresponding to leaves are missing, and has the effect that (on average) the random walk waits twice as long at leaf vertices compared to internal vertices. This reflects the fact that the random walks on $\Loop(\Tai)$ and $\Loopp(\Tai)$ can (almost, technically only after adding one extra vertex to the loop containing the root in $\Loop (\Tai)$) be coupled so that they move identically at internal vertices, but so that a random walk on $\Loopp(\Tai)$ remains in its present position whenever the random walk on $\Loop(\Tai)$ traverses a loop corresponding to a leaf vertex (note this can be traversed in either direction). It therefore makes sense that we should be taking a scaling limit of the variable speed random walk on $\Loopp(\Tai)$, rather than the constant speed one.

We similarly have to take a variable speed random walk on $\Loop^2(T_{\alpha}^{\infty, 2})$, although there is not such a simple coupling in this case. In the next theorem, we let $L_{\alpha}^{\infty, 1} = \Loopp(\Tai)$, $L_{\alpha}^{\infty, 2} = \Loop^2(T_{\alpha}^{\infty, 2})$, $Y^{\text{var}, i}$ denote a variable speed random walk on $L_{\alpha}^{\infty, i}$, and $\nu^i$ denote the measure giving mass $1$ to each vertex. The non-explosion condition is again satisfied by the same arguments as in Section \ref{sctn:RW limit local infinite} above. We then have the following analogues of Theorem \ref{thm:RW conv infinite intro}.

\begin{theorem}\label{thm:RW conv infinite BjornStef}
Take $i \in \{1, 2\}$. There exists a probability space $(\Omega', \mathcal{F}', \mathbf{P}')$ on which we can almost surely define a common metric space $(M, R_M)$ in which the spaces $(L_{\alpha}^{\infty, i}, a_n^{-1} \d, n^{-1} \nu', \rho)$ and $(\La^{\infty}, \d^{\infty}, \nu^{\infty}, \rho^{\infty})$ can all be isometrically embedded and such that
\[
(L_{\alpha}^{\infty, i}, a_n^{-1}\d, n^{-1} \nu^i, \rho) \overset{(d)}{\rightarrow} (\Lai, \d^{\infty}, \nu^{\infty}, \rho^{\infty})
\]
with respect to the Gromov-Hausdorff-vague topology, and the convergence specifically holds on the metric space $(M, R_M)$. Letting $Y^{\text{var},i}$ and $B^{\infty}$ be as above, we have that 
\[
(a_n^{-1}Y^{\text{var},i}_{\lfloor n a_n t \rfloor})_{t \geq 0} \overset{(d)}{\rightarrow} (B^{\infty}_t)_{t \geq 0}
\]
on the space $D(\R^+, M$) as $n \rightarrow \infty$.
\end{theorem} 

\begin{remark}
We could also prove other convergence results, for example by taking increasing sequences of increasingly rescaled discrete looptrees to approximate $\Lai$, in some sense combining Theorems \ref{thm:LLT} and \ref{thm:compact disc inv princ res}, and deduce similar convergence results for random walks, exactly as we did in the cases above. This corresponds to the diagonal line in Figure \ref{fig:commuting diag intro}.
\end{remark}

\subsection{Annealed results}

We can also prove similar results in the annealed setting by embedding into the Urysohn space, where we recall that if $(F, R, \mu, \rho, \phi)$ is a random element of $\mathbb{F}$ with law $\bPb$ such that $\phi: F \rightarrow U$ is a (possibly random) isometric embedding, and $(Y_t)_{t \geq 0}$ is a stochastic process on $F$, we define its annealed law by
\[
\prtilnstart{\phi(Y_t)_{t \geq 0} \in \cdot}{\phi(\rho)}{} = \int \prnstart{\phi(Y_t)_{t \geq 0} \in \cdot}{\phi(\rho)}{}  d\bPb,
\]
as introduced in Section \ref{sctn:res forms}.

Again we will restrict to the subsequence of integral $\l$ in Theorem \ref{thm:RW conv local annealed} below, but the result holds along any countable subsequence diverging to infinity.

\begin{theorem}\label{thm:RW conv local annealed}
Let $(\Lal, \d^{\l}, \nu^{\l}, \rho^{\l})_{\l \geq 1}$ be as in Theorem \ref{thm:LLT}. Then there exist (random) embeddings $\phi_{\l}: (\Lal, \d^{\l}, \nu^{\l}, \rho^{\l}) \rightarrow U, \phi: (\Lai, \d^{\infty}, \nu^{\infty}, \rho^{\infty}) \rightarrow U$ such that
\[
\prtilnstart{\phi_{\l}(B^{\l}_t)_{t \geq 0} \in \cdot}{\phi_{\l}(\rho^{\l})}{\l} \rightarrow \prtilnstart{(B^{\infty}_t)_{t \geq 0} \in \cdot}{\phi(\rho)}{}
\]
weakly as probability measures on the space $D(\R^+, U)$ endowed with the Skorohod-$J_1$ topology as $n \rightarrow \infty$.
\end{theorem} 
\begin{proof}
The proof of Theorem \ref{thm:LLT} implies that there exists a probability space on which $(\Lal, \d^{\l}, \nu^{\l}, \rho^{\l}) \rightarrow (\Lai, \d^{\infty}, \nu^{\infty}, \rho^{\infty})$ almost surely in the Gromov-Hausdorff vague topology. In particular, there exists a metric space $M = \Lai \sqcup \La^1 \sqcup \La^2 \sqcup \ldots$ defined on this probability space such that $\Lal \rightarrow \Lai$ almost surely. Moreover, by properties of the Urysohn space discussed in Section \ref{sctn:res forms}, there exists an isometry $\psi: M \rightarrow U$ such that $\psi(\rho^{\infty}) = u_0$.

For each $\l \in \N \cup \{\infty\}$, let $\phi_{\l}$ be the canonical isometry embedding $\Lal$ into $M$. It then follows that $\psi_{\l} := \psi \circ \phi_{\l}$ is an isometry from $\Lal$ to $U$ and moreover that almost surely, $\psi_{\l} (\Lal) \rightarrow \psi_{\infty} (\Lai)$ Gromov-Hausdorff vaguely as $\l \rightarrow \infty$. Viewing $(\phi_{\l})_{\l \geq 1}$ and $\psi_{\infty}$ as spatial embeddings, this therefore automatically implies that the spaces converge in the metric introduced at (\ref{eqn:spat U dist}). Since the topology induced by this metric is a particular instance of the spatial Gromov-Hausdorff topology used in \cite[Section 7]{DavidResForms}, we are in the right setting to apply Theorem \ref{thm:scaling lim RW resistance annealed}. Indeed, the non-explosion condition (\ref{eqn:nonexplosion annealed}) is satisfied as a direct consequence of Proposition \ref{prop:resistance results infinite}, which also uniformly holds along the sequence $(\Lal)_{l \geq 1}$. The theorem then follows by a direct application of Theorem \ref{thm:scaling lim RW resistance annealed}.
%
%
\end{proof}

We can also prove a similar results for the spaces $\Loop(\Tai)$, $L_{\alpha}^1$ and $L_{\alpha}^2$. We omit the proofs since they are essentially identical to that of Theorem \ref{thm:RW conv local annealed} above.

\begin{theorem}\label{thm:RW conv scaling annealed}
Let $(\textsf{Loop}(\Tai), a_n^{-1}\d, n^{-1} \nu^{\text{disc}}, \rho)$ be as in Theorem \ref{thm:main scaling lim}. Then there exist (random) embeddings $\phi_n: (\textsf{Loop}(\Tai), a_n^{-1}\d, n^{-1} \nu^{\text{disc}}, \rho) \rightarrow U, \phi: (\Lai, \d^{\infty}, \nu^{\infty}, \rho^{\infty}) \rightarrow U$ such that
\[
\prtilnstart{\phi_n\big(a_n^{-1} Y^{(n)}_{\lfloor 4n a_n t \rfloor}\big)_{t \geq 0} \in \cdot}{\phi_n(\rho_n)}{n} \rightarrow \prtilnstart{(B^{\infty}_t)_{t \geq 0} \in \cdot}{\phi(\rho)}{}
\]
weakly as probability measures on the space $D(\R^+, U)$ endowed with the Skorohod-$J_1$ topology as $n \rightarrow \infty$.
\end{theorem} 

\begin{theorem}\label{thm:RW conv infinite BjornStef annealed}
Take $i \in \{1, 2\}$, and let $(L_{\alpha}^{\infty, i}, a_n^{-1}\d, n^{-1} \nu^i, \rho)$ be as in Theorem \ref{thm:RW conv infinite BjornStef}. Then there exist (random) embeddings $\phi_n: (L_{\alpha}^{\infty, i}, a_n^{-1}\d, n^{-1} \nu^{i}, \rho) \rightarrow U, \phi: (\Lai, \d^{\infty}, \nu^{\infty}, \rho^{\infty}) \rightarrow U$ such that
\[
\prtilnstart{\phi_n(a_n^{-1}Y^{\text{var},i}_{\lfloor n a_n t \rfloor})_{t \geq 0} \in \cdot}{\phi_n(\rho_n)}{n} \rightarrow \prtilnstart{(B^{\infty}_t)_{t \geq 0} \in \cdot}{\phi(\rho)}{}
\]
weakly as probability measures on the space $D(\R^+, U)$ endowed with the Skorohod-$J_1$ topology as $n \rightarrow \infty$.
\end{theorem} 

\subsection{Heat kernel convergence and spectral dimension}\label{sctn:annealed HK from infinite looptrees}
To conclude, we now show how Theorem \ref{thm:LLT} can be applied to give results on the heat kernel of Brownian motion on compact stable looptrees. First, note that it follows from the scaling invariance of Proposition \ref{prop:Lai scale inv} that the annealed heat kernel for $\Lai$ satisfies the scaling relation
\begin{equation}\label{eqn:HK scaling}
\Eb{p_t^{\infty}(\rho, \rho)} = k^{\frac{\alpha}{\alpha + 1}} \Eb{p_{kt}^{\infty}(\rho, \rho)}
\end{equation}
for any $k>0$. Similarly, if we let $p_t^{\l}$ denote the transition density of Brownian motion on a looptree coded by an excursion of length $\l$, we have that
\[
\Eb{{p_t^{1}(\rho, \rho)}} = k^{\frac{\alpha}{\alpha + 1}} \Eb{p_{kt}^{k^{\frac{1}{\alpha + 1}}}(\rho, \rho)}. 
\]
Setting $k=t^{-1}$ we see that
\[
t^{\frac{\alpha}{\alpha + 1}} \Eb{p_t^{1}(\rho, \rho)} {=} \Eb{ p_{1}^{t^{\frac{-1}{\alpha + 1}}}(\rho, \rho)}.
\]
Moreover, since we are in a resistance framework, it follows from \cite[Theorem 2 and Proposition 14]{CroyHamLLT} that
\[
t^{\frac{\alpha}{\alpha + 1}} p_t^{1}(\rho, \rho) \overset{(d)}{\rightarrow} p_{1}^{\infty}(\rho, \rho)
\]
as $t \downarrow 0$. To deduce that the corresponding expectations also converge, we just need to show that $\Eb{p_{1}^{\infty}(\rho, \rho)}$ is finite. However, since the transition density can be bounded by bounding the volume and resistance growth (by a continuum version of \cite[Proposition 1.4]{KumMisumiHKStronglyRecurrent}, for example), the exponential tail decay of Propositions \ref{prop:vol results infinite} and \ref{prop:resistance results infinite} also give an upper exponential tail decay for the transition density. We therefore deduce that $\Eb{p_{1}^{\infty}(\rho, \rho)}$ is finite, so we can apply similar arguments to those in the previous section to deduce that
\[
t^{\frac{\alpha}{\alpha + 1}} \Eb{p_t^{1}(\rho, \rho)} \rightarrow \Eb{p_{1}^{\infty}(\rho, \rho)}
\]
as $t \rightarrow \infty$. This is stated as \cite[Theorem 1.8]{ArchBMCompactLooptrees}, where Brownian motion on $\La$ is studied more closely.

Similarly, it also follows from \cite[Theorem 1.5, Part II]{KumMisumiHKStronglyRecurrent} (adapted to the continuum) that the heat kernel $p^{\infty}_t(\rho, \rho)$ almost surely experiences at most log-logarithmic fluctuations around a leading term of $t^{\frac{-\alpha}{\alpha + 1}}$ as $t \uparrow \infty$ and as $t \downarrow 0$, and therefore that the quenched spectral dimension of $\La$ is almost surely equal to $\frac{2 \alpha}{\alpha + 1}$.

To establish the annealed spectral dimension, we take $k=t^{-1}$ in (\ref{eqn:HK scaling}) to deduce that
\[
\Eb{p_t^{\infty}(\rho, \rho)} = t^{\frac{-\alpha}{\alpha + 1}} \Eb{p_{1}^{\infty}(\rho, \rho)}.
\]
Since $\Eb{p_{1}^{\infty}(\rho, \rho)}$ is finite, this implies that the annealed spectral dimension is also equal to $\frac{2 \alpha}{\alpha + 1}$. This concludes the proof of Theorem \ref{thm:main spec dim annealed}.

\bibliographystyle{alpha}

\bibliography{biblio}

\begin{thebibliography}{CDKM15}

\bibitem[AD09]{AbDelWilliamsDecomp}
R.~Abraham and J-F. Delmas.
\newblock Williams' decomposition of the {L}\'evy continuum random tree and
  simultaneous extinction probability for populations with neutral mutations.
\newblock {\em Stochastic Process. Appl.}, 119(4):1124--1143, 2009.

\bibitem[AD15]{AbDelGWIntro}
R.~{Abraham} and J.-F. {Delmas}.
\newblock {An introduction to Galton-Watson trees and their local limits}.
\newblock {\em ArXiv e-prints}, 1506.05571, June 2015.

\bibitem[ADH13]{AbDelHoschNoteGromov}
R.~Abraham, J-F. Delmas, and P.~Hoscheit.
\newblock A note on the {G}romov-{H}ausdorff-{P}rokhorov distance between
  (locally) compact metric measure spaces.
\newblock {\em Electron. J. Probab.}, 18:no. 14, 21, 2013.

\bibitem[Ald91]{AldousFringeSinTree}
D.~Aldous.
\newblock Asymptotic fringe distributions for general families of random trees.
\newblock {\em Ann. Appl. Probab.}, 1(2):228--266, 1991.

\bibitem[ALW16]{AthLohrWinGromovGap}
S.~Athreya, W.~L\"ohr, and A.~Winter.
\newblock The gap between {G}romov-vague and {G}romov-{H}ausdorff-vague
  topology.
\newblock {\em Stochastic Process. Appl.}, 126(9):2527--2553, 2016.

\bibitem[{Arc}19]{ArchBMCompactLooptrees}
E.~{Archer}.
\newblock {Brownian motion on stable looptrees}.
\newblock {\em arXiv e-prints}, 1902.01713:arXiv:1902.01713, Feb 2019.

\bibitem[ARFK18]{AlRuiFreiKigSSG}
P.~Alonso~Ruiz, U.~Freiberg, and J.~Kigami.
\newblock Completely symmetric resistance forms on the stretched
  {S}ierpi\'{n}ski gasket.
\newblock {\em J. Fractal Geom.}, 5(3):227--277, 2018.

\bibitem[BBI01]{Burago}
D.~Burago, Y.~Burago, and S.~Ivanov.
\newblock {\em A course in metric geometry}, volume~33 of {\em Graduate Studies
  in Mathematics}.
\newblock American Mathematical Society, Providence, RI, 2001.

\bibitem[BCM19]{BerCurMierPerconTriang}
O.~{Bernardi}, N.~{Curien}, and G.~{Miermont}.
\newblock A {B}oltzmann approach to percolation on random triangulations.
\newblock {\em Canad. J. Math.}, 71(1):1--43, 2019.

\bibitem[Ber92]{BertoinPitmanExt}
J.~Bertoin.
\newblock An extension of {P}itman's theorem for spectrally positive {L}\'evy
  processes.
\newblock {\em Ann. Probab.}, 20(3):1464--1483, 1992.

\bibitem[Ber96]{BertoinLevy}
J.~Bertoin.
\newblock {\em L\'evy processes}, volume 121 of {\em Cambridge Tracts in
  Mathematics}.
\newblock Cambridge University Press, Cambridge, 1996.

\bibitem[BGT89]{BGTRegVariation}
N.~Bingham, C.~Goldie, and J.~Teugels.
\newblock {\em Regular variation}, volume~27 of {\em Encyclopedia of
  Mathematics and its Applications}.
\newblock Cambridge University Press, Cambridge, 1989.

\bibitem[BHS18]{BernardiHoldenSun}
O.~{Bernardi}, N.~{Holden}, and X.~{Sun}.
\newblock {Percolation on triangulations: a bijective path to Liouville quantum
  gravity}.
\newblock {\em ArXiv e-prints}, 1807.01684, July 2018.

\bibitem[Bil68]{BillsleyConv}
P.~Billingsley.
\newblock {\em Convergence of probability measures}.
\newblock John Wiley \& Sons, Inc., New York-London-Sydney, 1968.

\bibitem[BR18]{BaurRichUIPQSkew}
E.~{Baur} and L.~{Richier}.
\newblock {Uniform infinite half-planar quadrangulations with skewness}.
\newblock {\em Electronic Journal of Probability}, 23, 2018.

\bibitem[BS15]{BjornStef}
J.~Bj\"ornberg and S.~Stef\'ansson.
\newblock Random walk on random infinite looptrees.
\newblock {\em J. Stat. Phys.}, 158(6):1234--1261, 2015.

\bibitem[CDKM15]{CurKortDuqMan}
N.~Curien, T.~Duquesne, I.~Kortchemski, and I.~Manolescu.
\newblock Scaling limits and influence of the seed graph in preferential
  attachment trees.
\newblock {\em J. \'Ec. polytech. Math.}, 2:1--34, 2015.

\bibitem[CH08]{CroyHamLLT}
D.~Croydon and B.~Hambly.
\newblock Local limit theorems for sequences of simple random walks on graphs.
\newblock {\em Potential Anal.}, 29(4):351--389, 2008.

\bibitem[Cha97]{Chaumont}
L.~Chaumont.
\newblock Excursion normalis\'ee, m\'eandre et pont pour les processus de
  {L}\'evy stables.
\newblock {\em Bull. Sci. Math.}, 121(5):377--403, 1997.

\bibitem[CK08]{CroyKumRWGWTreeInfiniteVar}
D.~Croydon and T.~Kumagai.
\newblock Random walks on {G}alton-{W}atson trees with infinite variance
  offspring distribution conditioned to survive.
\newblock {\em Electron. J. Probab.}, 13:no. 51, 1419--1441, 2008.

\bibitem[CK14]{RSLTCurKort}
N.~Curien and I.~Kortchemski.
\newblock Random stable looptrees.
\newblock {\em Electron. J. Probab.}, 19:no. 108, 35, 2014.

\bibitem[CK15]{CurKortUIPTPerc}
N.~Curien and I.~Kortchemski.
\newblock Percolation on random triangulations and stable looptrees.
\newblock {\em Probab. Theory Related Fields}, 163(1-2):303--337, 2015.

\bibitem[CR18]{CurRichDualityRPMPerc}
N.~{Curien} and L.~{Richier}.
\newblock {Duality of random planar maps via percolation}.
\newblock {\em ArXiv e-prints}, 1802.01576, February 2018.

\bibitem[Cro18]{DavidResForms}
D.~Croydon.
\newblock Scaling limits of stochastic processes associated with resistance
  forms.
\newblock {\em Ann. Inst. Henri Poincar\'{e} Probab. Stat.}, 54(4):1939--1968,
  2018.

\bibitem[DLG02]{LeGDuqMono}
T.~Duquesne and J-F. Le~Gall.
\newblock Random trees, {L}\'evy processes and spatial branching processes.
\newblock {\em Ast\'erisque}, (281):vi+147, 2002.

\bibitem[DLG05]{DuqLeGPFALT}
T.~Duquesne and J-F. Le~Gall.
\newblock Probabilistic and fractal aspects of {L}\'evy trees.
\newblock {\em Probab. Theory Related Fields}, 131(4):553--603, 2005.

\bibitem[DLG09]{DuqLeGRerooting}
T.~Duquesne and J-F. Le~Gall.
\newblock On the re-rooting invariance property of {L}\'evy trees.
\newblock {\em Electron. Commun. Probab.}, 14:317--326, 2009.

\bibitem[Duq03]{DuqContourLimit}
T.~Duquesne.
\newblock A limit theorem for the contour process of conditioned
  {G}alton-{W}atson trees.
\newblock {\em Ann. Probab.}, 31(2):996--1027, 2003.

\bibitem[Duq09]{DuqSinTree}
T.~Duquesne.
\newblock Continuum random trees and branching processes with immigration.
\newblock {\em Stochastic Process. Appl.}, 119(1):99--129, 2009.

\bibitem[DW17]{DuqWangDiameter}
T.~Duquesne and M.~Wang.
\newblock Decomposition of {L}\'evy trees along their diameter.
\newblock {\em Ann. Inst. Henri Poincar\'e Probab. Stat.}, 53(2):539--593,
  2017.

\bibitem[Dwa69]{DwassProg}
M.~Dwass.
\newblock The total progeny in a branching process and a related random walk.
\newblock {\em J. Appl. Probability}, 6:682--686, 1969.

\bibitem[FOT11]{FOT}
M.~Fukushima, Y.~Oshima, and M.~Takeda.
\newblock {\em Dirichlet forms and symmetric {M}arkov processes}, volume~19 of
  {\em De Gruyter Studies in Mathematics}.
\newblock Walter de Gruyter \& Co., Berlin, extended edition, 2011.

\bibitem[GH10]{GoldHaasExtinctionStable}
C.~Goldschmidt and B.~Haas.
\newblock Behavior near the extinction time in self-similar fragmentations.
  {I}. {T}he stable case.
\newblock {\em Ann. Inst. Henri Poincar\'{e} Probab. Stat.}, 46(2):338--368,
  2010.

\bibitem[GM17]{GwynneMillerUIHPQscaling}
E.~Gwynne and J.~Miller.
\newblock Scaling limit of the uniform infinite half-plane quadrangulation in
  the {G}romov-{H}ausdorff-{P}rokhorov-uniform topology.
\newblock {\em Electron. J. Probab.}, 22:Paper No. 84, 47, 2017.

\bibitem[GP]{GwynnePfefferConnectivitySLE}
E.~{Gwynne} and J.~{Pfeffer}.
\newblock {Connectivity properties of the adjacency graph of SLE$\_\kappa$
  bubbles for $\kappa \in (4,8)$}.
\newblock {\em Ann. Prob., to appear}.

\bibitem[HPW09]{HPWSpinPart}
B.~Haas, J.~Pitman, and M.~Winkel.
\newblock Spinal partitions and invariance under re-rooting of continuum random
  trees.
\newblock {\em Ann. Probab.}, 37(4):1381--1411, 2009.

\bibitem[Hus08]{HuvekUrysohn}
M.~Husek.
\newblock Urysohn universal space, its development and {H}ausdorff's approach.
\newblock {\em Topology Appl.}, 155(14):1493--1501, 2008.

\bibitem[It{\^o}72]{ItoPP}
K.~It{\^o}.
\newblock Poisson point processes attached to {M}arkov processes.
\newblock In {\em Proceedings of the {S}ixth {B}erkeley {S}ymposium on
  {M}athematical {S}tatistics and {P}robability ({U}niv. {C}alifornia,
  {B}erkeley, {C}alif., 1970/1971), {V}ol. {III}: {P}robability theory}, pages
  225--239. Univ. California Press, Berkeley, Calif., 1972.

\bibitem[Jan12]{JansonSurvey}
S.~Janson.
\newblock Simply generated trees, conditioned {G}alton-{W}atson trees, random
  allocations and condensation.
\newblock {\em Probab. Surv.}, 9:103--252, 2012.

\bibitem[JS15]{JanStefLargeFace}
S.~Janson and S.~Stef\'{a}nsson.
\newblock Scaling limits of random planar maps with a unique large face.
\newblock {\em Ann. Probab.}, 43(3):1045--1081, 2015.

\bibitem[Kes86]{KestenIICtree}
H.~Kesten.
\newblock Subdiffusive behavior of random walk on a random cluster.
\newblock {\em Ann. Inst. H. Poincar\'e Probab. Statist.}, 22(4):425--487,
  1986.

\bibitem[Kig01]{AOF}
J.~Kigami.
\newblock {\em Analysis on fractals}, volume 143 of {\em Cambridge Tracts in
  Mathematics}.
\newblock Cambridge University Press, Cambridge, 2001.

\bibitem[Kig12]{KigamiResistanceFormsMono}
J.~Kigami.
\newblock Resistance forms, quasisymmetric maps and heat kernel estimates.
\newblock {\em Mem. Amer. Math. Soc.}, 216(1015):vi+132, 2012.

\bibitem[KM08]{KumMisumiHKStronglyRecurrent}
T.~Kumagai and J.~Misumi.
\newblock Heat kernel estimates for strongly recurrent random walk on random
  media.
\newblock {\em J. Theoret. Probab.}, 21(4):910--935, 2008.

\bibitem[KR]{KortRichBoundaryRPMLooptrees}
I.~{Kortchemski} and L.~{Richier}.
\newblock {The boundary of random planar maps via looptrees}.
\newblock {\em Ann. Fac. Sci. Toulouse Math, to appear}.

\bibitem[KR19]{KortRichCondensationCritical}
I.~Kortchemski and L.~Richier.
\newblock Condensation in critical {C}auchy {B}ienaym\'{e}-{G}alton-{W}atson
  trees.
\newblock {\em Ann. Appl. Probab.}, 29(3):1837--1877, 2019.

\bibitem[LGLJ98]{LeGLeJanExploration}
J-F. Le~Gall and Y.~Le~Jan.
\newblock Branching processes in {L}\'evy processes: the exploration process.
\newblock {\em Ann. Probab.}, 26(1):213--252, 1998.

\bibitem[LGM11]{LeGMiermontScalingLimitsLargeFaces}
J-F. Le~Gall and G.~Miermont.
\newblock Scaling limits of random planar maps with large faces.
\newblock {\em Ann. Probab.}, 39(1):1--69, 2011.

\bibitem[Mie05]{MiermontSplittingNodes}
G.~Miermont.
\newblock Self-similar fragmentations derived from the stable tree. {II}.
  {S}plitting at nodes.
\newblock {\em Probab. Theory Related Fields}, 131(3):341--375, 2005.

\bibitem[MS15]{MillSheff}
J.~{Miller} and S.~{Sheffield}.
\newblock {An axiomatic characterization of the Brownian map}.
\newblock {\em ArXiv e-prints}, June 2015.
\newblock arXiv no. 1506.03806.

\bibitem[Nev86]{Neveu}
J.~Neveu.
\newblock Arbres et processus de {G}alton-{W}atson.
\newblock {\em Ann. Inst. H. Poincar\'e Probab. Statist.}, 22(2):199--207,
  1986.

\bibitem[{Ric}18a]{RichierIICUIHPT}
L.~{Richier}.
\newblock The incipient infinite cluster of the uniform infinite half-planar
  triangulation.
\newblock {\em Electron. J. Probab.}, 23:Paper No. 89, 38, 2018.

\bibitem[{Ric}18b]{RichierMapBoundaryLimit}
L.~{Richier}.
\newblock Limits of the boundary of random planar maps.
\newblock {\em Probab. Theory Related Fields}, 172(3-4):789--827, 2018.

\bibitem[SS19]{StefStufBolzOuterplanar}
S.~{ Stef{\'a}nsson} and B.~{Stufler}.
\newblock Geometry of large {B}oltzmann outerplanar maps.
\newblock {\em Random Structures Algorithms}, 55(3):742--771, 2019.

\bibitem[Ste18a]{StephensonLocal}
R.~Stephenson.
\newblock Local convergence of large critical multi-type {G}alton-{W}atson
  trees and applications to random maps.
\newblock {\em J. Theoret. Probab.}, 31(1):159--205, 2018.

\bibitem[Ste18b]{StephLocalLim}
R.~Stephenson.
\newblock Local convergence of large critical multi-type {G}alton-{W}atson
  trees and applications to random maps.
\newblock {\em J. Theoret. Probab.}, 31(1):159--205, 2018.

\bibitem[Tet91]{Tetali}
P.~Tetali.
\newblock Random walks and the effective resistance of networks.
\newblock {\em J. Theoret. Probab.}, 4(1):101--109, 1991.

\bibitem[Wat10]{WataIto}
S.~Watanabe.
\newblock It\^o's theory of excursion point processes and its developments.
\newblock {\em Stochastic Process. Appl.}, 120(5):653--677, 2010.

\end{thebibliography}

\end{document}